\documentclass[11pt]{amsart}

\pdfoutput=1

\usepackage{etex,amsfonts, amsmath, amsthm, amssymb, stmaryrd, epsfig, graphics,
  psfrag, latexsym, mathtools, mathrsfs,enumitem, subfig,
  longtable, booktabs, yfonts, centernot,ifthen,eso-pic,longtable} 
\usepackage[clock,weather]{ifsym} %Just for StopWatchStart
\usepackage[normalem]{ulem}
\usepackage{xr-hyper}
\usepackage[bbgreekl]{mathbbol}
\DeclareSymbolFontAlphabet{\mathbb}{AMSb}
\DeclareSymbolFontAlphabet{\mathbbl}{bbold}
\usepackage{xcolor}
\usepackage[all,cmtip]{xy}
\usepackage{tikz} \usetikzlibrary{matrix,arrows,shapes,calc,braids,cd}
\definecolor{darkblue}{rgb}{0,0,0.4} 
 \usepackage[colorlinks=true, citecolor=darkblue, filecolor=darkblue, linkcolor=darkblue,urlcolor=darkblue]{hyperref} 
\usepackage[all]{hypcap}
\usepackage{xr-hyper}

\graphicspath{{draws/}{}}

\numberwithin{equation}{section}

\newtheorem{thm}{Theorem}
\newtheorem{theorem}[thm]{Theorem}

\newtheorem{lem}{Lemma}[section]               
\newtheorem{lemma}[lem]{Lemma}

\newtheorem{corollary}[lem]{Corollary}               

\newtheorem{proposition}[lem]{Proposition}
\newtheorem{citethm}[lem]{Theorem}
\theoremstyle{definition}

\newtheorem{definition}[lem]{Definition}

\theoremstyle{remark}

\newtheorem{remark}[lem]{Remark}

\newtheorem{example}[lem]{Example}

\newtheorem{convention}[lem]{Convention}

\numberwithin{figure}{section}
\numberwithin{table}{section}

\newcommand{\R}{\mathbb{R}}

\newcommand{\C}{\mathbb{C}}

\newcommand{\mc}{\mathcal}

\newcommand{\wt}{\widetilde}

\renewcommand{\emptyset}{\varnothing}
\newcommand{\interior}{\mathring}
\newcommand{\Wmirror}[1]{\widehat{#1}}

\newcommand{\into}{\hookrightarrow}

\newcommand{\onto}{\twoheadrightarrow}

\renewcommand{\th}{^{\text{th}}}

\DeclareMathOperator{\nbd}{nbd}

\DeclareMathOperator{\Ob}{Ob}
\DeclareMathOperator{\Id}{Id}

\DeclareMathOperator{\Hom}{Hom}

\newcommand{\dg}{\textit{dg} }

\newcommand{\Kh}{\mathit{Kh}}

\newcommand{\KhCx}{\mc{C}}

\newcommand{\Cat}{\mathscr{C}}

\newcommand{\Dat}{\mathscr{D}}
\newcommand{\Eat}{\mathscr{E}}

\newcommand{\op}{\mathrm{op}}

\newcommand{\gr}{\mathrm{gr}}
\newcommand{\intgr}{\gr_{q}}

\newcommand{\CubeCat}[1]{\underline{2}^{#1}}
\newcommand{\CubeCatP}[1]{\underline{2}_+^{#1}}

\newcommand{\co}{\colon}
\newcommand{\bdy}{\partial}
\newcommand{\RR}{\R}
\newcommand{\CC}{\C}

\newcommand{\ZZ}{\mathbb{Z}}

\newcommand{\RP}{\mathbb{R}\mathrm{P}}

\newcommand{\mathcenter}[1]{\vcenter{\hbox{$#1$}}}

\DeclareMathOperator{\Cone}{Cone}

\newcommand{\BurnsideCat}{\mathscr{B}}
\newcommand{\CCat}[1]{\CubeCat{#1}}%The cube category
\newcommand{\CCatP}[1]{\CubeCatP{#1}}%The cube category with an extra object
\newcommand{\AbelianGroups}{\mathsf{Ab}}
\newcommand{\Spectra}{\mathscr{S}}

\newcommand{\Complexes}{\mathsf{Kom}}

\newcommand{\SphereS}{\mathbb{S}} %The sphere spectrum.

\newcommand{\CobD}{\mathsf{Cob}_d}
\newcommand{\CobDenl}{\widetilde{\mathsf{Cob}}_d}

\DeclareMathOperator{\hocolim}{hocolim}

\newcommand{\Crossingless}[1]{{\mathsf{B}}({#1})}
\newcommand{\mHshape}[1]{\wt{\mathcal{S}}_{#1}}
\newcommand{\mHshapeS}[1]{\mathcal{S}_{#1}}

\newcommand{\mTshape}[2]{\wt{\mathcal{T}}_{#1;#2}}
\newcommand{\SmTshape}[2]{\mathcal{T}_{#1;#2}} %Strict

\newcommand{\KhSpace}{\mathscr{X}}

\newcommand{\mGlue}[3]{\wt{\mathcal{U}}_{#1;#2;#3}}
\newcommand{\mGlueS}[3]{\mathcal{U}_{#1;#2;#3}}

\newcommand{\mBurnside}{\BurnsideCat}
\newcommand{\ttimes}{\tilde{\times}}
\newcommand{\lab}[1]{$\scriptstyle #1$}

\newcommand{\mTinvNF}[1]{{\underline{\mathsf{MB}}}_{#1}}%from T to Burn. For non-flat tangles.
\DeclareMathOperator{\THH}{THH}
\DeclareMathOperator{\HH}{HH}

\DeclareMathOperator{\RHom}{RHom}

\newcommand{\annulus}{A}
\newcommand{\OneHalf}{{\textstyle\frac{1}{2}}}

\newcommand{\TangMulticat}{\mathscr{T}}
\newcommand{\TangMovMulticat}{\mathbb{T}}

\newcommand{\BimCat}{\mathsf{Bim}}
\newcommand{\SBimCat}{\mathsf{SBim}}

\newcommand{\enl}[1]{\widetilde{#1}}

\newcommand{\grs}[2]{\{#2,#1\}} %Grading shift by homological, quantum amount.

\newcommand{\sA}{\mathscr{A}}
\newcommand{\sB}{\mathscr{B}}

\newcommand{\Crossings}{\mathfrak{C}}

\newcommand{\tId}{\widetilde{\Id}}

\begin{document}

%%% Local Variables: 
%%% mode: latex
%%% TeX-master: "Cobordism2"
%%% End: 

\title{Homotopy functoriality for Khovanov spectra}

\author{Tyler Lawson}
\thanks{\texttt{TL was supported by NSF FRG Grant DMS-1560699.}}
\email{\href{mailto:tlawson@math.umn.edu}{tlawson@math.umn.edu}}
\address{Department of Mathematics, University of Minnesota, Minneapolis, MN 55455}

\author{Robert Lipshitz}
\thanks{\texttt{RL was supported by NSF Grant DMS-1810893.}}
\email{\href{mailto:lipshitz@uoregon.edu}{lipshitz@uoregon.edu}}
\address{Department of Mathematics, University of Oregon, Eugene, OR 97403}

\author{Sucharit Sarkar}
\thanks{\texttt{SS was supported by NSF Grant DMS-1905717.}}
\email{\href{mailto:sucharit@math.ucla.edu}{sucharit@math.ucla.edu}}
\address{Department of Mathematics, University of California, Los Angeles, CA 90095}

%\subjclass[2010]{\href{http://www.ams.org/mathscinet/search/mscdoc.html?code=57M25,55P42}{57M25,55P42}}


\date{}

\begin{abstract}
  We prove that the Khovanov spectra associated to links and tangles
  are functorial up to homotopy and sign.
\end{abstract}
\maketitle
\vspace{-1cm}
%\tableofcontents

%%% Local Variables: 
%%% mode: latex
%%% TeX-master: "Cobordism"
%%% End: 

\tableofcontents

% \listoftodos

\section{Introduction}
The goal of this paper is to prove that the Khovanov
spectrum~\cite{RS-khovanov,LLS-khovanov-product,HKK-Kh-htpy}, an
object in the homotopy category of spectra, is natural with respect to
link cobordisms, up to sign. That is:
\begin{theorem}\label{thm:main}
  If $L_0$ and $L_1$ are oriented link diagrams in $\RR^2$ and
  $\Sigma\co L_0\to L_1$ is an oriented cobordism then there is an
  induced homotopy class of maps of spectra  \label{not:KhSpace}
  \[
    \KhSpace(\Sigma)\co \KhSpace^j(L_0)\to
    \KhSpace^{j-\chi(\Sigma)}(L_1)
  \]
  from the Khovanov spectrum of $L_0$ to the Khovanov spectrum of
  $L_1$, well-defined up to sign. Given another oriented link
  cobordism $\Sigma'\co L_1\to L_2$,
  \[
    \KhSpace(\Sigma')\circ\KhSpace(\Sigma)=\pm \KhSpace(\Sigma'\circ\Sigma).
  \]
  Further, if $\Sigma$ consists of a single Reidemeister move then the
  map $\KhSpace(\Sigma)$ is homotopic to the map in the original proof
  of invariance of $\KhSpace(L)$~\cite{RS-khovanov}, and if $\Sigma$
  consists of a single birth, death, or saddle then $\KhSpace(\Sigma)$
  is homotopic to the map defined previously in those
  cases~\cite{LS-rasmus}.
\end{theorem}

For spectra, ``up to sign'' means the following. Roughly, reflection
across the first coordinate of $\RR^{n+1}$ induces an automorphism
$(-1)\co \SphereS\to\SphereS$ of the sphere
spectrum\phantomsection\label{not:sphere}; more precisely, to make
this automorphism commute with the structure maps of $\SphereS$ one
takes a cofibrant-fibrant replacement of the sphere spectrum
first. Then, for any spectrum $X$, there is an induced map
$X=\SphereS\wedge X\xrightarrow{(-1)\wedge\Id}\SphereS\wedge X=X$,
which plays the role of multiplication by $-1$.

Functoriality of Khovanov homology up to sign was first established by
Jacobsson~\cite{Jac-kh-cobordisms}, by checking directly that the maps
Khovanov had associated to elementary cobordisms~\cite[Section
6.3]{Kho-kh-categorification} were invariant under Carter-Saito's
movie moves~\cite{CS-knot-movie}. Shortly after, Khovanov and
Bar-Natan gave new proofs of this result, using extensions of Khovanov
homology to tangles to simplify checking most of the movie
moves~\cite{Kho-kh-cobordism,Bar-kh-tangle-cob}. Detailed
analyses of Jacobsson's proof led to better understanding of the sign
ambiguity, and ways to remove
it~\cite{CMW-kh-functoriality,Cap-kh-functoriality}. Recently,
Blanchet~\cite{Bla-kh-oriented} gave another approach to avoiding the
sign ambiguity of Khovanov homology, using Lee's
deformation~\cite{Lee-kh-endomorphism}; see also~\cite{Sano-Kh-func}. A spectral refinement of part
of Blanchet's work was given by Krushkal-Wedrich~\cite{KW-kh-Blanchet}.

The strategy to prove Theorem~\ref{thm:main} is generally similar to
Khovanov's proof of naturality. In a previous paper, we gave a
spectral refinement of Khovanov's tangle
invariants~\cite{LLS-kh-tangles}. (By contrast, a spectral refinement
of Bar-Natan's tangle invariant is not currently known, nor is a
spectral refinement of the Lee deformation.)  Much of Khovanov's
argument reduces to understanding the automorphisms of the bimodule
associated to the identity braid, and a few similar arguments. In the
spectral case, this bimodule has too many grading-preserving
automorphisms for Khovanov's argument to go through. (See
Section~\ref{sec:Khovanovs} for further discussion of this point.) We
avoid this problem by localizing further, analogous to Bar-Natan's
canopoly. In this more local form, the essence of Khovanov's argument goes
through.

This strategy gives somewhat more than Theorem~\ref{thm:main}. Like
Khovanov's and Bar-Natan's proofs, it gives an extension of
Theorem~\ref{thm:main} to tangle cobordisms
(Theorem~\ref{thm:Kh-spec-functorial}). Additionally, it shows that
Khovanov homology and the Khovanov spectrum are also functorial under
non-orientable cobordisms, though the grading shifts are harder to
track. Along the way, we also prove two structural results about the
Khovanov spectrum (as well as their analogues for Khovanov homology):
the Khovanov spectral bimodule associated to the mirror of $T$ is the
dual to the Khovanov spectral bimodule associated to $T$, and the
Khovanov spectral modules satisfy a planar algebra-like gluing
property. (For Khovanov's arc algebras, the analogous properties seem
to be well-known---see, for instance,~\cite[Section
5.3]{Roberts-kh-planar} for the latter---but we do not have a specific
citation for them.)

This paper is organized as follows. Section~\ref{sec:background}
recalls Khovanov's arc algebras and aspects of their spectral
refinements. Section~\ref{sec:Khovanovs} discusses why Khovanov's
proof of invariance does not immediately translate to the spectral
case. The failure motivates the rest of this
paper. Section~\ref{sec:planar} gives the planar algebra-like gluing
property of the Khovanov modules and their spectral refinements, using
the language of multicategories. Section~\ref{sec:duality} proves the
duality between the Khovanov bimodules of a tangle and its mirror, and
the spectral refinement of this
duality. Section~\ref{sec:functoriality-proof} combines these to prove
functoriality of the Khovanov spectra, Theorems~\ref{thm:main}
and~\ref{thm:Kh-spec-functorial}. We also give the analogous proof of
functoriality of Khovanov homology,
Theorem~\ref{thm:Kh-functorial}. Section~\ref{sec:neck} gives a
neck-cutting relation for the cobordism maps and uses it to deduce a
lift of Levine-Zemke's theorem about ribbon concordances and Khovanov homology.
Section~\ref{sec:computations} gives
an example of how to extract an explicit invariant of cobordisms from
the functor, in the spirit of the Hopf invariant.

\subsubsection*{Acknowledgments} We thank Jon Brundan, Slava
Krushkal, and Aaron Lauda for helpful conversations. We also thank the
referee for further helpful comments and corrections.

\section{Background and grading conventions}\label{sec:background}
\emph{Wherein we} summarize expeditiously
Khovanov's \textsc{arc algebras and bimodules} including their key
\textsc{gluing} and \textsc{invariance} properties. We then recall the \textsc{spectral
  refinements} of these algebraic objects, and corresponding
properties of these spectral refinements. We conclude with a
summary of the paper's \textsc{grading conventions}.

\subsection{Khovanov's arc algebras and modules}\label{sec:arc-alg}
\label{not:V}
Let $V=\ZZ[X]/(X^2)$ denote Khovanov's Frobenius algebra. The
comultiplication on $V$ is given by $\Delta(1)=1\otimes X+X\otimes 1$
and $\Delta(X)=X\otimes X$, and the counit is $\epsilon(1)=0$,
$\epsilon(X)=1$. Equivalently, we can view $V$ as a
$(1+1)$-dimensional TQFT. So, given a closed, oriented $1$-manifold $Z$, we have
an abelian group $V(Z)$, which is generated by all ways of labeling
the components of $Z$ by $1$ or $X$, and an oriented cobordism from $Z$ to $Z'$
induces a homomorphism from $V(Z)$ to $V(Z')$ (which is the multiplication
in $V$ if the cobordism is a single merge and the comultiplication
$\Delta$ if the cobordism is a single split.)

Let $\CCat{}$ be the category with two objects, $0$ and $1$, and a
single morphism from $0$ to $1$,
\[
  \CCat{}=\bigl(0\longrightarrow 1).
\]

\phantomsection\label{not:crossings}%
Given a link diagram $L$ with $N$ crossings $\Crossings$ one can consider the
cube of resolutions of $L$, a functor from $\CCat{\Crossings}$ to the
$(1+1)$-dimensional cobordism category. The edges of the cube
correspond to \emph{crossing change cobordisms}. A checkerboard coloring induces
orientations on all of the 1-manifolds and cobordisms in this cube.  Composing
with the TQFT $V$ gives a commutative cube
$\CCat{\Crossings}\to \AbelianGroups$, the category of abelian
groups. Traditionally, the
Khovanov complex is defined as the total complex or iterated mapping
cone of this cube. To avoid choosing a sign assignment or ordering of
the crossings, we will take another version of the iterated mapping
cone. Let $\CCatP{\Crossings}$ be the category obtained by adding one more
object $*$ to $\CCat{\Crossings}$ and a morphism from each object except the
terminal one in $\CCat{\Crossings}$ to $*$. Extend $V$ to a functor
$\CCatP{\Crossings}\to\AbelianGroups$ by sending $*$ to the trivial group.
Let $\KhCx(L)=\bigoplus_{i,j}\KhCx_{i,j}(L)$\phantomsection\label{not:KhCx} be the homotopy colimit of
this diagram (see, e.g.,~\cite[Section 4.2]{LLS-khovanov-product}
or~\cite[Definition 3.11]{HLS-flexible}), with an internal (quantum)
grading and a homological shift that use the orientation of $L$ or other auxiliary data (see
Section~\ref{sec:gradings}). That is, up to a shift,
\begin{equation}
  \KhCx(L)=\hocolim_{w\in \CCatP{\Crossings(L)}}V(L_w).
\end{equation}
This complex is homotopy equivalent to the usual Khovanov complex,
though the signs in the homotopy equivalence depend on some
choices. The Khovanov homology $\Kh(L)$ is the homology of $\KhCx(L)$.

Recall that an $(m,n)$-tangle consists of a compact 1-dimensional
manifold-with-boundary $T$ properly embedded in $[0,1]\times\RR$, so that
the boundary of $T$ is the $(m+n)$ points
$\{(0,1),\dots,(0,m)\}\cup\{(1,1),\dots,(1,n)\}$.
Khovanov extended his construction to tangles as
follows~\cite{Kho-kh-tangles}. Given an even integer $n$, let
$\Crossingless{n}$\phantomsection\label{not:crossingless} denote the set of crossingless matchings of $n$
points. View an element $a\in\Crossingless{n}$ as a
$(0,n)$-tangle, and let $\Wmirror{a}$ denote its mirror, an
$(n,0)$-tangle. Let $\KhCx(n)$ denote the linear category with:
\begin{itemize}
\item Objects $\Crossingless{n}$,
\item $\KhCx(n)(a,b)\coloneqq \Hom_{\KhCx(n)}(a,b)=V(a\Wmirror{b})$, and
\item Composition $\Hom_{\KhCx(n)}(b,c)\times\Hom_{\KhCx(n)}(a,b)\to\Hom_{\KhCx(n)}(a,c)$ induced by
  the TQFT $V$ and the \emph{canonical saddle cobordism}
  $\Wmirror{a}\amalg a\to \Id$, the identity braid on $n$ points.
\end{itemize}
Equivalently, we can view $\KhCx(n)$ as an algebra, by taking
\[
  \bigoplus_{a,b\in\Ob(\KhCx(n))}\KhCx(n)(a,b)
\]
with multiplication $(x\cdot y)=y\circ x$ when defined and $0$
otherwise. (Some elementary concepts
related to linear categories are recalled in Section~\ref{sec:lin-cat}.)

Given an $(m,n)$-tangle
diagram $T$ with $N$ crossings, there is a differential module
$\KhCx(T)$ over $\KhCx(m)$ and $\KhCx(n)$ defined by
\begin{equation}\label{eq:KhCx-T-hocolim}
  \KhCx(T)(a,b)=\KhCx(aT\Wmirror{b})=\hocolim_{v\in \CCatP{\Crossings(T)}}V(aT_vb),
\end{equation}
where $T_v$ is the resolution of $T$ associated to $v$, and for $v=*$ we define $V(aT_vb)=0$.
\label{not:KhCxT}
The module structure is induced by the canonical saddle cobordisms,
and the differential again comes from the crossing change
cobordisms. Far-commutativity of these cobordisms implies that the
module structure is associative and respects the differential.
(Note that any composition of these cobordisms is again orientable.)

Khovanov proves:
\begin{lemma}\label{lem:Kh-projective}\cite{Kho-kh-tangles}
  The module $\KhCx(T_v)$ associated to each resolution $T_v$ of $T$ is left-projective and right-projective. In
  fact, for each $a\in\Crossingless{m}$ there is a crossingless
  matching $a'\in\Crossingless{n}$ and an integer $j$ so that
  $\KhCx(T_v)(a,\cdot)\cong V^{\otimes j}\otimes
  \KhCx(n)(a',\cdot)$, and similarly in the other factor.
\end{lemma}

\begin{citethm}\label{thm:Kh-tangle-invt}\cite{Kho-kh-tangles}
  Up to quasi-isomorphism, the differential graded bimodule $\KhCx(T)$
  associated to an oriented tangle $T$ is invariant under Reidemeister
  moves.
\end{citethm}
In fact, Theorem~\ref{thm:Kh-tangle-invt} holds up to homotopy
equivalence of differential graded
bimodules, which could be used to simplify some of the discussion
below in the algebraic, but not the spectral, case; see
Remark~\ref{rem:simple-aa}. In the proof, Khovanov associates specific
homomorphisms to the Reidemeister moves. The orientation is needed to
pin down the grading shifts (see Section~\ref{sec:gradings}).

The other key property is that gluing of tangles corresponds to tensor
product of bimodules:
\begin{citethm}\label{thm:Kh-gluing}\cite{Kho-kh-tangles}
  Given an $(m,n)$-tangle $T_1$ and an $(n,p)$-tangle $T_2$, there is
  a quasi-isomorphism
  \[
    \KhCx(T_1)\otimes_{\KhCx(n)}\KhCx(T_2)\simeq \KhCx(T_1T_2).
  \]
\end{citethm}

\subsection{Terminology for linear and spectral categories}\label{sec:lin-cat}
Since we are working mainly in the language of linear or spectral
categories, we recall how some constructions and terminology for rings
extends to this setting. In the linear case, verifying that these
extensions have the expected properties is elementary; for the
spectral case, see for instance Blumberg-Mandell~\cite[Section 2]{BM-top-spectral}.

We call a linear category \emph{finite} if it has finitely many
objects and each morphism space is a finitely-generated free abelian
group. A spectral category is finite if it has finitely many objects
and each morphism space is weakly equivalent to a finite CW spectrum.

Let $\Cat$ and $\Dat$ be linear categories. The \emph{tensor product}
$\Cat\otimes\Dat$ has objects $\Ob(\Cat)\times\Ob(\Dat)$ and
$\Hom_{\Cat\otimes\Dat}((c_1,d_1),(c_2,d_2))=\Hom_{\Cat}(c_1,c_2)\otimes_\ZZ\Hom_{\Dat}(d_1,d_2)$.
A \emph{\dg $(\Cat,\Dat)$-bimodule} is a \dg functor
$\Cat^\op\otimes\Dat\to\Complexes$, where $\Complexes$ denotes the
category of chain complexes of free abelian groups (and the morphism
spaces in $\Cat$ and $\Dat$ have trivial differential). We will often
drop the term \dg even though we are considering \dg bimodules. If
$\Cat$ and $\Dat$ are spectral categories, their tensor product and
bimodules are defined similarly, with smash product in place of tensor
product and spectra in place of chain complexes.

Given a $(\Cat,\Dat)$-bimodule $M$ and a $(\Dat,\Eat)$-bimodule $N$, 
the tensor product of $M$ and $N$ is the $(\Cat,\Eat)$-bimodule
$M\otimes_{\Dat}N$ with
\[
  (M\otimes_{\Dat}N)(c,e)=\Bigl(\bigoplus_{d\in\Ob(\Dat)}M(c,d)\otimes_\ZZ
  N(d,e)\Bigr)/(f_*(m)\otimes n\sim m\otimes f^*(n))
\]
for $f\in \Dat(d,d')$, with the obvious structure maps. There is an
analogous tensor product in the spectral case.

Given $(\Cat,\Dat)$-bimodules $M$ and $N$, a \emph{chain map} from $M$ to
$N$ is a natural transformation. Explicitly, a chain map consists of
chain maps $F_{c,d}\co M(c,d)\to N(c,d)$ for each $c\in\Ob(\Cat)$ and
$d\in\Ob(\Dat)$ so that for any objects $(c_1,d_1), (c_2,d_2)\in
\Ob(\Cat^\op\times\Dat)$, $m\in M(c_1,d_1)$, $f\in \Cat(c_2,c_1)$, and
$g\in\Dat(d_1,d_2)$,
\begin{equation}
  F_{c_2,d_2}(M(f^\op,g)(m))=N(f^\op,g)(F_{c_1,d_1}(m));\label{eq:chain-map}
\end{equation}
if we write $M(f^\op,g)(m)$ in the perhaps more suggestive
notation $f\cdot m\cdot g$, and similarly for $N$, this equation becomes
\[
  F_{c_2,d_2}(f\cdot m\cdot g)=f\cdot F_{c_1,d_1}(m)\cdot g.
\]
Similarly, a chain homotopy from a chain map $F$ to a chain map $G$
consists of chain homotopies $H_{c,d}$ from $F_{c,d}$ to $G_{c,d}$ for
each $(c,d)\in\Ob(\Cat^\op\times\Dat)$ satisfying the same
compatibility condition~\eqref{eq:chain-map}.
One can also define the homology of a
$(\Cat,\Dat)$-bimodule, and hence a \emph{quasi-isomorphism} of
$(\Cat,\Dat)$-bimodules. The set of chain homotopy classes of chain
maps is not invariant under quasi-isomorphism, but is if one takes a
\emph{projective resolution} of $M$ first, so we define $\Hom(M,N)$ to be
the set of chain homotopy classes of chain maps from a projective
resolution of $M$ to $N$. Similarly, in the spectral
case, we define $\Hom(M,N)$ to be the path components of the space of
natural transformations from a cofibrant resolution of $M$ to a
fibrant resolution of $N$. This notion is invariant under weak
equivalence of $M$ and $N$. Tensor products are also not
invariant under quasi-isomorphism or weak equivalence, but if one
takes projective (cofibrant) replacements first then they become so.

Sometimes, we will be interested in the chain complex of maps between complexes
(after taking a projective resolution of the source), or the spectrum of maps
between spectra (after taking a cofibrant replacement); in this case, we will
use the notation $\RHom$. So, for chain complexes $\Hom(M,N)=H_0\RHom(M,N)$ and
for spectra $\Hom(M,N)=\pi_0\RHom(M,N)$.

\subsection{Spectral arc algebras and modules}\label{sec:spec-arc-alg}
The construction of the spectral Khovanov algebras and modules uses
Elmendorf-Mandell's $K$-theory of permutative categories~\cite{EM-top-machine},
and the first step is translating the notion of algebras and modules into that
language. For each even integer $n$, there is a \emph{arc algebra shape
  multicategory} $\mHshapeS{n}$ with an object for each pair of crossingless
matchings $(a_1,a_2)\in\Crossingless{n}\times\Crossingless{n}$---these remember
the $\Hom$-spaces in the arc algebra $\KhCx(n)$---and
morphisms encoding when $\Hom$'s can be composed~\cite[Section 2.3]{LLS-kh-tangles}. That is, there is a unique
multimorphism
\[
  (a_1,a_2),(a_2,a_3),\dots,(a_{\alpha-1},a_\alpha)\to (a_1,a_\alpha)
\]
for each $n$-tuple of crossingless matchings
$a_1,\dots,a_\alpha\in\Crossingless{n}$. Khovanov's arc algebra $\KhCx(n)$
is a multifunctor from $\mHshapeS{n}$ to abelian groups (and
multilinear maps).\phantomsection\label{not:KhSpaceAlg} To define the spectral arc algebra $\KhSpace(n)$,
it suffices to lift the arc algebra multifunctor to a functor
$\mHshapeS{n}\to\Spectra$, the multicategory of spectra. Similarly,
there is a \emph{tangle shape multicategory} $\SmTshape{m}{n}$ so that
a multifunctor from the tangle shape multicategory to chain complexes
or spectra encodes the notion of a pair of linear or spectral categories with
object sets $\Crossingless{m}$ and $\Crossingless{n}$, respectively, and a
differential bimodule or spectral bimodule over them. Khovanov's
bimodules $\KhCx(T)$ define a functor from $\SmTshape{m}{n}$ to chain
complexes, and to construct the spectral tangle invariants it suffices
to lift these bimodules to $\Spectra$. (See also the discussion
in~\cite[Section 3.3]{LLS-kh-CK}.) There are also groupoid-enriched
versions $\mHshape{n}$, $\mTshape{m}{n}$ of these shape
multicategories. It is easier to define functors from the
groupoid-enriched versions (because this encodes a kind of lax
multifunctor), and Elmendorf-Mandell's rectification theorem implies
that the space of functors from the groupoid-enriched versions is
equivalent to the space of functors from the honest
versions~\cite[Sections 2.4 and 2.9]{LLS-kh-tangles}.

The construction of the functors from $\mHshapeS{n}$ and $\SmTshape{m}{n}$ to
spectra proceeds in several steps. Elmendorf and Mandell's $K$-theory is a
multifunctor from the category of permutative categories to spectra. The
Burnside category $\BurnsideCat$ (of the trivial group) is the multicategory
enriched in groupoids with objects finite sets, morphisms
$\Hom(X_1,\dots,X_k;Y)$ the finite correspondences from
$X_1\times\cdots\times X_k$ to $Y$, and $2$-morphisms bijections of
correspondences. There is a functor from the Burnside category to permutative
categories sending a set $X$ to the category of sets over $X$. So, to construct
functors $\mHshapeS{n}\to\Spectra$ and $\SmTshape{m}{n}\to\Spectra$ it suffices to
give functors to the Burnside category.

The \emph{embedded cobordism category} has objects closed $1$-manifolds embedded
in $\RR^2$ or $(0,1)^2$, $1$-morphisms cobordisms embedded in
$[0,1]\times\RR^2$, and $2$-morphisms isotopies of embedded cobordisms. In a
previous paper~\cite[Section 2.11]{LLS-khovanov-product}, we constructed the
\emph{Khovanov-Burnside functor}, from the embedded cobordism category to the
Burnside category. (See also~\cite[Section 2.11]{LLS-kh-tangles}
and~\cite{HKK-Kh-htpy}.)

To avoid needing to check that no loops of cobordisms where the
Khovanov-Burnside functor has nontrivial monodromy appear in the construction of
the tangle invariants, we introduce another auxiliary category, the
\emph{divided cobordism category}~\cite[Section 3.1]{LLS-kh-tangles}. The
following is a trivial generalization of that definition:
\begin{definition}\label{def:CobD}
  Let $U$ be a subset of $\RR^2$. The \emph{divided cobordism category of $U$},
  denoted $\CobD(U)$, is the category enriched in groupoids defined as follows:
  \begin{enumerate}
\item An object of $\CobD(U)$ is an equivalence class of the following data:
  \begin{itemize}
  \item A smooth, closed $1$-manifold $Z$ embedded in the interior of $U$.
  \item A compact $1$-dimensional submanifold-with-boundary $A\subset Z$, the
    \emph{active arcs}, satisfying the following: If $I$ denotes the closure of
    $Z\setminus A$, then each component of $A$ and each component of $I$ is an interval. The
    components of $I$ are the \emph{inactive arcs}.
  \end{itemize}
\item A morphism from $(Z,A)$ to $(Z',A')$ is an equivalence class of
  pairs $(\Sigma,\Gamma)$ where
  \begin{itemize}
  \item $\Sigma$ is a smoothly embedded cobordism in $[0,1]\times
    \interior{U}$ from $\{0\}\times Z$ to $\{1\}\times Z'$, vertical near $\{0,1\}\times\interior{U}$.
  \item $\Gamma\subset \Sigma$ is a collection of properly embedded
    arcs in $\Sigma$, vertical near $\bdy\Sigma$, with $(\bdy A\cup \bdy A')=
    \bdy\Gamma$, and so that every component of $\Sigma\setminus
    \Gamma$ has one of the following forms:
    \begin{enumerate}[label=(\Roman*)]
    \item A rectangle, with two sides components of $\Gamma$
      and two sides components of $A\cup A'$. 
    \item A $(2k+2)$-gon, $k\geq 0$, with $(k+1)$ sides in $\Gamma$,
      one side in $I'$, and the other $k$ sides in $I$.
    \end{enumerate}
  \end{itemize}
  The pairs $(\Sigma,\Gamma)$ and $(\Sigma',\Gamma')$ are equivalent if there is
  an orientation-preserving diffeomorphism $\phi\co[0,1]\to[0,1]$ so that
  $(\phi\times\Id_{U})(\Sigma)=\Sigma'$ and
  $(\phi\times\Id_{U})(\Gamma)=\Gamma'$.
\item There is a unique $2$-morphism from $(\Sigma,\Gamma)$ to
  $(\Sigma',\Gamma')$ whenever $(\Sigma,\Gamma)$ is isotopic to
  $(\Sigma',\Gamma')$ rel boundary.
\item Composition of divided cobordisms is defined in the obvious way.
\end{enumerate}
\end{definition}

In the case that $U$ is a square $(0,1)^2$, the diffeomorphism group of the
first $(0,1)$-factor acts on the divided cobordism category, and we quotient by
this action. More precisely, we quotient the object set by the action of the
orientation-preserving diffeomorphisms of $(0,1)$ which are the identity near
$\{0,1\}$ and the morphism sets by the group of orientation-preserving
diffeomorphisms of $[0,1]\times(0,1)$ which are the identity near
$[0,1]\times\{0,1\}$ and which are independent of the first coordinate near
$\{0,1\}\times(0,1)$. (This last condition ensures that the diffeomorphisms
preserve the property of the cobordisms being vertical near the boundary.) Then
concatenation in the first $(0,1)$-factor gives a strictly associative
multiplication or \emph{horizontal composition} $\amalg$ on
$\CobD((0,1)^2)$. This horizontal composition allows us to view $\CobD((0,1)^2)$
as a multicategory, with multimorphisms from
$(Z_1,A_1),\dots,(Z_n,A_n)$ to $(Z,A)$ given by the morphisms in
$\CobD((0,1)^2)$ from $(Z_1,A_1)\amalg\cdots\amalg(Z_n,A_n)$ to
$(Z,A)$. (In the language of Hu-Kriz-Kriz, this is an example of a
$\star$-category~\cite{HKK-Kh-htpy}.)

Another case of composition is if $U=D^2\setminus(D_1\cup\dots\cup
D_k)$ and $V=D^2\setminus(D'_1\cup\dots\cup D'_\ell)$ are complements
of disjoint round disks inside $D^2$. Then, we can form the composition
$U\circ_iV$ by rescaling and translating $V$ to identify the outer
$D^2$ of $V$ with $D_i$ in $U$; see Definition~\ref{def:CobD-annulus}
below for more details. We will sometimes also call this composition
\emph{horizontal}, to distinguish it from composition of
cobordisms.

\phantomsection\label{not:enlCobD}
The category $\CobD$ has a \emph{canonical groupoid enrichment}
$\CobDenl$~\cite[Section 2.4]{LLS-kh-tangles}.

\begin{lemma}
  For any $U$, the Khovanov-Burnside functor induces a functor
  $\CobDenl(U)\to\BurnsideCat$. In the case $U=(0,1)^2$, this
  functor respects the action of the diffeomorphism group of
  $(0,1)$ on the first factor, and for $U$, $V$ equal to the
  complements of disks in $D^2$ there is a natural isomorphism between
  the horizontal composition of the Khovanov-Burnside functors for $U$
  and $V$ and the Khovanov-Burnside functor for $U\circ_i V$.
\end{lemma}
\begin{proof}
  The first statement is a trivial generalization of~\cite[Proposition
  3.2]{LLS-kh-tangles}. The statement about invariance under the
  diffeomorphism action is immediate from the construction of the
  Khovanov-Burnside functor (see~\cite[Section 2.11]{LLS-kh-tangles}),
  as are the statements about gluing.
\end{proof}

Given crossingless matchings $a_1,a_2\in\Crossingless{m}$, there is an
associated object of $\CobD((0,1)^2)$ with underlying $1$-manifold
$a_1\Wmirror{a_2}$ and inactive arcs a small neighborhood of
$\bdy a_1=\bdy \Wmirror{a_2}$. The canonical saddle cobordisms have natural
choices of divides, giving divided cobordisms
$a_1\Wmirror{a_2}\amalg a_2\Wmirror{a_3}\to a_1\Wmirror{a_3}$~\cite[Section
3.2]{LLS-kh-tangles}, where $\amalg$ means concatenation and then rescaling in
the first $(0,1)$-direction. This gives a functor
$\mHshape{m}\to \CobDenl((0,1)^2)$. Composing with the Khovanov-Burnside functor then gives the spectral arc algebra.

\phantomsection\label{not:ttimes1}
For the spectral tangle invariants, given an $(m,n)$-tangle $T$ with crossings
$\Crossings$ (and $m$, $n$ even), there is a multicategory $\CCat{\Crossings}\ttimes \mTshape{m}{n}$
enriched in groupoids, a kind of thickened product of the cube category and the
tangle shape multicategory~\cite[Section 3.2.4]{LLS-kh-tangles}. There is a
multifunctor $\CCat{\Crossings}\ttimes \mTshape{m}{n}\to \CobDenl((0,1)^2)$ which
sends an object $(v,a,T,b)$ to the $1$-manifold $aT_v\Wmirror{b}$ with active
arcs at the boundary of $T_v$, a region around each crossing of $T$ which was
given the $0$-resolution, and a small neighborhood of at least one point in the
interior of each segment of $T$. (The last active arcs come from giving $T$
\emph{pox}; to reduce clutter, we will suppress the pox in this paper.)
Composing with the Khovanov-Burnside functor and $K$-theory gives a functor
$\CCat{\Crossings}\ttimes \mTshape{m}{n}\to \Spectra$. Applying
Elmendorf-Mandell's rectification procedure gives a functor
$\CCat{\Crossings}\times \SmTshape{m}{n}\to \Spectra$ from an ordinary
(non-enriched) multicategory. On the full subcategories $\mHshapeS{m}$ and
$\mHshapeS{n}$, this functor agrees with the arc algebra multifunctor. (This uses
the fact that those categories are \emph{blockaded}~\cite[Proposition
2.39]{LLS-kh-tangles}.) For each pair of crossingless
matchings $a,b$ we can restrict the functor to the subcategory spanned by
objects $(v,a,T,b)$, i.e., to the different resolutions of $T$ capped-off by $a$
and $b$, to get a map $\CCat{\Crossings}\to \Spectra$. Take the iterated mapping
cone of this functor by extending it to $\CCatP{\Crossings}$ by
sending $*$ to a one-point space and then
taking the homotopy colimit.\phantomsection\label{not:KhSpaceT1} Doing this for all pairs $(a,b)$ gives a functor
$\SmTshape{m}{n}\to\Spectra$, which corresponds to the spectral Khovanov tangle bimodule $\KhSpace(T)$.

A key property is that applying singular chains to these spectral invariants
gives the ordinary Khovanov algebras and chain complexes of bimodules up to
chain homotopy equivalence~\cite[Proposition 4.2]{LLS-kh-tangles}. (In fact,
the chain homotopy equivalences are canonical up to homotopy.) So, by
Whitehead's theorem, to verify invariance of the bimodules, it suffices to
construct maps associated to Reidemeister moves which induce Khovanov's homotopy
equivalences at the level of singular chains. Doing so is
straightforward~\cite[Sections 3.5 and 4.2]{LLS-kh-tangles}.

\phantomsection\label{not:mGlueS}
The final basic property of the tangle invariants is that gluing tangles
corresponds to tensor product of bimodule spectra. To prove this, we use yet
another multicategory: the \emph{gluing shape multicategory}
$\mGlueS{m}{n}{p}$, which encodes the notion of three bimodules $X$, $Y$, and $Z$ and a map from the
derived tensor product of $X$ and $Y$ to $Z$~\cite[Section
5]{LLS-kh-tangles}, and its groupoid enrichment $\mGlue{m}{n}{p}$. (See also
Section~\ref{sec:planar-spectral} for a generalization of this construction.)
Given an $(m,n)$-tangle $S$ and an $(n,p)$-tangle $T$, the same scheme
gives a multifunctor $\CCat{\Crossings}\ttimes \mGlue{m}{n}{p}\to\CobD((0,1)^2)$. Composing with the
Khovanov-Burnside functor and $K$-theory, then rectifying, gives a functor
$\mGlue{m}{n}{p}\to\Spectra$, which encodes a map from
$\KhSpace(S)\otimes^L_{\KhSpace(n)}\KhSpace(T)\to \KhSpace(T\circ S)$. At the
level of singular chains, this agrees with Khovanov's gluing map, hence is a
weak equivalence~\cite[Theorem 5]{LLS-kh-tangles}.

\subsection{Gradings}\label{sec:gradings}
To avoid keeping track of orientations of tangles, we will assign
Khovanov complexes to pairs $(T,P)$ where $T$ is a tangle and $P$ is
an integer. (This is similar to Khovanov's category
$\mathbb{ETL}$~\cite{Kho-kh-tangles}.) Given an oriented tangle, we
recover the usual Khovanov invariants by letting $P$ be the number of
positive crossings of $T$. Other than this, we follow the grading
conventions from our previous paper~\cite[Section
2.10.1]{LLS-kh-tangles}.

Grade the Khovanov Frobenius algebra $V$ by $\intgr(1)=-1$ and
$\intgr(X)=1$.

On the arc algebras: 
\begin{itemize}
\item For the quantum grading on $\KhCx(n)$, we shift
  $\KhCx(a\Wmirror{b})$ up by $n/2$, so the lowest-graded elements are
  idempotents in $\KhCx(a\Wmirror{a})$ in grading $0$.
\item For the homological grading, $\KhCx(n)$ lies in grading $0$.
\end{itemize}

Next, fix an $(m,n)$-tangle $T$ with $N$ crossings and an integer
$P$.
Recall that $\KhCx(T)(a,b)$ is the iterated mapping cone
(via a homotopy colimit) of a diagram
$\CCat{\Crossings(T)}\to\AbelianGroups$ (see
Equation~\eqref{eq:KhCx-T-hocolim}).
\begin{itemize}
\item \phantomsection\label{not:abs-v} For the quantum grading,  we shift the grading on
  $\KhCx(aT_v\Wmirror{b})$ up by $n/2-|v|+2N-3P$. (Here, $|v|$ is the height of $v$, i.e., the sum of the entries of $v$.)
\item For the homological grading, we let
  $\KhCx(aT_v\Wmirror{b})$ lie in homological grading $-P$. (This is
  before taking the mapping cone. After taking the mapping
  cone, the grading of the term corresponding to
  $\KhCx(aT_v\Wmirror{b})$ will be shifted up by $N-|v|$, so it will
  lie in homological grading $N-|v|-P$.)
\end{itemize}

\phantomsection\label{not:grs}
In formulas, if we let $\grs{q}{h}$ denote shifting the quantum
grading up by $q$ and the homological grading up by $h$, then 
\begin{align*}
  \KhCx(n)&=\bigoplus_{a\in\Crossingless{n}}V(a\Wmirror{a})\grs{n/2}{0}\\
  \KhCx(T,P)(v,a,b)&=V(aT_v\Wmirror{b})\grs{n/2-|v|+2N-3P}{-P}\\
  \KhCx(T,P)(a,b)&=\hocolim_{v \in\CCatP{\Crossings(T)}}\KhCx(T,P)(v,a,b).
\end{align*}

The homotopy equivalence for gluing tangles
(Theorem~\ref{thm:Kh-gluing}) consists of grading-preserving maps
\begin{align*}
  \KhCx(T_1,P_1)\otimes_{\KhCx(n)}\KhCx(T_2,P_2)&\stackrel{\simeq}{\longrightarrow}\KhCx(T_2\circ T_1,P_1+P_2).
\end{align*}

Given graded modules $M,N$, we define a homogeneous morphism
$f\co M\to N$ to have grading $k$ if $f$ increases the grading by $k$.
(This is the opposite of the typical grading convention for
cohomology, and would result in the cohomology of a topological space
being supported in negative gradings.)

\begin{remark}\label{rem:grading-help}
  With our grading conventions, the graded Euler characteristic of the
  Khovanov homology of $L$ is the (unnormalized) Jones polynomial of $m(L)$, the
  mirror of $L$, and positive knots have Khovanov homology supported
  in negative gradings. The differential on the Khovanov complex
  decreases the homological grading.
\end{remark}

\section{Khovanov's argument and why it does not translate
  immediately}\label{sec:Khovanovs}
\emph{Wherein we} recall key points of \textsc{Khovanov's proof of functoriality}
of Khovanov homology, observe \textsc{subtleties obstructing} one of these
key arguments in the spectral case, and note an \textsc{idea to
partly circumvent} this obstruction by \textsc{further localizing} the
problem which sets the scene for the rest of the paper.

The rest of the paper is independent of the discussion in this section.

Like all known proofs of functoriality of Khovanov homology, Khovanov's starts from a movie description of a cobordism. Each elementary movie is a cobordism between layered tangles. Several of the elementary movies are planar isotopies of tangles; the others are Reidemeister moves, births or deaths of zero-crossing unknots, and local saddles. Khovanov associates a map of bimodules to each of these elementary movies: for planar isotopies there are obvious isomorphisms, for Reidemeister moves he associated isomorphisms when he proved invariance of the tangle bimodules, and the maps for births, deaths, and saddles come from the unit, counit, and multiplication and comultiplication maps in his Frobenius algebra. The map associated to a movie is obtained by tensoring the maps for elementary movies, on the local slices of the layered tangle, with the identity map on the rest of the tangle, and then composing these maps. 

The next step is to prove that two movies representing isotopic cobordisms induce the same map on Khovanov homology, by checking that the maps are invariant under Carter-Saito's movie moves. Rather than laboriously checking each move (as Jacobsson did~\cite{Jac-kh-cobordisms}), the local description of the movie moves allows Khovanov to reduce this check to three principles and minor variants on them:
\begin{enumerate}[label=(\arabic*)]
\item Movies involving no crossings correspond to cobordisms in
  $\RR^3$, and he verified earlier that the maps associated to
  cobordisms in $\RR^3$ are isotopy invariants of those
  cobordisms~\cite{Kho-kh-cobordism}. (In fact, they depend only on the
  combinatorics of the cobordism, and not even its embedding.) This
  principle is used for movie moves 8, 9, 10, 23(b), and 24 in his list.
\item\label{item:Kh-hoch-coh} If $\Sigma$ is a movie between invertible tangles inducing a
  quasi-isomorphism on the Khovanov complex of bimodules (e.g.,
  because each piece is a Reidemeister move) then $\Sigma$ corresponds
  to a unit in
  $\HH^{0,0}(\KhCx(n))=\RHom^{0,0}_{\KhCx(n)\otimes
    \KhCx(n)^\op}(\KhCx(n),\KhCx(n))$, the
  Hochschild cohomology of $\KhCx(n)$ in bigrading $(0,0)$. This
  Hochschild cohomology group is identified with the part of the
  center of $\KhCx(n)$ in quantum grading $0$, which in turn is
  isomorphic to $\ZZ$. So, the only units are $\pm 1$. 
  This principle is used for movie moves 6, 12, 13, 23a, and 25, and variants on it are used for moves 7, 11, 14--22, and 26--30.
\item The map associated to the inverse of a Reidemeister move is the inverse of the map associated to a Reidemeister move (up to sign). (In fact, as Khovanov notes, this also follows from the previous principle and its variants.) This principle is used for movie moves 1--5.
\item The tensor product is a bifunctor, i.e., $f\otimes g = (f\otimes\Id)\circ (\Id\otimes g)=(\Id\otimes g)\circ (f\otimes\Id)$. This principle is used for movie move 31.
\end{enumerate}

To extend this argument to the Khovanov homotopy type, there are two
difficulties. The first is that we have not verified that the maps
associated to cobordisms in $\RR^3$ are isotopy invariants: in
constructing the homotopy refinements of Khovanov's tangle invariants,
we allow only certain isotopies of surfaces. (This restriction is
because of how the Khovanov-Burnside functor is defined on genus-1
surfaces with boundary~\cite[Section 2.11]{LLS-kh-tangles}.) For movies 8, 9, and 10, it is clear that the maps of homotopy types are the same, and it would not be hard to verify directly that the maps are homotopic for moves 23(b) and 24.

The second, more serious difficulty is with
Point~\ref{item:Kh-hoch-coh}. The difficulty can already been seen in
the case of the identity braid with two strands, i.e., for $\KhCx(2)\cong
\ZZ[X]/(X^2)$ and its spectral refinement $\KhSpace(2)\simeq
\SphereS\vee\SphereS$. Using the biprojective resolution
\[
  0\leftarrow \ZZ[X,Y]/(X^2,Y^2)\stackrel{X-Y}{\longleftarrow}\ZZ[X,Y]/(X^2,Y^2)\stackrel{X+Y}{\longleftarrow}\ZZ[X,Y]/(X^2,Y^2)\stackrel{X-Y}{\longleftarrow}\cdots,
\]
the Hochschild cohomology of $\KhCx(2)$ is the homology of the complex
\[
  0\to \KhCx(2)\stackrel{0}{\longrightarrow}\KhCx(2)\stackrel{2X}{\longrightarrow}\KhCx(2)\stackrel{0}{\longrightarrow}\cdots.
\]

The set of homotopy classes of bimodule homomorphisms from
$\KhSpace(2)$ to itself is $\pi_0\THH^*(\KhSpace(2))$. Since
$\KhCx(2)$ is flat over $\ZZ$ and $\KhSpace(2)$ is connective, there is a spectral sequence converging
to $\pi_*\THH^*(\KhSpace(2))$ with $E^1$-page given by
\[
  0\to [\KhCx(2)\otimes \pi_*(S^0)]_0\stackrel{0}{\longrightarrow}[\KhCx(2)\otimes \pi_*(S^0)]_1\stackrel{2X}{\longrightarrow}[\KhCx(2)\otimes \pi_*(S^0)]_2\stackrel{0}{\longrightarrow}\cdots
\]
where the subscripts $0$, $1$, $2$ are just labels for the different terms.
(This is the spectral sequence associated to the smash product of $\KhSpace(2)$
with the Postnikov tower of $\SphereS$.)
\phantomsection\label{not:grh}
The gradings are as follows. The homological grading of
$a\otimes \zeta\in [\KhCx(2)\otimes\pi_j(S^0)]_i$ is $\gr_h(a)+j-i$,
so the differential decreases the homological grading by $1$. The
quantum grading of $a\otimes \zeta\in [\KhCx(2)\otimes\pi_j(S^0)]_i$
is $\gr_q(a)-2i$, so the differential preserves the quantum grading.

Let $\eta\in \pi_1(S^0)\cong\ZZ/2\ZZ$ be the Hopf map. Then, the
(homological, quantum) bigrading $(0,0)$ part of the $E^2$-page is
\[
  \ZZ\langle [1\otimes1]_0\rangle \oplus(\ZZ/2\ZZ)\langle[X\otimes\eta]_1\rangle.
\]
Since the only elements in quantum grading $0$ have the form are $[1\otimes
\zeta]_0$ and $[X\otimes \zeta]_1$, for quantum grading $0$ the spectral sequence collapses
at the $E^2$-page. Hence, $\pi_0\THH^0(\KhSpace(2))$ fits into a short exact
sequence
\[
  0\to (\ZZ/2\ZZ)\langle[X\otimes\eta]_1\rangle\to
  \pi_0\THH^0(\KhSpace(2))\to \ZZ\langle [1\otimes 1]_0\rangle\to 0,
\]
so
\[
  \pi_0\THH^0(\KhSpace(2))\cong  \ZZ\langle [1\otimes1]_0\rangle \oplus(\ZZ/2\ZZ)\langle[X\otimes\eta]_1\rangle.
\]

Abusing notation, we denote the generator of this $\ZZ/2\ZZ$ by
$X\eta$. Then $\Id+X\eta$ is a nontrivial, grading-preserving (derived)
automorphism of the bimodule $\KhSpace(2)$, interfering with
technique~\ref{item:Kh-hoch-coh}.

For the case $n>2$, presumably $\KhSpace(n)$ has other, more
complicated automorphisms, as well.

On a more optimistic note, in the case of $\KhSpace(2)$, the only
obstruction to technique~\ref{item:Kh-hoch-coh} was a $2$-torsion
class. So, if we invert $2$ then Khovanov's argument would apply, to
show that the homotopy classes of bimodule automorphisms are the units
in $\ZZ[1/2]$. More generally, if we were only interested in
$\KhSpace(n)$ for finitely many $n$, there would be a finite list of
primes, corresponding to the torsion in $\pi_i(S^0)$ for $i$ small, so
that after inverting them Khovanov's argument applies. So, it is
natural to adapt Khovanov's argument to be more local, so that only
the $\KhSpace(n)$ for $n\leq 8$, say, appear. In fact, we will see that, perhaps
surprisingly, this adaptation leads to a proof of naturality without
inverting any primes.

\section{Planar composition for Khovanov's tangle invariants and their spectral refinements}\label{sec:planar}
\emph{Wherein we} formulate certain \textsc{multicategories of tangles
  and tangle cobordisms}, and use this language to give a minor
extension of Khovanov's \textsc{gluing results} for bimodules over the
arc algebra~\cite{Kho-kh-tangles}, in the spirit of
Bar-Natan's \textsc{canopoly}~\cite[Section 8]{Bar-kh-tangle-cob} or
of Jones's
\textsc{planar algebras}~\cite{Jones-oth-planar}. This material seems
to be \textsc{well-known}
  to experts (see, e.g.,~\cite[Section 5.3]{Roberts-kh-planar}).  We
follow this with \textsc{analogous} extensions for the
\textsc{spectral} tangle invariants.

\subsection{Multicategories of tangles}\label{sec:tang-multicat}

Let $S^1=\{z\in\CC\mid |z|=1\}$ and
$D^2=\{z\in\CC\mid |z|\leq 1\}$. A \emph{round disk} in $D^2$
is a subset of the form
$\{z\in D^2\mid |z-z_0|\leq r\}$ for some
$z_0\in \interior{D}^2$ and some $0<r<1-|z_0|$. If $D$ is a round disk then translation and
scaling gives a canonical identification $\phi_D\co D\to D^2$. Let
$\annulus=\{z\in D^2\mid 1/2\leq |z|\leq 1\}$ denote the annulus
with inner radius $1/2$ and outer radius $1$.

\begin{definition}\label{def:diskular}
  Fix non-negative, even integers $n,m_1,\dots,m_k$.  A
  \emph{diskular $(m_1,\dots,m_k;n)$-tangle} is a tangle diagram
  $T=T^{m_1,\dots,m_k;n}$ in
  $D^2\setminus (\interior{D}_1\cup\cdots\cup \interior{D}_k)$, where
  $D_1,\dots,D_k$ are disjoint round disks in $D^2$, so that the
  boundary of $T$ consists of
  \begin{itemize}
  \item the points $e^{2\pi i j/(n+1)}$, $j=1,\dots,n$, in $\bdy D^2$, and
  \item the points $\phi_{D_i}^{-1}(e^{2\pi i j/(m_i+1)})$, $j=1,\dots,m_i$, in $\bdy D^2_i$,
  \end{itemize}
  and $T$ is radial near $\bdy D^2$ and each $\bdy D_i$. See Figure~\ref{fig:diskular}.
  The disks $D_1,\dots,D_k$ are viewed as ordered. 

  Given diskular tangles $T^{m_1,\dots,m_k;n}$ and
  $S^{\ell_1,\dots,\ell_{j_i};m_i}_i$, $i=1,\dots,k$, let
  \[
    T\circ(S_1,\dots,S_k)=T\cup \phi_{D_1}(S_1)\cup\cdots\cup\phi_{D_k}(S_k).
  \]
  Alternatively, given an integer $1\leq i\leq k$ and pair of tangles
  $T^{m_1,\dots,m_k;n}$ and
  $S^{\ell_1,\dots,\ell_{j};m_i}$, there is a pairwise
  composition\phantomsection\label{not:TangComp}
  \[
    T\circ_i S=T\cup\phi_{D_i}(S).
  \]
  Again, see Figure~\ref{fig:diskular}.
  These are related by
  \[
    T\circ(S_1,\dots,S_k)=(\cdots((T\circ_k
    S_k)\circ_{k-1}S_{k-1})\circ_{k-2}\cdots)\circ_1 S_1.
  \]
\end{definition}

\begin{figure}
  \centering
  %Font is 12 point.
  \includegraphics[scale=.833333]{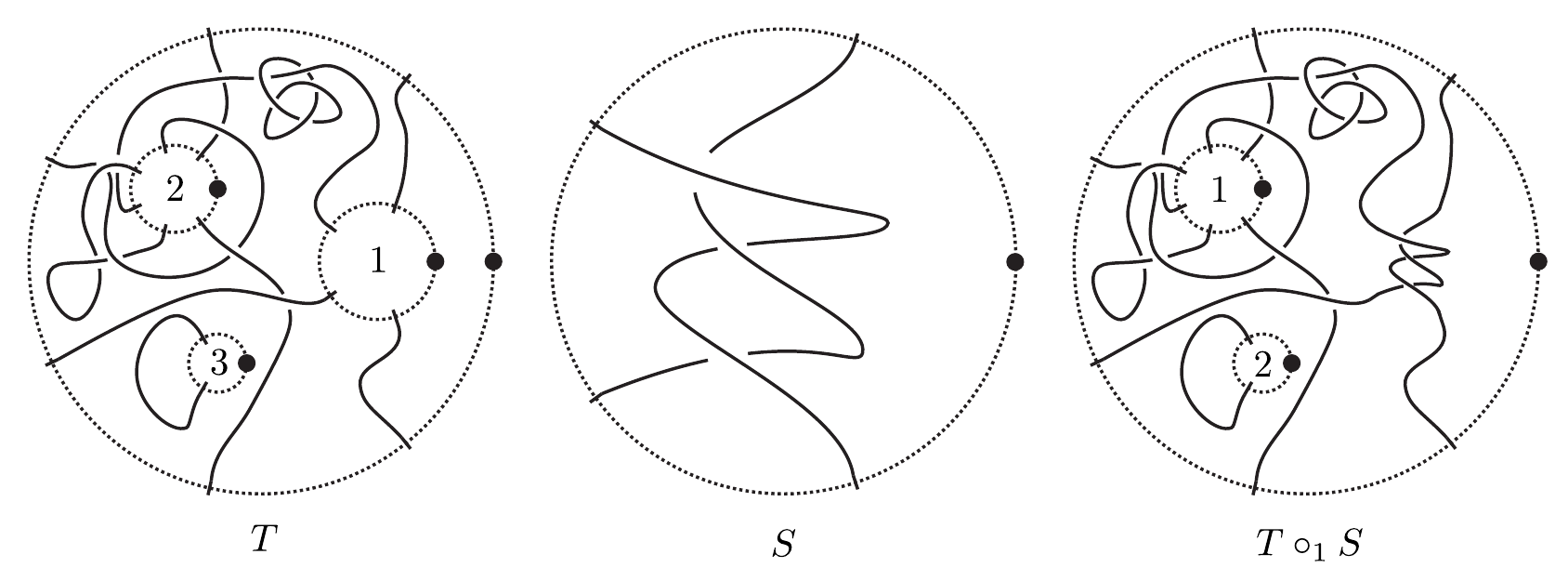}
  \caption{\textbf{Two diskular tangles and a composition of them.}
    Left: a diskular $(4,6,2;6)$ tangle $T=T^{4,6,2;6}$. Center: a diskular
    $(;4)$-tangle $S=S^{;4}$. Right: the composition $T\circ_1 S$.}
  \label{fig:diskular}
\end{figure}

We will call a diskular $(;n)$-tangle (i.e., a diskular tangle
involving no sub-disks) simply a \emph{diskular $n$-tangle.}

In Definitions~\ref{def:el-cob} and~\ref{def:cobordism}, we define an
essentially combinatorial version of cobordisms, in the spirit of
movies. We give a topological interpretation of these in (and
immediately preceding) Theorem~\ref{thm:CarterSaito}.

\begin{definition}\label{def:el-cob}
  Fix diskular $(m_1,\dots,m_k;n)$-tangles $S$ and $T$. An
  \emph{elementary cobordism} from $S$ to $T$ is any of the following:
  \begin{enumerate}
  \item\label{item:planar-isotopy} A planar ambient isotopy $\Phi_t\co D^2\to D^2$ from $S$ to $T$, so that
    $\Phi_t|_{\nbd(S^1)}$ is the identity for all $t$ and
    $\Phi_t|_{\nbd(D_i)
    }$ is the
    composition of translation and scaling for all $i,t$. The
    \emph{support} of the ambient isotopy is the union over $t$ of the
    support of $\Phi_t$ (the set of points where
    $\Phi_t\neq\Id$).
  \item\label{item:Reidemeister-move} A single Reidemeister move. For
    each type of Reidemeister move, we fix once and for all a pair of
    tangles $R,R'$ in $D^2$ corresponding to (the before and after
    pictures of) that Reidemeister move.  Then, a Reidemeister
    elementary cobordism is a pair of tangle diagrams obtained by
    replacing the image of $\phi_D^{-1}(R)\subset S$, for some round
    disk $D$, with $\phi_D^{-1}(R')$. (In particular, this is only
    permitted if $S$ contains such a $\phi_D^{-1}(R)$.)  The disk $D$
    is the \emph{support} of the Reidemeister move.
  \item\label{item:birth-death} A birth or death of a 0-crossing unknot disjoint from
    $S$. Again, this is the image of a fixed standard birth in $D^2$
    under the map $\phi_D^{-1}$ for some round disk $D$. The disk $D$
    is the \emph{support} of the birth or death.
  \item\label{item:saddle} A planar saddle. Again, this is the image of a fixed standard saddle in $D^2$
    under the map $\phi_D^{-1}$ for some round disk $D$. The disk $D$
    is the \emph{support} of the saddle.
  \end{enumerate}
\end{definition}

We will call births, deaths, and saddles \emph{Morse moves}. As we
will discuss below, each of these elementary cobordisms corresponds to
a particular embedded cobordism in the usual sense, but we treat the
elementary
cobordisms~(\ref{item:Reidemeister-move}),~(\ref{item:birth-death}),
and~(\ref{item:saddle}) as formal objects---just a pair of tangle
diagrams.

\begin{definition}\label{def:cobordism}
  Given an elementary cobordism $\Sigma$ from $S$ to $T$, let
  \begin{align*}
    P(\Sigma)&=
    \begin{cases}
      1 & \text{if $\Sigma$ is an R1 move creating a positive
        crossing or an R2 move creating crossings}\\
      -1 & \text{if $\Sigma$ is an R1 move removing a positive
        crossing or an R2 move removing crossings}\\
      0 & \text{otherwise}.
    \end{cases}\\
    \chi'(\Sigma)&=
                   \begin{cases}
                     -1 & \text{if $\Sigma$ is a saddle}\\
                     1 & \text{if $\Sigma$ is a birth or death}\\
                     0 & \text{otherwise}.
                   \end{cases}
  \end{align*}

  Given a sequence $\Sigma=(\Sigma_1,\dots,\Sigma_k)$ of
  elementary cobordisms starting at some tangle $S$ and ending at a
  tangle $T$, and an integer $P_S$, define
  \begin{align*}
    P(\Sigma)&=\sum_i P(\Sigma_i)\\
    \chi'(\Sigma)&=\sum_i \chi'(\Sigma_i).
  \end{align*}
  We say that $\Sigma$ \emph{goes from} $(S,P_S)$ to
  $(T,P_S+P(\Sigma))$.
\end{definition}

An alternative definition of $\chi'$ for a cobordism $\Sigma\co S\to
T$ (viewed as a surface) is the Euler characteristic of $\Sigma$ minus half the number of
endpoints of $S$. (This works for both elementary cobordisms and
compositions of elementary cobordisms.)

\begin{definition}\label{def:tangle-movie-multicat}
  The \emph{tangle movie multicategory} $\TangMovMulticat$ is the
  multicategory enriched in categories defined as follows. The objects
  of $\TangMovMulticat$ are the non-negative, even integers $n$. Given
  objects $m_1,\dots,m_k$ and $n$, an object of
  $\Hom_{\TangMovMulticat}(m_1,\dots,m_k;n)=\TangMovMulticat(m_1,\dots,m_k;n)$
  is a diskular tangle $T^{m_1,\dots,m_k;n}$, together with an
  integer $P$. In the special case $k=1$ and $m_1=n$, we also
  include an identity map of $n$ as a morphism. Composition of
  morphism objects is given by composition of diskular tangles as in
  Definition~\ref{def:diskular} and adding the integers
  $P$. Given objects
  $(S,P_S),\ (T,P_T)\in \TangMovMulticat(m_1,\dots,m_k;n)$, a
  morphism from $(S,P_S)$ to $(T,P_T)$ is a finite sequence of
  elementary cobordisms which goes from $(S,P_S)$ to
  $(T,P_T)$, modulo (the transitive closure of) the following relations:
  \begin{enumerate}[label=(D\arabic*)]
  \item\label{item:disj-support} Elementary cobordisms with disjoint supports
    commute.
  \item \label{item:isoiso} Isotopic ambient
    isotopies are equal.
  \item\label{item:iso-comp} The formal composition of two ambient
    isotopies is equal to the composition of the two ambient isotopies
    in the usual sense.
  \item\label{item:iso-Reid} Applying a Reidemeister move or Morse
    move and then an ambient isotopy is equivalent to
    performing the ambient isotopy first and then the corresponding
    Reidemeister move or Morse move. Here, the diagram must have a
    disk which has exactly the form of the model Reidemeister or Morse
    move both before and after the ambient isotopy.
  \end{enumerate}
  Given multi-morphism morphisms (2-morphisms)
  $f\co (S,P_S)\to(S',P'_{S'})$ and
  $g_i\co (T_i,P_{T_i})\to (T'_i,P'_{T'_i})$ so that
  $(S,P_s)\circ\bigl((T_1,P_{T_1}),\cdots,(T_\ell,P_{T_\ell})\bigr)$
  is sensible, the multi-composition $f\circ(g_1,\dots,g_\ell)$ is
  defined by scaling down the elementary cobordisms in the $g_i$ and
  inserting them in the corresponding disks for $f$.
\end{definition}
Lemma~\ref{lem:tang-is-multicat} below states that this does, in fact,
define a multicategory.

\begin{example}
  Given an oriented diskular tangle $T$, there is a corresponding
  multi-morphism object $(T,P(T))$ in the tangle movie multicategory where
  $P(T)$ is the number of positive crossings in $T$. (See also
  Remark~\ref{rem:oriented-tang}.)
\end{example}

Consider Carter-Saito's movie moves~\cite[Figures 23--38]{CS-knot-movie}, as listed by
Khovanov~\cite[Figures 5--9]{Kho-kh-cobordism}\footnote{The only
  differences are that some of Khovanov's moves are rotated by $\pi/2$ from
Carter-Saito's, and Khovanov arranges that all strands end on the top
or bottom.} Each is a move of
layered $(m,n)$-tangles. There are two kinds of moves. Moves 8--22,
24, and 31 correspond to composing planar isotopies, or to commuting a
planar isotopy past a Reidemeister move or Morse move. The remaining
moves (moves 1--7, 23(a,b), 25--30) correspond to nontrivial
sequences of Morse moves and Reidemeister moves, at least on one
side. We will call the first class of moves \emph{Type I movie moves},
and the second class \emph{Type II movie moves}. Each Type II movie
move has a \emph{main piece}, drawn in the figure, and an identity
braid to the left and right. Identify the square with $D^2$, so the
main piece of each movie move consists of diskular $n$-tangles. We
will call the main pieces of the Type II movie moves, viewed this way,
\emph{diskular movie moves}.

\begin{definition}\label{def:tangle-multicat}
  The tangle multicategory $\TangMulticat$ is the same as the tangle
  movie multicategory $\TangMovMulticat$ except that we quotient the $2$-morphisms by
  (the transitive closure of) the following relations:
  \begin{enumerate}[resume,label=(D\arabic*)]
  \item\label{item:movies} If two sequences
    of Morse moves and Reidemeister moves are related by a diskular
    movie move then we declare them to be equal. More generally, if
    there is a round disk so that over that disk the two sequences
    differ by a movie move, and away from the disk they are the same,
    then we declare the two sequences to be equal.
  \end{enumerate}
\end{definition}

\begin{lemma}\label{lem:tang-is-multicat}
  The tangle movie multicategory and tangle multicategory are, in
  fact, multicategories.
\end{lemma}
\begin{proof}
  In both cases, we must check that:
  \begin{enumerate}[label=(MC-\arabic*)]
  \item\label{item:hor-defd} Horizontal composition (of $1$-morphisms and of $2$-morphisms) is well-defined.
  \item\label{item:hor-assoc} Horizontal composition is associative.
  \item\label{item:vert-defd} Vertical composition (of $2$-morphisms) is well-defined.
  \item\label{item:vert-assoc} Vertical composition is associative.
  \item\label{item:hor-vert} Vertical composition commutes with horizontal composition.
  \end{enumerate}

  For the tangle movie multicategory, Point~\ref{item:hor-defd} is obvious for
  $1$-morphisms, and for $2$-morphisms it follows from the fact that we imposed
  the relations that elementary cobordisms with disjoint supports commute and
  elementary cobordisms commute with planar
  isotopies. Points~\ref{item:hor-assoc},~\ref{item:vert-defd},
  and~\ref{item:vert-assoc} are obvious. Point~\ref{item:hor-vert} again uses
  the facts that elementary cobordisms with disjoint supports commute and
  elementary cobordisms commute with planar isotopies.

  For the tangle multicategory, we must check
  Points~\ref{item:hor-defd} and~\ref{item:vert-defd}; then the others
  follow from the previous case. But both of these points are still
  obvious: both horizontal and vertical gluing respect the equivalence
  relation in Definition~\ref{def:tangle-multicat}.
\end{proof}

Given a diskular tangle $T$, we can view $T$ as a
1-manifold-with-boundary inside $D^2\times\RR$, with the boundary
contained in $D^2\times \{0\}$ or, more specifically,
$(S^1\times\{0\})\cup \bigcup_i (\bdy D_i\times\{0\})$.  Given diskular
tangles $S,T$, a \emph{genuine cobordism} from $S$ to $T$ consists of:
\begin{itemize}
\item A smoothly-varying family $D_{i,t}$ of round disks inside $D^2$,
  $t\in[0,1]$, disjoint for each $t\in[0,1]$ and so that the $D_{i,0}$
  are the disks corresponding to
  $S$ and the $D_{i,1}$ are the disks corresponding to $T$.
\item A smoothly and properly embedded surface
  \[
    \Sigma\subset
    \bigl([0,1]\times D^2\times\RR\bigr)\setminus
    \bigl(\bigcup_{i,t}\{t\}\times D_{i,t}\times\RR\bigr)
  \]
  with boundary:
  \begin{itemize}
  \item $\{0\}\times S$, 
  \item $\{1\}\times T$, 
  \item the points $(t,p,0)\in [0,1]\times S^1\times\RR$ where
    $(p,0)\in \bdy S$, and
  \item the points on
    $(t,p,0)\in [0,1]\times \bdy D_{i,t}\times\RR$ which are the
    images of $\bdy S$ under the translation and scaling that sends
    $D_i$ to $D_{i,t}$.
  \end{itemize}
\end{itemize}
In particular, if $S$ and $T$ are links, this reduces to the usual
definition of a link cobordism. More generally, this is also the
standard notion of a tangle cobordism when there are no sub-disks
$D_i$.

Fixing topological models for the elementary tangle
cobordisms (mapping cylinders or traces for
types~(\ref{item:planar-isotopy}) and~(\ref{item:Reidemeister-move}), and
elementary Morse cobordisms for types~(\ref{item:birth-death}) and~(\ref{item:saddle})), any sequence of elementary cobordisms gives rise to a
genuine cobordism between diskular tangles. The following is
essentially due to Carter-Saito~\cite{CS-knot-movie}:
\begin{theorem}\label{thm:CarterSaito}
  Every isotopy class of genuine cobordisms is represented by a
  sequence of elementary tangle cobordisms. Further, two sequences of
  elementary tangle cobordisms represent isotopic genuine cobordisms
  if and only if they represent the same 2-morphism in the tangle
  multicategory $\TangMulticat$.
\end{theorem}
\begin{proof}
  The first statement, that every isotopy class of genuine cobordisms
  is represented by a sequence of elementary tangle cobordisms, is
  clear: one can isotope the cobordism to be a sequence of isotopies
  (in which the boundary disks are also allowed to move) and Morse
  moves, and then perturb the isotopy steps so each consists of a
  sequence of planar isotopies and model Reidemeister moves. For each
  of the planar isotopies, one also chooses an ambient isotopy
  covering it.

  It is also clear that each of the
  moves~\ref{item:disj-support}--\ref{item:movies} induces an isotopy
  of genuine cobordisms.

  We reduce the rest of the theorem to Carter-Saito's result by using
  the following canonical factorization of genuine cobordisms. Call a
  genuine cobordism $(\Sigma,\{D_{i,t}\})$ from $S$ to $T$
  \emph{classical} if the family of disks $D_{i,t}$ is constant
  (independent of $t$), and \emph{braid-like} if there is an ambient
  isotopy $\psi_t$ of $D^2$ extending the isotopy of the $D_i$ and so
  that $\Sigma\cap(\{t\}\times\RR^3) = \psi_t(S)$ for all
  $t\in[0,1]$. For braid-like cobordisms, we consider the ambient
  isotopy map $\psi_t$ part of the data.

  Given a genuine cobordism $(\Sigma, \{D_{i,t}\})$, there is an
  ambient isotopy $\psi_t$ of $D^2$ extending the isotopy $\{D_{i,t}\}$ (with
  $\psi_0=\Id$), and a 1-parameter family of isotopies can be
  lifted to a 1-parameter family of ambient isotopies. Let
  $\Psi\co [0,1]\times D^2\times\RR \to [0,1]\times D^2\times\RR$ be
  the trace of $\psi_t$. Given $\psi_t$, there is a
  canonical isotopy from $(\Sigma,\{D_{i,t}\})$ to
  \[
    \bigl(\Psi([0,1]\times\psi_1^{-1}(T)),\{D_{i,t}\}\bigr)
    \circ (\Psi^{-1}(\Sigma),\{D_{i,0}\}).
  \]
  Call this a \emph{braid-classical factorization} of $(\Sigma, \{D_{i,t}\})$ into the
  composition of a classical cobordism and a braid-like
  cobordism. Given a $1$-parameter family of cobordisms, there is a
  corresponding $1$-parameter family of braid-classical factorizations.
  
  Now, suppose $M$ and $M'$ are two sequences of elementary tangle
  cobordisms representing isotopic genuine cobordisms. We want to show
  that $M$ and $M'$ are related by a sequence of moves of
  type~\ref{item:disj-support}--\ref{item:movies}. By applying a
  sequence of moves of type~\ref{item:iso-Reid}, we can assume that both
  $M$ and $M'$ consist of a classical cobordism followed by a
  braid-like cobordism. Further, since the braid-classical factorization
  applies in 1-parameter families, the resulting classical cobordisms
  are isotopic through classical cobordisms, and the braid-like
  cobordisms are isotopic through braid-like cobordisms. Hence,
  the braid-like cobordisms are related by
  move~\ref{item:isoiso}.  By
  Carter-Saito's theorem~\cite{CS-knot-movie} (or rather, its folklore
  extension to tangles with fixed ends, as used by
  Khovanov~\cite{Kho-kh-cobordism} and
  Bar-Natan~\cite{Bar-kh-tangle-cob}), the classical cobordisms differ
  by a sequence of movie moves. Every movie move is either a diskular
  movie move or a move of
  type~\ref{item:disj-support},~\ref{item:isoiso},~\ref{item:iso-comp}, 
  or~\ref{item:iso-Reid}. Hence, $M$ and $M'$ differ by a sequence of
  moves of types~\ref{item:disj-support}--\ref{item:movies}, as
  desired.
\end{proof}

There is an enlargement of these categories which is also
useful. Before giving it, we introduce some terminology about
trees. Given a rooted tree $Y$ there is a partial order on the
vertices of $Y$, induced by $v>w$ if there is an edge from $v$ to $w$
and $v$ is closer to the root than $w$. This ordering induces a
function $L$ from the vertices of $Y$ to $\ZZ_{\geq 0}$ by declaring
that if $v$ and $w$ are connected by an edge, with $v>w$, then
$L(v)=L(w)+1$, and that $L$ takes the value $0$ on some vertex (which
is necessarily a leaf). So, the maximum value of $L$ occurs at the
root $v_0$; $L(v_0)-L(v)$ is the distance from $v$ to the root. We
call the vertices $L^{-1}(i)$ (for each $i\in\ZZ$) the \emph{$i\th$
  layer} of $Y$.

\begin{definition}\label{def:enlargement}\cite[Section 2.4.1]{LLS-kh-tangles}
  Given a multicategory $\Cat$ enriched in categories, the
  \emph{canonical enlargement} $\enl{\Cat}$ of $\Cat$ is defined as follows:
  \begin{itemize}
  \item The objects $\Ob(\enl{\Cat})$ are the same as $\Ob(\Cat)$.
  \item An object of $\enl{\Cat}(x_1,\dots,x_n;y)$ is a planar, rooted
    tree $Y$ with $n$ distinguished leaves called inputs, together with a
    labeling of the $k\th$ vertex of $Y$ in layer $\ell$ by a
    multi-morphism $f_{k,\ell}$ of
    $\Cat$ with the same number of inputs as the vertex, so that:
    \begin{itemize}
    \item All inputs of $Y$ are at layer $0$,
    \item The target of the last morphism (the one closest to the
      root) of $Y$ is $y$,
    \item The sources of the first layer of morphisms
      $f_{1,1},\dots,f_{p_1,1}$ are $x_1,\dots,x_n$, and
    \item Successive layers of morphisms are composable, i.e., if
      $f_{k,\ell}$ is the $i\th$ input to $f_{k',\ell+1}$ in $Y$ then
      $f_{k',\ell+1}\circ_i f_{k,\ell}$ is defined.
    \end{itemize}
    (Note that $0$-input vertices can appear at any layer of $Y$ below
    the root.)

    Given a labeled tree $(Y,\{f_{k,\ell}\})$, let
    $\circ(Y,\{f_{k,\ell}\})$ be the result of composing the
    multi-morphisms $f_{k,\ell}$ according to $Y$.
  \item Multi-composition of morphism objects is induced by composition of trees.
  \item Given morphism objects $(Y,\{f_{k,\ell}\})$ and
    $(Z,\{g_{k',\ell'}\})$, the morphisms in $\enl{\Cat}$ from $(Y,\{f_{k,\ell}\})$ to
    $(Z,\{g_{k',\ell'}\})$ are the morphisms in $\Cat$ from
    $\circ(Y,\{f_{k,\ell}\})$ to
    $\circ(Z,\{g_{k',\ell'}\})$.
  \item Multi-composition of morphism morphisms in $\enl{\Cat}$ is
    induced by multi-composition of morphism morphisms in $\Cat$.
  \end{itemize}

  There is a canonical quotient map $q\co \enl{\Cat}\to\Cat$ which is
  the identity on objects and sends $(Y,\{f_{k,\ell}\})$ to
  $\circ(Y,\{f_{k,\ell}\})$.
\end{definition}

\begin{convention}
  For the rest of the paper, the word \emph{tree} means a planar
  rooted tree.
\end{convention}

\begin{remark}
  We can visualize $\enl{\TangMovMulticat}$ as follows. Consider a
  morphism $(Y,\{T_{k,\ell}\})$ in $\enl{\TangMovMulticat}$ from
  $m_1,\dots,m_n$ to $m'$. (So, $Y$ is a tree and the $T_{k,\ell}$ are
  diskular tangles. We are suppressing the integer $P$ from
  this discussion.) There is an associated diskular tangle
  $T=\circ(Y,\{T_{k,\ell}\})$ in
  $D^2\setminus (D^2_1\cup\cdots\cup D^2_n)$. There is also a
  collection of disjoint, embedded, round circles $Z_{k,\ell}$ in
  $D^2\setminus (D^2_1\cup\cdots\cup D^2_n)$: the images of the outer
  boundaries of the $T_{k,\ell}$ under the composition
  maps. Conversely, given $T$ and the round circles $Z_{k,\ell}$, one
  can reconstruct $(Y,\{T_{k,\ell}\})$ uniquely. These circles must
  satisfy a condition on their nesting depth.
  A 2-morphism is a sequence of elementary
  tangle cobordisms, paying no regard to the extra round circles.
\end{remark}

The following lemma will be useful for constructing multifunctors
below:
\begin{lemma}\label{lem:tId}
  Given a multicategory $\Cat$ enriched in categories and a morphism
  object $(Y,\{f_{k,\ell}\})\in \enl{\Cat}(a_1,\dots,a_n;b)$ let
  \[
    \tId\co (Y,\{f_{k,\ell}\})\to \circ(Y,\{f_{k,\ell}\})
  \]
  be the morphism morphism corresponding to the identity map of
  $\circ(Y,\{f_{k,\ell}\})$. Given another morphism object
  $(Y',\{f'_{k',\ell'}\})$, any morphism morphism
  $\alpha\co (Y,\{f_{k,\ell}\})\to (Y',\{f'_{k',\ell'}\})$ can be
  factored uniquely as
  \[
    \alpha=\tId^{-1}\circ \alpha'\circ \tId
  \]
  where $\alpha'$ is a morphism from $\circ(Y,\{f_{k,\ell}\})$ to
  $\circ(Y',\{f'_{k',\ell'}\})$.
\end{lemma}
\begin{proof}
  This is immediate from the definitions: $\alpha'$ is just the
  morphism inducing $\alpha$ in the definition of $\enl{\Cat}$.
\end{proof}

\begin{remark}\label{rem:oriented-tang}
  There is an \emph{oriented tangle multicategory} with:
  \begin{itemize}
  \item one object for each pair of an even integer $m$ and a function
    $\{1,\dots,m\}\to\{\pm 1\}$, or equivalently an orientation on
    $\{e^{2\pi i j/(m+1)}\mid j=1,\dots,m\}$,
  \item a $1$-morphism for each pair of 
    a diskular tangle and an orientation of its components,
    compatible with the orientations of the points on its boundary, and
  \item a $2$-morphism for each oriented tangle cobordism.
  \end{itemize}
  There is a forgetful
  functor from the oriented tangle multicategory to the tangle
  multicategory, sending an oriented tangle $\vec{T}$ to the pair
  $(T,P)$ where $T$ is the underlying unoriented tangle and $P$ is the
  number of positive crossings of $\vec{T}$. The composition of the
  Khovanov multifunctor defined below with this forgetful functor
  gives an invariant of oriented tangles. While this is arguably a
  more natural invariant to study from the point of view of topology
  (e.g., it is clearer what Khovanov homology is an invariant of in
  this setting), we find it more convenient, and also slightly more
  general, to work at the level of the tangle multicategory.
\end{remark}

\subsection{Arc algebra multi-modules and gluing}
The goal of this section is to prove that Khovanov's arc algebras and
bimodules extend to give a functor from $\TangMovMulticat$, as a
warm-up for the spectral case. None of the ideas involved are new.

\subsubsection{The target multicategory}\label{sec:alg-target}
To be parallel with the spectral situation, we give a somewhat
elaborate multicategory as the target of Khovanov's arc algebra
functor. See Remark~\ref{rem:simple-aa} for a simpler option which is,
however, not parallel to the spectral case.

Let $A_1,\dots,A_n$ and $B$ be graded linear categories (or, less
generally, rings). A \emph{multi-module} over $A_1,\dots,A_n$ and
$B$ is just a \dg $(A_1\otimes_\ZZ\cdots\otimes_\ZZ A_n,B)$-bimodule. More
generally, the \emph{derived category of multi-modules} over
$A_1,\dots,A_n$ and $B$ is the derived category of \dg
$(A_1\otimes_\ZZ\cdots\otimes_\ZZ A_n,B)$-bimodules. For
multi-modules, however,
we can form more tensor products: given algebras $C$, $B_1,\dots,B_n$,
and $A_{i,1},\dots, A_{i,m_i}$ ($i=1,\dots,m$), multi-modules $M_i$
over $A_{i,1},\dots,A_{i,m_i}$ and $B_i$, and a multi-module $N$
over $B_1,\dots,B_n$ and $C$, we can form the tensor product
\[
   (M_1,\dots,M_n)\otimes_{B_1,\dots,B_n}N=(M_1\otimes_\ZZ\cdots\otimes_\ZZ M_n)\otimes_{B_1\otimes_\ZZ\cdots\otimes_\ZZ
    B_n}N.
\]
We can also form the derived tensor product of multi-modules by first
replacing $N$ and/or the $M_i$ by projective (or flat) multi-modules
and then taking the tensor product.

It is convenient to have a model of the derived category of
multi-modules so that the derived tensor product is strictly
associative, strictly functorial, and has a strict unit. There are
standard ways to do this; here is one. First, fix a functor from the
category of multimodules over $A_1,\dots,A_n$ and $B$ to the category
of projective multimodules, for each collection of linear categories
$A_1,\dots,A_n$ and $B$ (for example, the bar resolution if
$A_1,\dots,A_n$ and $B$ are finitely generated and free over
$\ZZ$). Then instead of the usual derived category consider the
category with objects planar, rooted trees with $n+1$ leaves, together
with a labeling of each edge by an algebra and each internal vertex by
a multi-module over the algebras associated to the edges incident to
it, so that the edge adjacent to the root is labeled by $B$ and the
edges associated to the other leaves are labeled by $A_1,\dots,A_n$
(in that order). The morphism set between two objects is obtained by
taking the (chosen) projective resolution of the module associated to
each internal vertex, tensoring the results together according to the
edges, and then taking homotopy classes of \dg module
homomorphisms. Tensor product of objects is formal: it is just given
by composition of trees. This tensor product is automatically
associative. The identity elements correspond to the tree with two
leaves and no internal vertices. It is straightforward to verify that
this extends to a strictly associative tensor product of morphisms as
well. For each tuple $A_1,\dots,A_n,B$, taking projective resolutions
and then tensoring according to the tree gives a functor from this
derived category to the usual derived category of
$(A_1\otimes_\ZZ\cdots\otimes_\ZZ A_n,B)$-bimodules. This map is fully
faithful and essentially surjective (i.e., an equivalence) by definition.

Fix this or any other model for the derived category of multimodules,
with a strictly associative, unital derived tensor product. Then, the
target of the arc algebra multifunctor is the following:
\begin{definition}\label{def:BimCat}
  Let $\BimCat$ be the multicategory enriched in categories with
  \begin{enumerate}
  \item Objects finite, graded linear categories in which each
    morphism space is free as a $\ZZ$-module.
  \item Morphisms
    $\Hom_\BimCat(A_1,\dots,A_n;B)=\BimCat(A_1,\dots,A_n;B)$ the
    derived category of graded multi-modules over $A_1,\dots,A_n$ and $B$.
  \item Multi-composition of morphisms given by the derived tensor
    product of multi-modules.
  \end{enumerate}
\end{definition}

\begin{lemma}
  The definitions above make $\BimCat$ into a multicategory.
\end{lemma}
\begin{proof}
  This is immediate from our hypotheses on the derived tensor product.
\end{proof}

\begin{definition}
  The \emph{projectivization} of $\BimCat$ is the result of
  quotienting each set of $2$-morphisms by the relation $f\sim -f$. (Note that
  multi-composition respects this equivalence relation. Also
  that, after quotienting, the 2-morphism sets are no longer abelian groups.) A
  \emph{projective functor} to $\BimCat$ is a functor to the
  projectivization of $\BimCat$.
\end{definition}

\subsubsection{The arc algebra multifunctor}\label{sec:aa-mfunc}
Given an even integer $n$, consider the $n$ points $e^{2\pi i j/(n+1)}$, $j=1,\dots,n$, in
$S^1$. Identifying $S^1\setminus\{1\}$ with $(0,1)$, we can view an
element $a\in\Crossingless{n}$ as a crossingless matching of
$\{e^{2\pi i j/(n+1)}\}$ and hence as a flat tangle in the annulus
$\annulus\subset D^2$, with boundary
$\{e^{2\pi i j/(n+1)}\}$. 
For definiteness, choose this embedding of
$a$ in $\annulus$ to be  disjoint from the line segment
$\{r e^{0i}\mid r\in[1/2,1]\}$.  Abusing notation, we continue to denote this
flat tangle by $a$.
Reflecting $a$ in the radial direction of
the annulus (i.e., reflecting across the mid-circle) gives a flat
tangle $\Wmirror{a}$, with boundary $\{\OneHalf e^{2\pi i j/(n+1)}\}$
(which corresponds, under the embedding, to the previous definition of
$\Wmirror{a}$).
Using the standard homeomorphism
\[
  \annulus\cup_{S^1\times \{1\}\sim S^1\times \{1/2\}}
  \annulus\cong \annulus,
\]
we can view $\Wmirror{a}\amalg a$ as lying in $\annulus$, and the
standard saddle cobordism $\Wmirror{a}\amalg a\to \Id$ as lying in
$[0,1]\times \annulus$. In particular, for any other crossingless
matching $b$, there is an induced cobordism $b\Wmirror{a}a\to b$
inside $\annulus\subset D^2$.

We extend Khovanov's arc algebra modules to associate multi-modules to
diskular tangles. 
Given a diskular $(m_1,\dots,m_k;n)$-tangle and an
integer $P$, as well as crossingless matchings
$a_i\in\Crossingless{m_i}$ and $b\in\Crossingless{n}$, define
\[
  \KhCx(T,P)(a_1,\dots,a_k;b)=\KhCx(\Wmirror{b}\circ T\circ(a_1,\dots,a_k),P)\grs{n/2}{0}.
\]
The right-hand side is the Khovanov complex of a link diagram in
$\RR^2$, with a grading shift, and $\circ$ denotes gluing
tangles. This has an action of $\KhCx(n)$ and $\KhCx(m_i)$ by the
standard saddle cobordisms.

Next, we note that Khovanov's theorem that gluing tangles corresponds
to the tensor product of arc algebra bimodules~\cite[Proposition
13]{Kho-kh-tangles} extends to this setting. Given a diskular
$(m_1,\dots,m_k;n_i)$-tangle $S$ and a diskular
$(n_1,\dots,n_\ell;p)$-tangle $T$, construct a \emph{gluing map}
\[
  \KhCx(T,P_T)\otimes_{\ZZ}\KhCx(S,P_S)\to \KhCx(T\circ_i S,P_S+P_T)
\]
as follows. Given crossingless matchings $(a_1,\dots,a_k)$,
$(b_1,\dots,b_\ell)$, and $c$, and resolutions $S_v$ of $S$ and $T_w$
of $T$, the canonical saddle cobordism $\Wmirror{b_i}\amalg b_i\to
\Id$ gives a cobordism
\begin{equation}\label{eq:gluing1}
  [\Wmirror{c}\circ T_w\circ(b_1,\dots,b_\ell)]\circ_i[\Wmirror{b_i}\circ S_v\circ (a_1,\dots,a_k)]
  \to
  [\Wmirror{c}\circ (T_w\circ_i S_v)\circ (b_1,\dots,b_{i-1},a_1,\dots,a_k,b_{i+1},\dots,b_\ell)].
\end{equation}
The flat tangle on the left of Formula~\eqref{eq:gluing1} is the
disjoint union of the closed 1-manifolds $[\Wmirror{c}\circ
T_w\circ(b_1,\dots,b_\ell)]$ and $[\Wmirror{b_i}\circ S_v\circ
(a_1,\dots,a_k)]$, so
\[
  V([\Wmirror{c}\circ T_w\circ(b_1,\dots,b_\ell)]\circ_i[\Wmirror{b_i}\circ S_v\circ (a_1,\dots,a_k)])=
  V([\Wmirror{c}\circ T_w\circ(b_1,\dots,b_\ell)])\otimes_{\ZZ} V([\Wmirror{b_i}\circ S_v\circ (a_1,\dots,a_k)]).
\]
Hence, applying $V$ to this cobordism gives a map
\begin{multline*}
V([\Wmirror{c}\circ T_w\circ(b_1,\dots,b_\ell)])\otimes_\ZZ  V([\Wmirror{b_i}\circ S_v\circ (a_1,\dots,a_k)])\\
  \to V([\Wmirror{c}\circ (T_w\circ_i S_v)\circ (b_1,\dots,b_{i-1},a_1,\dots,a_k,b_{i+1},\dots,b_\ell)]),
\end{multline*}
shifting the quantum grading by $|b_i|$, the number of arcs in $b_i$
(half the number of endpoints).
Far-commutativity of the saddle maps implies that these gluing maps
commute with the edge maps in the cube of resolutions. Hence, they
induce maps of iterated mapping cones
\begin{multline*}
  \KhCx([\Wmirror{c}\circ T\circ(b_1,\dots,b_\ell)],P_T) \otimes_\ZZ \KhCx([\Wmirror{b_i}\circ S\circ (a_1,\dots,a_k)],P_S)\\
  \to \KhCx([\Wmirror{c}\circ (T\circ_i S)\circ
  (b_1,\dots,b_{i-1},a_1,\dots,a_k,b_{i+1},\dots,b_\ell)],P_S+P_T).
\end{multline*}

\begin{lemma}\label{lem:Kh-gluing}
  This gluing map is a map of
  $(\KhCx(n_1),\dots,\KhCx(n_{i-1}),\KhCx(m_1),\dots,\allowbreak\KhCx(m_k),\KhCx(n_{i+1}),\dots,\allowbreak\KhCx(n_\ell);\KhCx(p))$-multi-modules.
  Further, it descends to an isomorphism
  \[
    \KhCx(T,P_T)\otimes_{\KhCx(n_i)}\KhCx(S,P_S)\to \KhCx(T\circ_i S,P_S+P_T)
  \]
  and hence to a homotopy equivalence
  \[
    \KhCx(S,P_S)\circ_i\KhCx(T,P_T)\to \KhCx(T\circ_i S,P_S+P_T),
  \]
  where the left side is composition in the tangle movie multicategory $\TangMovMulticat$.
\end{lemma}
\begin{proof}
  That the gluing map respects the multi-module structure
  and that it descends to the tensor product over $\KhCx(n_i)$ follow
  from far-commutation of disjoint saddles: The
  multi-module operations are induced by saddles away from the gluing
  region, while the gluing map is induced by saddles in the gluing
  region, so these commute. Similarly, descending to the tensor product corresponds
  to commuting saddles in different parts of the gluing region. The
  fact that the map is an isomorphism follows from the fact that it is
  an isomorphism for the case of flat diskular tangles (tangles with
  no crossings), which is proved by the argument given by
  Khovanov~\cite[Theorem 1]{Kho-kh-tangles}.  The last statement
  follows from the second and the fact that $\KhCx(S,P_S)$ is a complex of
  projective modules over $\KhCx(n_i)$.
\end{proof}

\begin{definition}\label{def:planar-Kh}
  Define $\KhCx\co \enl{\TangMovMulticat}\to \BimCat$ as follows:
  \begin{enumerate}[label=(\arabic*)]
  \item On an object $n\in2\ZZ$, $\KhCx(n)$ is the Khovanov arc algebra on
    $n$ points.
  \item Given an elementary morphism object ($1$-morphism) $(T,P)$ of
    $\TangMovMulticat$, $\KhCx(T,P)$ is the multi-module defined
    above.
  \item For a general morphism object, which is a formal composition
    of elementary morphism objects, $\KhCx$ is the corresponding composition of its
    value on the elementary morphism objects.
  \item\label{item:KhCx-el-cob} Given an elementary cobordism $\Sigma$ from $(T_0,P_0)$ to
    $(T_1,P_1)$, the map $\KhCx(\Sigma)\co \KhCx(T_0,P_0)\to
    \KhCx(T_1,P_1)\grs{\chi'(\Sigma)}{0}$ is defined in the expected
    way. 
    That is:
    \begin{enumerate}
    \item If $\Sigma$ is a planar isotopy $\Phi_t$ then $\KhCx(\Sigma)$ is the
      isomorphism obtained by applying $\Phi_1$ to each resolution.
    \item If $\Sigma$ is a Reidemeister move then $\KhCx(\Sigma)$ is
      the quasi-isomorphism coming from Khovanov's proof of invariance
      of Khovanov homology for tangles~\cite[Section 4]{Kho-kh-tangles}.
    \item If $\Sigma$ is a birth then $\KhCx(\Sigma)$ is the inclusion
      induced by labeling the new circle by the unit, and if $\Sigma$ is
      a death then $\KhCx(\Sigma)$ is the projection induced by
      applying the counit to the disappearing circle.
    \item If $\Sigma$ is a planar saddle then $\KhCx(\Sigma)$ is the
      result of applying a merge or split map to each resolution.
    \end{enumerate}
  \item\label{item:KhCx-Id} For the morphism morphism $\tId$ from
    Lemma~\ref{lem:tId} from the formal tree composition of elementary
    morphisms to the honest composition, $\KhCx(\tId)$ is the gluing
    quasi-isomorphism from Lemma~\ref{lem:Kh-gluing}.
  \item On a general morphism morphism, $\KhCx$ is induced from
    points~\ref{item:KhCx-el-cob} and~\ref{item:KhCx-Id} via
    Lemma~\ref{lem:tId}.
  \end{enumerate}
\end{definition}

The planar composition property of the arc algebra modules is
contained in the following:
\begin{proposition}\label{prop:Kh-movie-multifunc}
  Definition~\ref{def:planar-Kh} defines a projective multifunctor.
\end{proposition}
\begin{proof}
  We must verify:
  \begin{enumerate}[label=(PMF-\arabic*)]
  \item\label{item:Kh-mm-1} $\KhCx$ respects multi-composition of
    morphism objects.
  \item\label{item:Kh-mm-2} $\KhCx(\tId)$ is invertible. (This is
    needed since invertibility of $\KhCx(\tId)$ is used to define
    $\KhCx$ of arbitrary morphism morphisms.)
  \item\label{item:Kh-mm-2b} $\KhCx$ respects the equivalence relation we imposed on morphism morphisms.
  \item\label{item:Kh-mm-3} $\KhCx$ respects multi-composition
    of morphism morphisms.
  % \item\label{item:Kh-mm-4} $\KhCx$ respects the far-commutation
  %   relation that we imposed on elementary cobordisms.
  \item\label{item:Kh-mm-5} $\KhCx$ respects $2$-composition of
    morphism morphisms.
  \end{enumerate}

  Point~\ref{item:Kh-mm-1} is immediate from the definitions.

  Point~\ref{item:Kh-mm-2} follows from Lemma~\ref{lem:Kh-gluing}.

  For Point~\ref{item:Kh-mm-2b}, invariance of $\KhCx$ under
  type~\ref{item:disj-support},~\ref{item:isoiso} and~\ref{item:iso-comp} moves is obvious. Invariance
  under type~\ref{item:iso-Reid} moves follows from the definitions of the
  Reidemeister, birth, death, and saddle maps: none of these maps depend on
  the location of the tangle in the plane.
  
  For Point~\ref{item:Kh-mm-3}, we need to check two basic cases: that
  the gluing map $\KhCx(\tId)$ is associative, in the sense that given
  three tangles $R,S,T$ and integers $P_R,P_S,P_T$, the
  following diagram commutes
  \begin{equation}\label{eq:tId-tId}
    \mathcenter{
    \begin{tikzpicture}
      \node at (0,0) (tl) {$\begin{array}{l}\KhCx(T,P_T)\circ_i \bigl(\KhCx(S,P_S)\circ_{j'} \KhCx(R,P_R)\bigr)\\
          \qquad=\bigl(\KhCx(T,P_T)\circ_i \KhCx(S,P_S)\bigr)\circ_j \KhCx(R,P_R)\end{array}$};
      \node at (8,0) (tr) {$\KhCx(T,P_T)\circ_i\KhCx(S\circ_{j'}R,P_R+P_S)$};
      \node at (0,-2) (bl) {$\KhCx(T\circ_i S,P_S+P_T)\circ_j \KhCx(R,P_R)$};
      \node at (8,-2) (br) {$\begin{array}{l}\KhCx\bigl(T\circ_i(S\circ_{j'} R),P_R+P_S+P_T\bigr)\\\qquad=\KhCx\bigl((T\circ_i S)\circ_j R,P_R+P_S+P_T\bigr)\end{array}$};
      \draw[->] (tl) to node[above]{\lab{\Id_{\KhCx(T,P_T)}\otimes \KhCx(\tId)}} (tr);
      \draw[->] (tl) to node[left]{\lab{\KhCx(\tId)\otimes \Id_{\KhCx(R,P_R)}}} (bl);
      \draw[->] (tr) to node[right]{\lab{\KhCx(\tId)}} (br);
      \draw[->] (bl) to node[below]{\lab{\KhCx(\tId)}} (br);
    \end{tikzpicture}}
  \end{equation}
  and that the gluing map commutes with the maps associated to
  elementary cobordisms, in the sense that given tangles $R,S,T$ and
  an elementary cobordism $\Sigma$ from $R$ to $S$, the following
  diagram and its analogue where $T$ is pre-composed instead of post-composed commute
  \begin{equation}\label{eq:tId-Sigma}
    \mathcenter{
      \xymatrix{
        \KhCx(T,P_T)\circ_i\KhCx(R,P_R)\ar[r]^-{\KhCx(\tId)}\ar[d]_{\Id_{\KhCx(T,P_T)}\otimes \KhCx(\Sigma)}
        & \KhCx(T\circ_i R,P_T+P_R)\ar[d]^{\KhCx(\Id\circ_i\Sigma)}\\
        \KhCx(T,P_T)\circ_i\KhCx(S,P_S)\grs{\chi'(\Sigma)}{0}\ar[r]_-{\KhCx(\tId)} & \KhCx(T\circ_i S)\grs{\chi'(\Sigma)}{0}.
      }}
  \end{equation}
  Commutativity of Diagram~\eqref{eq:tId-tId} follows from
  far-commutativity of the saddle maps.  Commutativity of
  Diagram~\eqref{eq:tId-Sigma} is immediate from the local nature of
  the definition of $\KhCx(\Sigma)$.
  
  % Point~\ref{item:Kh-mm-4} is immediate from Point~\ref{item:Kh-mm-3}.
  
  For Point~\ref{item:Kh-mm-5}, it suffices to prove the result for
  morphisms between trees with a single internal vertex, i.e.,
  cobordisms between tangles in $\TangMovMulticat$. This is then
  immediate from the construction of $\KhCx$, as the composition of
  its value on the elementary cobordisms.
\end{proof}

\begin{remark}\label{rem:simple-aa}
  Since the arc algebra multi-modules are projective over $\KhCx(n)$
  and the maps associated to births, deaths, and Reidemeister moves
  are homotopy equivalences rather than just quasi-isomorphisms, we
  do not need to include taking resolutions in the composition maps
  for the target of
  $\KhCx$. That is, we could define the multi-composition to be the
  ordinary tensor product of multi-modules, and
  2-morphisms to be homotopy classes of chain maps of
  multi-modules. In the spectral case, we do not have an analogue of
  this stricter approach.
\end{remark}

\subsection{Spectral refinements}\label{sec:planar-spectral}
The target category for the spectral Khovanov multifunctor is the
spectral analogue of $\BimCat$. First, given spectral algebras
or categories $\sA_1,\dots,\sA_n$ and $\sB$ there is a notion of a
\emph{spectral multi-module} over $\sA_1,\dots,\sA_n$ and $\sB$: a functor
$(\sA_1\times\cdots\times \sA_n)^\op\times \sB\to\Spectra$ or,
equivalently, a spectrum with commuting actions of
$\sA_1,\dots,\sA_n,\sB$.  (This is a simple extension of the notion of
a bimodule from, e.g.,~\cite[Section 2]{BM-top-spectral}.) For each
$\sA_1,\dots,\sA_n$ and $\sB$, choose a cofibrant replacement functor
(the analogue of a functorial projective resolution) for the category
of spectral multi-modules. Using this, define a derived category of spectral
multi-modules with a strictly associative tensor (or smash) product as
in Section~\ref{sec:alg-target}.
Then, the target multi-category is the following adaptation of
Definition~\ref{def:BimCat}:
\begin{definition}\label{def:SBimCat}
  Let $\SBimCat$ be the multicategory enriched in spectral categories with
  \begin{enumerate}
  \item Objects finite, graded spectral categories.
  \item Multi-morphisms $\SBimCat(\sA_1,\dots,\sA_n;\sB)$ given by the derived
    category of
    multi-modules over $\sA_1,\cdots, \sA_n$ and $\sB$.
  \item Multi-composition given by the derived smash product.
  \end{enumerate}
\end{definition}

\begin{lemma}
  The definitions above make $\SBimCat$ into a multicategory.
\end{lemma}
\begin{proof}
  The proof is left to the reader.
\end{proof}

The goal of this section is to define a spectral Khovanov multifunctor
$\KhSpace\co \TangMovMulticat \to \SBimCat$.
We start constructing $\KhSpace$ by
defining it on objects of $\TangMovMulticat$, i.e., on pairs $(T,P)$
of a diskular $(m_1,\dots,m_k;n)$-tangle $T$ with $N$ crossings and an
integer $P$. The construction is essentially the same as for $(m,n)$-tangles in our
previous paper~\cite{LLS-kh-tangles} (see also
Section~\ref{sec:spec-arc-alg}).\phantomsection\label{not:big-tang-shape}
There is a \emph{tangle shape multicategory}
$\SmTshape{m_1,\dots,m_k}{n}$ with an object
$(a_1,\dots,a_k;a'_1,\dots,a'_k)$ for each pair of $k$-tuples of
crossingless matchings $a_i,a'_i\in\Crossingless{m_i}$, an object
$(b,b')$ for each pair of crossingless matchings
$b,b'\in\Crossingless{n}$, and an object $(a_1,\dots,a_k,T,b)$ for a
tuple of crossingless matchings $a_i\in \Crossingless{m_i}$ and
$b\in\Crossingless{n}$. Let $\vec{a}$ denote a $k$-tuple of
crossingless matchings $a_i\in \Crossingless{m_i}$. The multicategory $\SmTshape{m_1,\dots,m_i}{n}$ has a unique morphism of each of the following forms:
\begin{align*}
  (\vec{a}^1,\vec{a}^2),(\vec{a}^2,\vec{a}^3),\dots,(\vec{a}^{\alpha-1},\vec{a}^\alpha)&\to(\vec{a}^1,\vec{a}^\alpha)\\
  (b_1,b_2),(b_2,b_3),\dots,(b_{\beta-1},b_\beta)&\to(b_1,b_\beta)\\
  (\vec{a}^1,\vec{a}^2),\dots,(\vec{a}^{\alpha-1},\vec{a}^\alpha),(\vec{a}^\alpha,T,b_1),(b_1,b_2),\dots,(b_{\beta-1},b_\beta)&\to (\vec{a}^1,T,b_\beta).
\end{align*}
There is an associated multicategory
$\CCat{\Crossings}\ttimes \mTshape{m_1,\dots,m_k}{n}$ enriched in
groupoids~\cite[Section 3.2.4]{LLS-kh-tangles}.

Recall that we introduced a category of divided cobordisms, in
Definition~\ref{def:CobD}. To construct the tangle invariants, we will take the
quotient of this category by certain diffeomorphisms:
\begin{definition}\label{def:CobD-annulus}
  The \emph{divided cobordism category of the annulus}, $\CobD(\annulus)$, is the
  result of quotienting the divided cobordism category from
  Definition~\ref{def:CobD} by radial rescaling. That is, identifying $\annulus$
  with $[1/2,1]\times S^1$, we declare two objects of $\CobD(\annulus)$ to be
  equal if they differ by an orientation-preserving diffeomorphism of $[1/2,1]$
  which is the identity near $\{1/2,1\}$, and declare two morphisms to be equal
  if they differ by a diffeomorphism of $[0,1]\times [1/2,1]$ which is invariant
  in the $[0,1]$-direction near $\{0,1\}\times[1/2,1]$ and is the identity near
  $[0,1]\times\{1/2,1\}$. Composition descends to this quotient in an obvious
  way.

  Given a finite collection of disjoint disks $\{D_i\}\subset D^2$, define
  $\CobD(D^2\setminus\bigcup_i D_i)$ as follows. Glue the annulus $\annulus$ to
  each boundary component $\bdy D_i$ by using the maps $\phi_{D_i}$, and glue
  the annulus $\{z\in \CC\mid 1\leq |z|\leq 2\}$ to $\bdy D^2$. The result is a
  region $V\subset \CC$ containing $(D^2\setminus\bigcup_i D_i)$ in its
  interior. Then $\CobD(D^2\setminus\bigcup_i D_i)$ is $\CobD(V)$ modulo radial
  rescaling of each of the annuli we glued in.

  There are associative multi-composition maps
  \begin{align}
    \CobD(\annulus)\times\cdots\times\CobD(\annulus)\times\CobD(D^2\setminus\bigcup_i D_i)&\to \CobD(D^2\setminus\bigcup_i D_i)\label{eq:tmcomp-1}\\
    (Y_1,\dots,Y_k,Z)&\mapsto Z\circ(Y_1,\dots,Y_k)\nonumber\\
    \intertext{and}\nonumber\\
    \CobD(D^2\setminus\bigcup_i D_i)\times \CobD(\annulus)&\to\CobD(D^2\setminus\bigcup_i D_i)\label{eq:tmcomp-2}\\
    (Z,Y)\mapsto Y\circ Z.\nonumber
  \end{align}

  We can arrange the data of $\CobD(\annulus)$ and
  $\CobD(D^2\setminus\bigcup_i D_i)$ into a multicategory with objects
  \[
    \Ob\bigl(\CobD(\annulus)\times\cdots\times \CobD(\annulus)\bigr)\amalg \Ob\bigl(\CobD(D^2\setminus\bigcup_i D_i)\bigr)\amalg \Ob\bigl(\CobD(\annulus)\bigr)
  \]
  and three types of multi-morphisms, analogous to the three cases in
  $\SmTshape{m_1,\dots,m_i}{n}$, but using the composition maps from
  Formulas~\eqref{eq:tmcomp-1} and~\eqref{eq:tmcomp-2}. As in the
  usual divided cobordism category, there is a unique 2-morphism
  between isotopic morphism objects (divided cobordisms). We will abuse
  notation and denote this multicategory $\CobD(D^2\setminus\bigcup_i D_i)$.
\end{definition}

The crossing
change cobordisms and canonical saddle cobordisms induce a multifunctor
\[
  \CCat{\Crossings}\ttimes \mTshape{m_1,\dots,m_k}{n}\to
  \CobD(D^2\setminus\bigcup_i D_i).
\]
(As mentioned earlier, to define
this multifunctor one needs to choose \emph{pox} on the tangle, as
in~\cite[Definition 3.10]{LLS-kh-tangles}. The functor is, however,
independent of the choice of pox.)
Composing with the Khovanov-Burnside functor gives a multifunctor
$\mTinvNF{T}\co \CCat{\Crossings}\ttimes \mTshape{m_1,\dots,m_k}{n}\to
\mBurnside$. Applying Elmendorf-Mandell's
$K$-theory~\cite{EM-top-machine} and then rectifying gives a
multifunctor\phantomsection\label{not:KhSpaceT2}
\[
  \CCat{\Crossings}\times \SmTshape{m_1,\dots,m_k}{n}\to\Spectra.
\]
Desuspending $P$ times and taking iterated mapping cones gives a functor
$\SmTshape{m_1,\dots,m_k}{n}\to\Spectra$. This functor can be
reinterpreted (analogously to~\cite{LLS-kh-tangles}) as a spectral
multi-module $\KhSpace(T,P)$. The constructions decompose along
quantum gradings, so
\begin{align*}
  \KhSpace(n) &= \bigvee_j \KhSpace^j(n)\\
  \KhSpace(T,P)&=\bigvee_j \KhSpace^j(T,P).
\end{align*}
Here, we use the same quantum grading shifts as in the combinatorial case (Sections~\ref{sec:gradings} and~\ref{sec:aa-mfunc}), so that
\[
  C_i(\KhSpace^j(T,P))\simeq \KhCx_{i,j}(T,P)
\]
The quantum grading shift is formal, just changing the indexing in the
decomposition along quantum gradings. 

The next step in constructing the multifunctor $\KhSpace$ is to define
the maps associated to elementary cobordisms.

Consider the model diskular tangles $T_0$ and $T_1$ for a
Reidemeister move. For an R1 move, these are $2$-tangles,
for an R2 move, these are $4$-tangles, and for an
R3 move these are $6$-tangles. Let $p=1$ for a
R1 move introducing a positive crossing or a R2 move, and $0$ otherwise. In our previous paper~\cite[Proof of
Theorem 4]{LLS-kh-tangles}, we associated a zig-zag of weak
equivalences between $\KhSpace(T_0,P)$ and $\KhSpace(T_1,P+p)$. From
the definition of the derived category
this zig-zag gives an equivalence
$\KhSpace(T_0,P)\to \KhSpace(T_1,P+p)$. (It follows from
Lemma~\ref{lem:spec-bridges} below that this equivalence is, in fact,
unique up to sign.)

\phantomsection\label{not:KhSpSigma}
Given any two diskular tangles $(T,P)$ and $(T',P')$ related by a Reidemeister
move $\Sigma$, tensoring the equivalence from the previous paragraph
with the identity map of the multi-module associated to the rest of
the diskular tangle gives a map
$\KhSpace(\Sigma)\co\KhSpace(T,P)\to\KhSpace(T',P')$.

Similarly, if $U$ is an unknot diagram (with no crossings) and
$\Sigma\co \emptyset\to U$ is the birth cobordism then we 
associate to $\Sigma$ the inclusion
\[
  \KhSpace(\Sigma)\co \KhSpace(\emptyset,0)=\SphereS\into \SphereS\{-1\}\vee \SphereS\{1\}=\KhSpace(U,0)
\]
of the summand $\SphereS\{-1\}$, where the number inside braces indicates the quantum
grading. This shifts the quantum grading down by $1$. If
$\Sigma'\co U\to\emptyset$ is the death cobordism then we associate to $\Sigma'$
the projection
\[
  \KhSpace(\Sigma')\co \KhSpace(U,0)=\SphereS\{-1\}\vee \SphereS\{1\}\onto\SphereS=\KhSpace(\emptyset,0),
\]
to the summand $\SphereS\{1\}$. Again, this decreases the quantum grading by $1$. (These definitions are exactly
analogous to Khovanov homology, and are special cases of the maps
associated to elementary cobordisms of links in our previous
paper~\cite{LS-rasmus}.) As in the case of Reidemeister moves, these
extend to births or deaths of unknots in arbitrary diskular tangles,
by taking the tensor product with the identity map on the rest of the
tangle.

Let $T$ be a diskular $4$-tangle with no closed components and a single
crossing. There is an associated multifunctor
$\mTinvNF{T}\co \CCat{1}\ttimes\mTshape{}{4}\to \mBurnside$~\cite[Section
3.5]{LLS-kh-tangles}. If $T_0$ and $T_1$ denote the $0$- and $1$-resolutions
of $T$, respectively, then $\mTinvNF{T_1}$ is (isomorphic to) an insular
subfunctor of $\mTinvNF{T}$ with corresponding quotient functor (isomorphic
to) $\mTinvNF{T_0}$~\cite[Definition 3.29]{LLS-kh-tangles}. Applying
$K$-theory and rectifying, this gives a cofibration sequence
\[
  \KhSpace^j(T_1,P)\to \KhSpace^{j+1}(T,P)\to\Sigma\KhSpace^{j-1}(T_0,P).
\]
(This also uses the fact that naturally isomorphic functors give
equivalent modules; see~\cite[Proof of Proposition
4.7]{LLS-kh-tangles}.)  The Puppe construction gives a map
$\Sigma\KhSpace^{j-1}(T_0,P)\to \Sigma\KhSpace^j(T_1,P)$. This is the cobordism map
associated to a basic saddle. For a saddle in a general link diagram,
the associated map is the tensor product of this map with the identity
map of the rest of the diagram.

The last ingredients in constructing the spectral Khovanov multifunctor
are the gluing equivalences. These are defined using the analogue of
the gluing multicategory~\cite[Section 5]{LLS-kh-tangles}. Fix even
integers $\vec{m}=(m_1,\dots,m_k)$, $\vec{n}=(n_1,\dots,n_\ell)$, and
$p$, and an integer $j$ with $1\leq j\leq \ell$.\phantomsection\label{not:mGlue-2}
The \emph{tangle
  gluing multicategory} $\mGlueS{\vec{m}}{\vec{n}}{p}$ has five kinds
of objects:
\begin{itemize}
\item Pairs $(\vec{a}^1,\vec{a}^2)$ where
  $a^1_i,a^2_i\in\Crossingless{m_i}$.
\item Pairs $(\vec{b}^1,\vec{b}^2)$ where
  $b^1_i,b^2_i\in\Crossingless{n_i}$.
\item Pairs $(c^1,c^2)$ where $c^1,c^2\in\Crossingless{p}$.
\item Triples $(\vec{a},S,b)$ where $a_i\in\Crossingless{m_i}$,
  $b\in\Crossingless{n_j}$, and $S$ is a placeholder.
\item Triples $(\vec{b},T,c)$ where $b_i\in\Crossingless{n_i}$,
  $c\in\Crossingless{p}$, and $T$ is a placeholder.
\item Quadruples $(\vec{a},S,\vec{b},T,c)$ where
  $a_i\in\Crossingless{m_i}$, $b_i\in\Crossingless{n_i}$,
  $c\in\Crossingless{p}$, and $S$ and $T$ are placeholders.
\end{itemize}
The tangle shape multicategories $\SmTshape{\vec{m}}{n_j}$ and
$\SmTshape{\vec{n}}{p}$ are full subcategories of
$\mGlueS{\vec{m}}{\vec{n}}{p}$. There is also a unique multi-morphism
\begin{multline*}
  \bigl((\vec{a}^1,\vec{a}^2),\cdots,(\vec{a}^{\alpha-1},\vec{a}^\alpha),(\vec{b}^1,\vec{b}^2),\cdots,(\vec{b}^{\beta-1},\vec{b}^\beta),(\vec{a}^\alpha,S,\vec{b}^\beta,T,c^1),(c^1,c^2),\cdots,(c^{\gamma-1},c^\gamma)\bigr)
  \\
  \to(\vec{a}^1,S,\vec{b}^1,T,c^\gamma).
\end{multline*}
Given a finite set $\Crossings$, there is a groupoid-enriched product
$\CCat{\Crossings}\ttimes \mGlue{\vec{m}}{\vec{n}}{p}$ (a trivial
adaptation of the construction in~\cite[Section 5]{LLS-kh-tangles}).
Given 1-morphisms $(S,P_S)$ and $(T,P_T)$ of $\TangMovMulticat$, where $S$
is a diskular $(m_1,\dots,m_k;n_j)$-tangle and $T$ is a diskular
$(n_1,\dots,n_\ell;p)$-tangle, if we let
$\Crossings=\Crossings(S)\cup\Crossings(T)$ denote the set of
crossings of $S\cup T$, then there is a functor 
\[
  \CCat{\Crossings}\ttimes \mGlue{\vec{m}}{\vec{n}}{p}\to \CobD
\]
induced by the canonical saddle cobordisms
$\Wmirror{a_i^\alpha}\amalg a_i^\alpha\to \Id$,
$\Wmirror{b_i^\beta}\amalg b_i^\beta\to \Id$, $\Wmirror{c_i}\amalg c_i\to\Id$,
and the saddles between different resolutions of $S$ and $T$. Here, $\CobD$ is a
mild generalization of the multicategory $\CobD(D^2\setminus\bigcup_i D_i)$ from
Definition~\ref{def:diskular}, allowing diskular tangles of the form of
capped-off resolutions of $S$, capped-off resolutions of $T$, and capped-off
resolutions of $T\circ_i S$. In the last case, in addition to quotienting by
radial diffeomorphisms near the boundary circles, we also quotient by radial
reparametrization near the circle where the gluing $\circ_i$ occurred.

Composing with the Khovanov-Burnside
functor and Elmendorf-Mandell's $K$-theory gives a multifunctor
$\CCat{\Crossings}\ttimes \mGlue{\vec{m}}{\vec{n}}{p}\to\Spectra$,
which rectifies to a functor
$\CCat{\Crossings}\times \mGlueS{\vec{m}}{\vec{n}}{p}\to\Spectra$. Such
a functor induces a map of multi-modules
\begin{equation}\label{eq:s-glue}
  \KhSpace(S,P_S)\otimes_{\KhSpace(n_i)}\KhSpace(T,P_T)\to \KhSpace(T\circ_i S,P_S+P_T)
\end{equation}
(cf.~\cite[Lemma 5.4]{LLS-kh-tangles}).

\begin{lemma}\label{lem:Kh-space-gluing}
  The gluing map of spectral multi-modules from
  Formula~\eqref{eq:s-glue} is a weak equivalence.
\end{lemma}
\begin{proof}
  The induced map on homology agrees with the Khovanov gluing map from
  Section~\ref{sec:aa-mfunc} (see~\cite[Lemma 5.6]{LLS-kh-tangles}),
  so the result follows from Lemma~\ref{lem:Kh-gluing} and Whitehead's
  theorem.
\end{proof}

The following is a straightforward adaptation of
Definition~\ref{def:planar-Kh} to the spectral setting:
\begin{definition}\label{def:planar-Kh-space}
  Define $\KhSpace\co \enl{\TangMovMulticat}\to \SBimCat$ as follows:
  \begin{enumerate}[label=(\arabic*)]
  \item On an object $n\in2\ZZ$, $\KhSpace(n)$ is the spectral Khovanov arc algebra on
    $n$ points.
  \item Given an elementary morphism object (1-morphism) $(T,P)$ of
    $\TangMovMulticat$, where $T$ is a diskular
    $(m_1,\dots,m_k;n)$-tangle with $N$ crossings, $\KhSpace(T,P)$
    is the spectral Khovanov multi-module defined above.
  \item For a general morphism object, which is a formal composition
    of elementary morphism objects, $\KhSpace$ is the corresponding composition of its
    value on the elementary morphism objects.
  \item\label{item:KhSpace-el-cob} Given an elementary cobordism $\Sigma$ from $(T_0,P_0)$ to
    $(T_1,P_1)$, the map $\KhSpace(\Sigma)\co \KhSpace(T_0,P_0)\to
    \KhSpace(T_1,P_1)\grs{\chi'(\Sigma)}{0}$ is defined in the expected way. That is:
    \begin{enumerate}
    \item If $\Sigma$ is a planar isotopy then $\KhSpace(\Sigma)$ is the
      isomorphism obtained by applying $\Phi_1$ to each resolution.
    \item If $\Sigma$ is a Reidemeister move then $\KhSpace(\Sigma)$ is
      the map associated above to the Reidemeister move.
    \item If $\Sigma$ is a Morse move (birth, death, or planar saddle)
      then $\KhSpace(\Sigma)$ is the map of spectral multi-modules
      defined above.
    \end{enumerate}
  \item\label{item:KhSpace-Id} For the morphism morphism $\tId$ from
    Lemma~\ref{lem:tId} from the formal tree composition of elementary
    morphisms to the honest composition, $\KhSpace(\tId)$ is the gluing
    weak equivalence from Lemma~\ref{lem:Kh-space-gluing}.    
  \item On a general morphism morphism, $\KhSpace$ is induced from
    points~\ref{item:KhSpace-el-cob} and~\ref{item:KhSpace-Id} via
    Lemma~\ref{lem:tId}.
  \end{enumerate}
\end{definition}

\begin{proposition}\label{prop:Kh-space-movie-multifunc}
  Definition~\ref{def:planar-Kh-space} defines a projective multifunctor.
\end{proposition}
\begin{proof}
  We must check the same points~\ref{item:Kh-mm-1}--\ref{item:Kh-mm-5}
  as in the proof of Proposition~\ref{prop:Kh-movie-multifunc}.
  As there, Point~\ref{item:Kh-mm-1} is immediate from the local
  definition of the multifunctor $\KhCx$ on morphism objects.
  
  Point~\ref{item:Kh-mm-2} is Lemma~\ref{lem:Kh-space-gluing}.

  Point~\ref{item:Kh-mm-2b} follows from the same reasoning as in the
  combinatorial case. In particular, the maps associated to Reidemeister moves and Morse moves
  are again independent of the location in the plane.
  
  For Point~\ref{item:Kh-mm-3}, we must check that the analogues of
  Diagrams~\eqref{eq:tId-tId} and~\eqref{eq:tId-Sigma}
  commute. Consider first Diagram~\eqref{eq:tId-tId}. To keep notation
  simple, assume $R$ is an $(\ell;m)$-tangle, $S$ is an $(m;n)$-tangle,
  and $T$ is an $(n;p)$-tangle; only the notation is more complicated
  in the general case. Construct an analogue of the gluing
  multicategory but for the three tangles $R,S,T$. There are three maps from
  this triple-gluing multicategory to a corresponding divided
  cobordism category:
  \begin{itemize}
  \item An analogue of the gluing multifunctor, merging $aR_u\Wmirror{b}$, $bS_v\Wmirror{c}$, and
    $cT_w\Wmirror{d}$ all at once.
  \item The composition of the gluing multifunctor merging
    $aR_u\Wmirror{b}$ and $bS_v\Wmirror{c}$ with the gluing
    multifunctor merging $aR_uS_v\Wmirror{c}$ and
    $cT_w\Wmirror{d}$. That is, this corresponds to first doing the saddle
    maps $aR_u\Wmirror{b}\amalg bS_v\Wmirror{c}\to aR_uS_v\Wmirror{c}$
    and the identity on $cT_w\Wmirror{d}$ and then doing the
    saddle maps
    $aR_uS_v\Wmirror{c}\amalg cT_w\Wmirror{d}\to
    aR_uS_vT_w\Wmirror{d}$.
  \item The composition of the gluing multifunctor merging
    $bS_v\Wmirror{c}$ and $cT_w\Wmirror{d}$ with the gluing multifunctor
    merging $aR_u\Wmirror{b}$ and $bS_vT_w\Wmirror{d}$.
  \end{itemize}
  By far-commutation of saddles in the divided cobordism category, all
  three of these multifunctors are naturally isomorphic. Hence,
  composing with the Khovanov-Burnside functor and $K$-theory gives
  three naturally isomorphic multifunctors from the
  triple-gluing multicategory to the homotopy category of
  spectra. Each of these can be reinterpreted as a map
  \begin{equation}\label{eq:triple-glue}
    \KhSpace(R,P_R)\otimes^L\KhSpace(S,P_S)\otimes^L\KhSpace(T,P_T)\\
    \to\KhSpace(T\circ S\circ R,P_R+P_S+P_T).
  \end{equation}
  The fact that these maps are equal is the desired associativity
  property.

  As in the combinatorial case, commutativity of the analogue of
  Diagram~\eqref{eq:tId-Sigma} is immediate from the local definition
  of $\KhSpace(\Sigma)$.
  
  Again as in the combinatorial case,
  % Point~\ref{item:Kh-mm-4} is
  % immediate from Point~\ref{item:Kh-mm-3} and
  Point~\ref{item:Kh-mm-5} is immediate from the definitions.
\end{proof}

\section{Duality properties of Khovanov's tangle invariants and their spectral refinements}\label{sec:duality}
\emph{Wherein we} show that the arc algebra bimodule associated to the
\textsc{mirror} of a tangle $T$ is homotopy equivalent to the
\textsc{one-sided dual} of the bimodule for $T$, a result that is
\textsc{well-known} to experts, and deduce the \textsc{analogous
  result} for the \textsc{spectral refinements}.

We only need these duality results for $n$-tangles, but prove them in general.
\subsection{Dualizability for the modules and spectra}
To verify the duality theorem for the spectral bimodules, we need a technical
condition on the spectral arc algebras and modules, called
dualizability. Essentially, this is a finiteness condition, like the
fact that the isomorphism between a vector space and its double dual
holds only for finite-dimensional vector spaces. Dualizability has a
number of implications, including relating the chains on the dual with
the cochains on the original spectral module.

\begin{definition}
  Let $A$ be a \dg algebra or spectral algebra. A (\dg or spectral)
  $A$-module $X$ is \emph{dualizable} if, for all $A$-modules $Z$, the
  natural map
  \begin{equation}\label{eq:dual-nat-map}
    \RHom_A(X,A) \otimes_A Z \to \RHom_A(X,Z)
  \end{equation}
  is a weak equivalence. Here, $\RHom_A$ denotes the derived functor of the
  space of maps of left (dg or spectral) $A$-modules (cf.\ Section~\ref{sec:lin-cat}).

  Given another (\dg or spectral) algebra $B$, an $(A,B)$-bimodule $X$
  is \emph{left-dualizable} if $X$ is dualizable as an $A$-module, and
  \emph{right-dualizable} if $X$ is dualizable as a $B$-module.
\end{definition}

The following properties are straightforward to verify:

\begin{proposition}\label{prop:thick-is-dualizable}
  For any (\dg or spectral) algebra $A$, the collection of dualizable
  $A$-modules is closed under the following.
  \begin{enumerate}
  \item \emph{Equivalence}: if $X$ is dualizable and $Y \simeq X$
    then $Y$ is dualizable.
  \item \emph{Retracts}: if $Y$ is dualizable and $X$ is a retract of
    $Y$ then $X$ is dualizable.
  \item \emph{Sums}: If $X$ and $Y$ are dualizable then so is the sum
    $X \oplus Y$.
  \item \emph{Shifts}: If $X$ is dualizable then so are the shifts
    $\Sigma^n X$ for $n \in \ZZ$.
  \item \emph{Cofibers}: if $f\co X \to Y$ is a map of dualizable
    $A$-modules then the mapping cone $Cf$ is dualizable.
  \item \emph{Unit}: $A$ is dualizable.
  \end{enumerate}
  Further, the category of dualizable $A$-modules is the smallest category of
  $A$-modules with this property.
\end{proposition}
In other words, the category of dualizable $A$-modules is the smallest
\emph{thick subcategory} of the homotopy category of $A$-modules
containing $A$.

For spectra, the homology Whitehead theorem implies the following
well-known criterion for dualizability as modules over the sphere
spectrum $\SphereS$.
\begin{proposition}
  \label{prop:dualizable-spectrum}
  A spectrum $X$ is dualizable over $\SphereS$ if and only if $X$ is $k$-connective for
  some $k$ and its homology
  \[
    H_*(X) = \bigoplus_n H_n(X;\ZZ)
  \]
  is a finitely generated abelian group.
\end{proposition}

\begin{definition}
  An $R$-algebra $A$ is \emph{proper} if it is dualizable as an
  $R$-module.
\end{definition}

\begin{proposition}
  If $A$ is a proper $R$-algebra then every dualizable
  $A$-module is also a dualizable $R$-module.
\end{proposition}

\begin{proof}
  This follows from the fact that, given a dualizable $A$-module $X$ and an $R$-module $Z$, the natural map from
  Equation~\eqref{eq:dual-nat-map} factors as
  \begin{align*}
    \RHom_R(X,R)\otimes_R Z&\cong \RHom_R(X\otimes_A A,R)\otimes_RZ
                            \cong \RHom_A(X,\RHom_R(A,R))\otimes_R Z\\
                          &\cong \RHom_A(X,A)\otimes_A\RHom_R(A,R)\otimes_RZ
                            \cong \RHom_A(X,A)\otimes_A\RHom_R(A,Z)\\
                          &\cong \RHom_A(X,\RHom_R(A,Z))
                            \cong \RHom_R(X\otimes_AA,Z)\\
                          &\cong \RHom_R(X,Z)
  \end{align*}
  where the second line uses dualizability of $X$ over $A$ and then of $A$ over
  $R$. (Throughout, by tensor product we mean the derived functor associated to
  tensor product.)
\end{proof}

\begin{proposition}\label{prop:dual-chains}
  If $A$ a dualizable spectral algebra and $X$ is a dualizable
  $A$-module then the natural map
  \[
    C_*(\RHom_A(X,A))\to \RHom_{C_*(A)}(C_*(X),C_*(A))
  \]
  from singular chains on the morphism spectrum to the morphism complex of
  singular chain complexes induces an isomorphism on homology.  The
  same applies to one-sided Homs of left-dualizable spectral
  bimodules.
\end{proposition}
\begin{proof}
  This follows from Proposition~\ref{prop:thick-is-dualizable} and induction. In
  the case $X=A$, $\RHom_A(A,A)\simeq A$ so the left side is $C_*(A)$, while the
  right side is $\RHom_{C_*(A)}(C_*(A),C_*(A))\simeq C_*(A)$; the natural map
  respects the right action of $C_*(A)$ (by the target), so is determined by
  where it sends the identity map, and hence is an equivalence. The category of
  $A$-modules for which the result holds is closed under equivalences, retracts,
  sums, shifts, and cofibers, hence contains all dualizable $A$-modules.
\end{proof}

\begin{proposition}\label{prop:Kh-dualizable}
  The arc algebra module $\KhCx(T,P_T)$ associated to an $(m,n)$-tangle
  $T$ is left-dualizable and right-dualizable.
\end{proposition}
\begin{proof}
  We prove right-dualizability; left-dualizability is similar.
  An \emph{elementary projective right module} over $\KhCx(m)$ is a
  module of the form $\KhCx(m)(a,\cdot)$, for some crossingless
  matching $a$.
  Elementary projective modules over $\KhCx(m)$ are retracts of
  $\KhCx(m)$. The homological grading
  gives a filtration of $\KhCx(T,P_T)$ so that each sub-quotient is
  homotopy equivalent to a finite direct sum of shifts of elementary
  projective modules. By Proposition~\ref{prop:thick-is-dualizable},
  the category of dualizable modules is closed under shifts, sums, and
  retracts, and contains the algebra, so each sub-quotient is
  dualizable. The fact that dualizability is preserved by mapping
  cones and induction then gives the result.
\end{proof}

Similarly:
\begin{proposition}\label{prop:Kh-dualizable-spectrum}
  The spectral arc algebras $\KhSpace(n)$ are dualizable and the
  spectral bimodules $\KhSpace(T,P)$ are left- and right-dualizable.
\end{proposition}
\begin{proof}
  The first statement follows from
  Proposition~\ref{prop:dualizable-spectrum}. The proof of the second
  statement is the same as the proof of
  Proposition~\ref{prop:Kh-dualizable}: the cube induces a filtration
  of $\KhSpace(T,P)$ so that each sub-quotient is equivalent to a wedge
  sum of shifts of retracts of $\KhSpace(m)$.
\end{proof}

\subsection{Arc algebra bimodules for mirrors}\label{sec:aa-dual}
In this section, we give two proofs of a well-known duality property
for Khovanov's tangle invariants. The first is a TQFT-style argument,
using functoriality of Khovanov homology. This proof is elegant and
will gives a useful framework for the spectral case discussed in the
next section. Since we are also re-proving functoriality of Khovanov
homology itself in this paper, to avoid circular reasoning we give a
second, direct proof of the case of this duality result needed there.

% In this section, we write down the proof of a well-known duality
% property for Khovanov's tangle invariants, which follows from
% functoriality of Khovanov homology and a familiar TQFT-style
% argument. We use this formulation to prove functoriality of the
% spectral refinements. Since we are also proving functoriality of
% Khovanov homology itself, we also give a direct proof of the case of
% this duality result needed there.

The duality results in this section perhaps first appeared in the work
of Clark-Morrison-Walker~\cite[Theorem 1.3]{CMW-kh-functoriality}.

\begin{figure}
  \centering
  \includegraphics{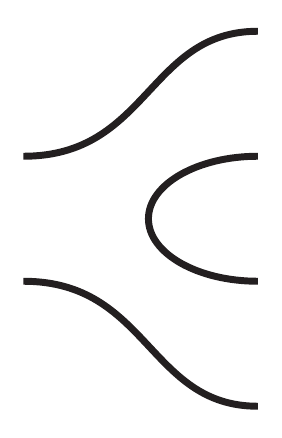}\qquad\qquad
  \includegraphics[width=1.75in]{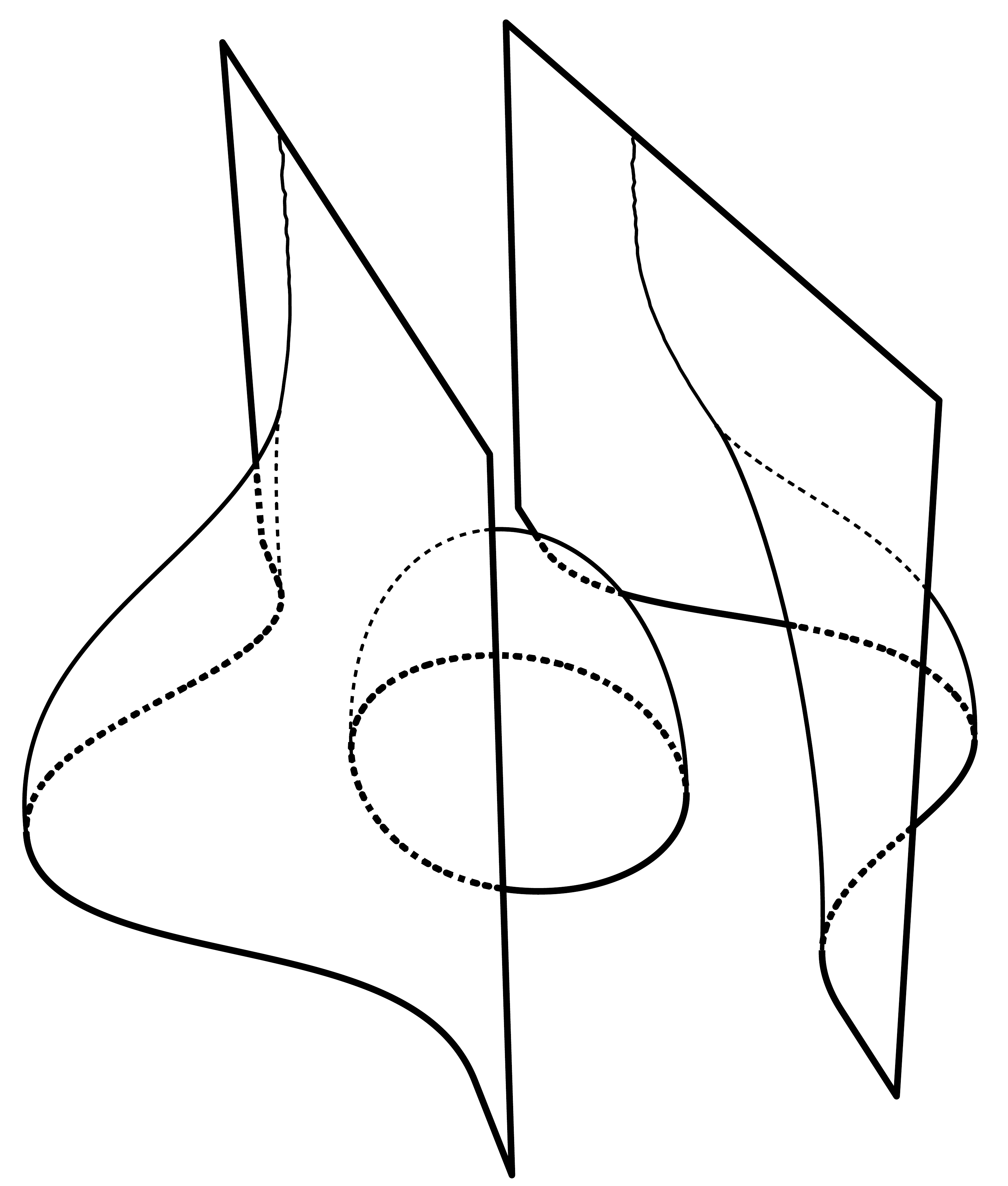}\qquad\qquad
  \includegraphics[width=2.25in]{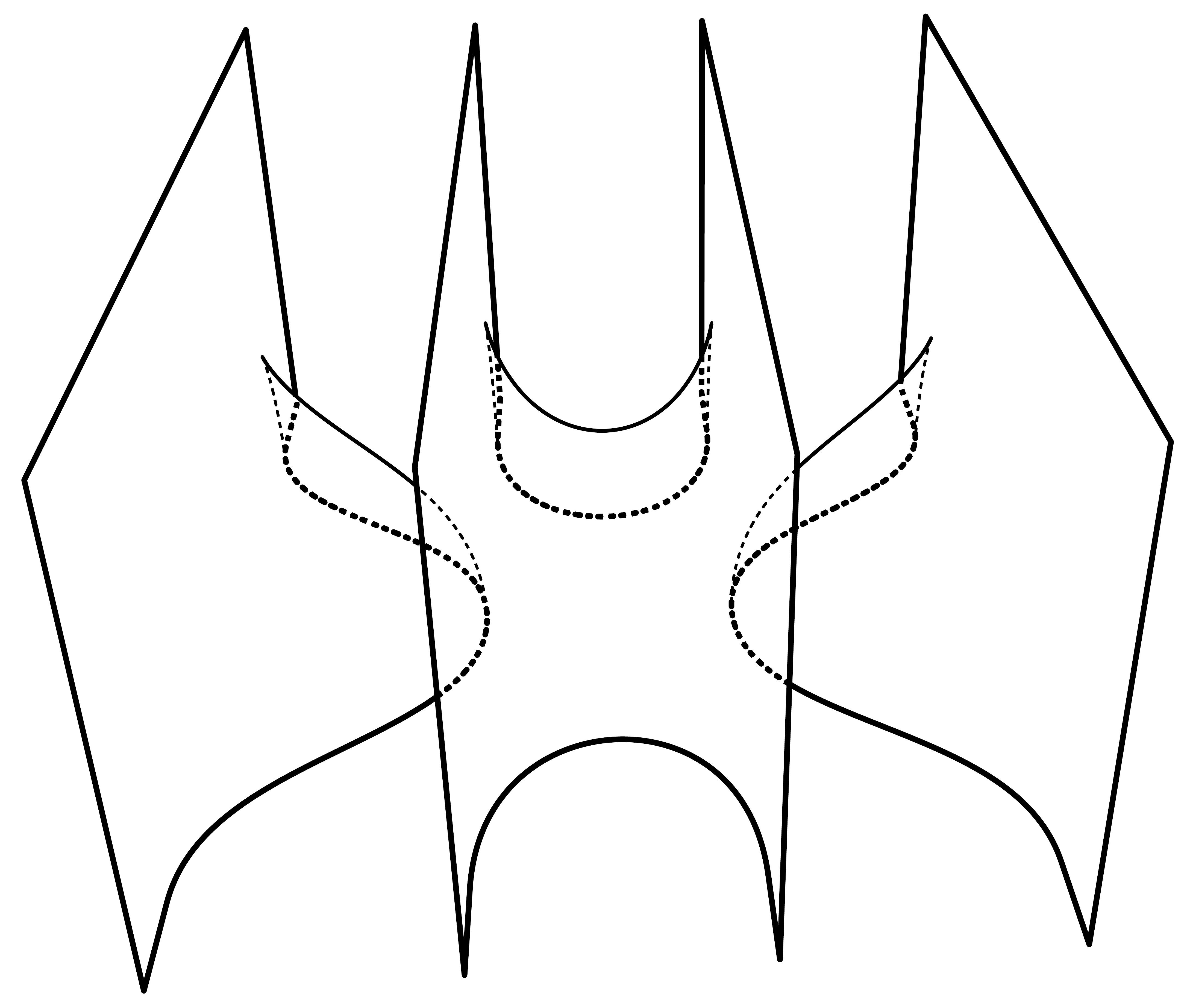}
  \caption{\textbf{The cobordisms $\Sigma_{T\Wmirror{T}}$ and
      $\Sigma_{\Wmirror{T}T}$.} Left: the tangle $T$. Center and
    right: the cobordisms $\Sigma_{T\Wmirror{T}}$ and
    $\Sigma_{\Wmirror{T}T}$. The space shown is the projection
    $[0,1]\times [0,1]\times (0,1)$ of the space
    $[0,1]\times[0,1]\times(0,1)^2$ where the tangle cobordisms
    lie. The cobordism direction is the first coordinate, read
    upwards. Anaglyph versions of this figure and
    Figure~\ref{fig:TTcob-id} can be found at \url{http://pages.uoregon.edu/lipshitz/CSI.html}}
  \label{fig:TTcob}
\end{figure}

Given an $(m,n)$-tangle $T$ in $[0,1]\times(0,1)\times(0,1)$,
$[0,1]\times T$ is a tangle cobordism in
$[0,1]\times [0,1]\times(0,1)\times(0,1)$. Identifying
$(\{0\}\times [0,1])\cup ([0,1]\times \{1\})\cup (\{1\}\times [0,1])$ with
$[0,1]$, this cobordism can be viewed as a tangle cobordism
$\Sigma_{T\Wmirror{T}}$ from the $(m,m)$-tangle $T\Wmirror{T}$ to
the identity braid on $m$ points. Similarly, this cobordism can be
viewed as a tangle cobordism $\Sigma_{\Wmirror{T}T}$ from the identity
braid on $n$ points to the $(n,n)$-tangle $\Wmirror{T}T$. See
Figure~\ref{fig:TTcob}. (There are also similar cobordisms
$\Wmirror{T}T\to\Id_n$ and $\Id_m\to T\Wmirror{T}$, but we will not name
or need these.) Let $N$ be the number of crossings of $T$. For
any integer $P$ there are corresponding maps
\begin{align*}
     \KhCx(\Sigma_{T\Wmirror{T}})&\co \KhCx(T,P)\otimes_{\KhCx(n)}\KhCx(\Wmirror{T},N-P)\grs{\tfrac{m-n}{2}}{0}=\KhCx(T\Wmirror{T},N)\grs{\tfrac{m-n}{2}}{0}\to \KhCx(\Id_{m})=\KhCx(m)\\
     \KhCx(\Sigma_{\Wmirror{T}T}) &\co \KhCx(n)=\KhCx(\Id_{n})\to \KhCx(\Wmirror{T}T,N)\grs{\tfrac{m-n}{2}}{0}=\KhCx(\Wmirror{T},N-P)\otimes_{\KhCx(m)}\KhCx(T,P)\grs{\tfrac{m-n}{2}}{0}.
\end{align*}
The cobordisms $\Sigma_{T\Wmirror{T}}$ and $\Sigma_{\Wmirror{T}T}$
satisfy that
\[
  (\Sigma_{T\Wmirror{T}}\cup \Id_{T})\circ(\Id_{T}\cup \Sigma_{\Wmirror{T}T})
\]
is isotopic to the obvious ambient isotopy from $T\cup\Id$ to $\Id\cup T$
and
\[
  (\Id_{\Wmirror{T}}\cup \Sigma_{T\Wmirror{T}})\circ(\Sigma_{\Wmirror{T}T}\cup\Id_{\Wmirror{T}})
\]
is isotopic to the obvious ambient isotopy from $\Id\cup \Wmirror{T}$
to $\Wmirror{T}\cup\Id$. See
Figure~\ref{fig:TTcob-id}.  Hence, if we identify $\KhCx(\Id\cup
T,P)=\KhCx(T,P)=\KhCx(T\cup\Id,P)$ and
$\KhCx(\Wmirror{T}\cup\Id,N-P)=\KhCx(\Wmirror{T},N-P)=\KhCx(\Id\cup\Wmirror{T},N-P)$ 
via the ambient isotopy then functoriality of Khovanov homology
implies that
\begin{align}
  (\KhCx(\Sigma_{T\Wmirror{T}})\otimes\Id_{\KhCx(T)})\circ(\Id_{\KhCx(T)}\otimes\thinspace \KhCx(\Sigma_{\Wmirror{T}T}))&\sim\Id\co \KhCx(T,P)\to \KhCx(T,P)\label{eq:witness-dual-1}\\
  (\Id_{\KhCx(\Wmirror{T})}\otimes\thinspace
  \KhCx(\Sigma_{T\Wmirror{T}}))\circ(\KhCx(\Sigma_{\Wmirror{T}T})\otimes
  \Id_{\KhCx(\Wmirror{T})}) &\sim\Id\co \KhCx(\Wmirror{T},N-P)\to \KhCx(\Wmirror{T},N-P)\label{eq:witness-dual-2}.
\end{align}

\begin{figure}
  \centering
  \includegraphics[width=4in]{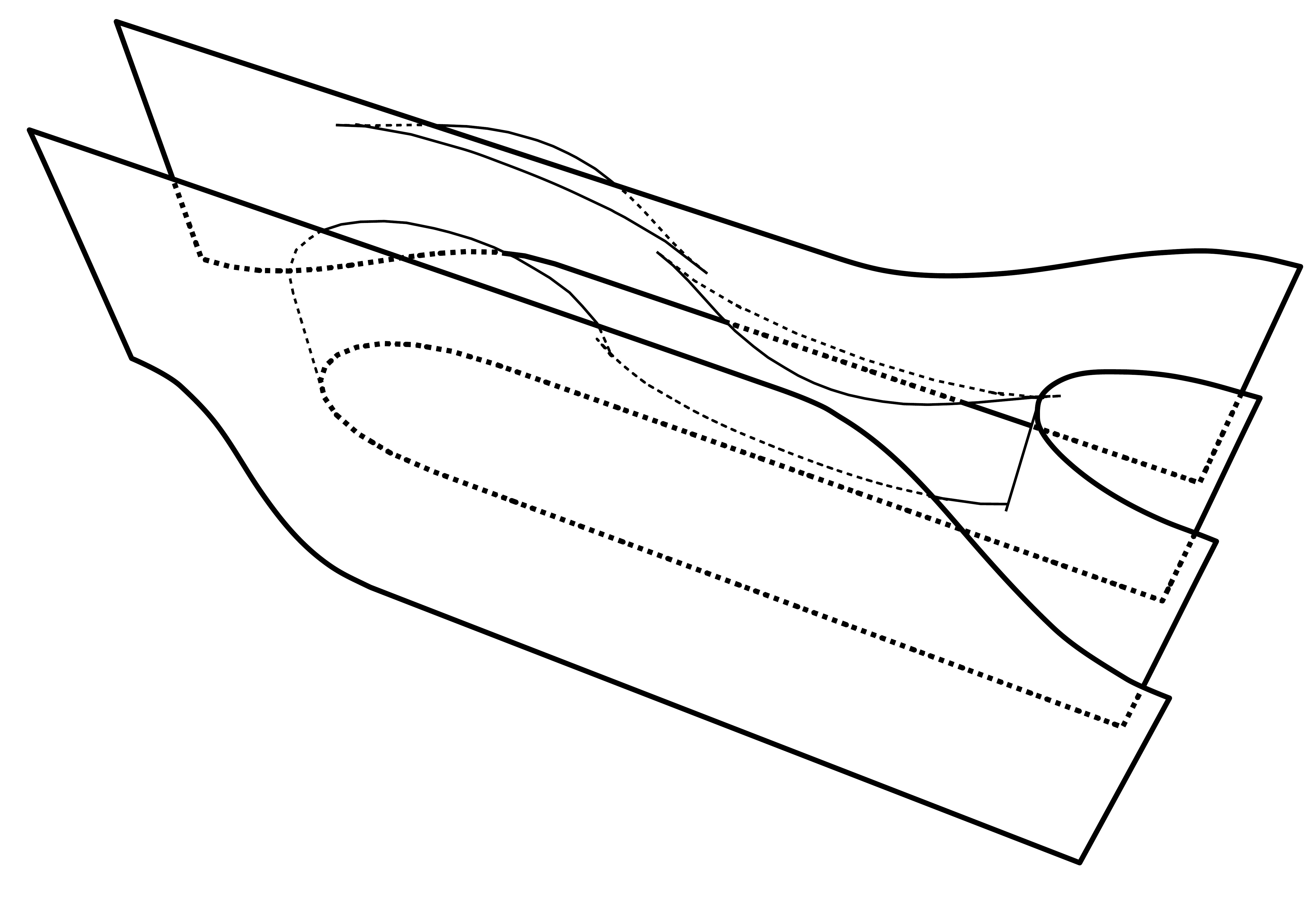}\qquad\qquad  
  \caption{\textbf{Composing cobordisms to get the identity.} With the
    same tangles and conventions as in Figure~\ref{fig:TTcob}, this is
    $(\Sigma_{T\Wmirror{T}}\cup\Id_T)\circ(\Id_T\cup
    \Sigma_{\Wmirror{T}T})$. The cobordism $(\Id_{\Wmirror{T}}\cup
    \Sigma_{T\Wmirror{T}})\circ(\Sigma_{\Wmirror{T}T}\cup
    \Id_{\Wmirror{T}})$ is similar.}
  \label{fig:TTcob-id}
\end{figure}

\begin{proposition}\label{prop:dual}
  Let $T$ be an $(m,n)$-tangle with $N$ crossings and $\Wmirror{T}$ its mirror. For any integer $P$ the
  map
  \begin{equation}\label{eq:dual}
    \begin{split}
      D\co \KhCx(\Wmirror{T},N-P)_{h,q+\frac{n-m}{2}}&\to \RHom_{\KhCx(m)}(\KhCx(T,P),\KhCx(m))_{h,q},\\
      D(x)(y)&=\KhCx(\Sigma_{T\Wmirror{T}})(y\otimes x)
    \end{split}
  \end{equation}
  is a quasi-isomorphism. (Here, the subscripts denote the homological and quantum gradings.)
  
  In particular, given $m$-tangles $T_1$ and $T_2$ with $N_1$ and
  $N_2$ crossings, respectively, and integers $P_1$ and $P_2$, we have
  \[
     \RHom_{\KhCx(m)}(\KhCx(T_1,P_1),\KhCx(T_2,P_2))_{h,q}\cong \KhCx(\Wmirror{T_1}T_2,N_1-P_1+P_2)_{h,q-m/2}.
   \]
 \end{proposition}
 (Note that in Formula~\eqref{eq:dual} we are taking the chain complex
 of left-module morphisms, not the complex of bimodule
 morphisms.)
\begin{proof}
  Equations~\eqref{eq:witness-dual-1}
  and~\eqref{eq:witness-dual-2} are the statement that $\KhCx(T,P)$ and
  $\KhCx(\Wmirror{T},N-P)\grs{\tfrac{n-m}{2}}{0}$ are dual 1-morphisms in the bicategory of
  $\ZZ$-algebras, chain complexes of bimodules, and homotopy classes
  of chain maps~\cite[Definition 6.1]{SP-other-dual}. Since
  $\KhCx(T,P)$ and $\RHom_{\KhCx(m)}(\KhCx(T,P),\KhCx(m))$ are also a
  dual pair, the result follows from (the proof of) uniqueness of the
  dual of a dualizable $1$-morphism (essentially~\cite[Proposition 2.10.5]{EGNO-other-tens-book}, for instance).   
  The second statement follows by tensoring the first statement with
  $\KhCx(T_2,P_2)$ and then applying Proposition~\ref{prop:Kh-dualizable} and the composition theorem
  for the tangle invariants.  
\end{proof}

To avoid circular reasoning, we also give a direct proof of the isomorphism in
Proposition~\ref{prop:dual} for the special case of
$(m,0)$-tangles. (The eventual reasoning is that
Proposition~\ref{prop:dual-direct} can be used to prove functoriality of Khovanov
homology, Theorem~\ref{thm:Kh-functorial}, which then implies
Formulas~\eqref{eq:witness-dual-1} and~\eqref{eq:witness-dual-2},
hence Proposition~\ref{prop:dual}. Proposition~\ref{prop:dual} is then
used to prove its spectral version, Proposition~\ref{prop:spec-dual},
which in turn implies functoriality of the stable homotopy type, Theorem~\ref{thm:Kh-spec-functorial}.)
\begin{proposition}\label{prop:dual-direct}
  Let $T$ be an $(m,0)$-tangle and $\Wmirror{T}$ its mirror. Then, there is an isomorphism
  \begin{equation}\label{eq:dual-direct}
    \RHom_{\KhCx(m)}(\KhCx(T,P),\KhCx(m))_{h,q}\cong \KhCx(\Wmirror{T},N-P)_{h,q-m/2}.
  \end{equation}
  In particular, given $(m,0)$-tangles $T_1$ and $T_2$, we have
  \[
     \RHom_{\KhCx(m)}(\KhCx(T_1,P_1),\KhCx(T_2,P_2))_{h,q}\cong \KhCx(\Wmirror{T_1}T_2,N_1-P_1+P_2)_{h,q-m/2}.
   \]
\end{proposition}
\begin{proof}
  For the first statement, suppose initially that $T$ is a flat
  tangle and $P=0$. Write $T$ as the union of (the mirror of) a crossingless matching $\Wmirror{a}$ and
  $k$ unknots. Then, $\KhCx(T,0)=V^{\otimes k}\otimes
  \KhCx(\Wmirror{a},0)$. Since $\KhCx(\Wmirror{a},0)$ is an elementary projective
  module, an element $f\in \RHom(\KhCx(\Wmirror{a},0),\KhCx(m))$ is determined by
  $f(1_a)$ (where $1_a\in V(a\Wmirror{a})$). Further, $f(1_a)=1_af(1_a)$, so $f(1_a)$ must be an
  element of $1_a\KhCx(m)=\KhCx(\Wmirror{a},0)$. The map $1_a\mapsto 1_a$
  generates this $\KhCx(m)$-module. The element $1_a\in \KhCx(\Wmirror{a},0)$ has
  quantum grading $-m/2$, while $1_a\in \KhCx(m)$ has quantum grading $0$, so
  this map shifts the quantum grading up by $m/2$.  We also have
  $\KhCx(\Wmirror{T},0)\cong V^{\otimes k}\otimes
  \KhCx(a,0)$, but here the quantum grading is shifted up by $m/2$ (so
  if $k=0$, $1_a$ would have quantum grading $0$, not $-m/2$).
  Finally, the isomorphism $V\cong V^*$ which sends
  \[
    1\mapsto (X\mapsto 1,\ 1\mapsto 0)\qquad\qquad X\mapsto (X\mapsto
    0,\ 1\mapsto 1)
  \]
  preserves the quantum grading. Hence, overall, the isomorphism
  decreases the quantum grading by $m/2$.

  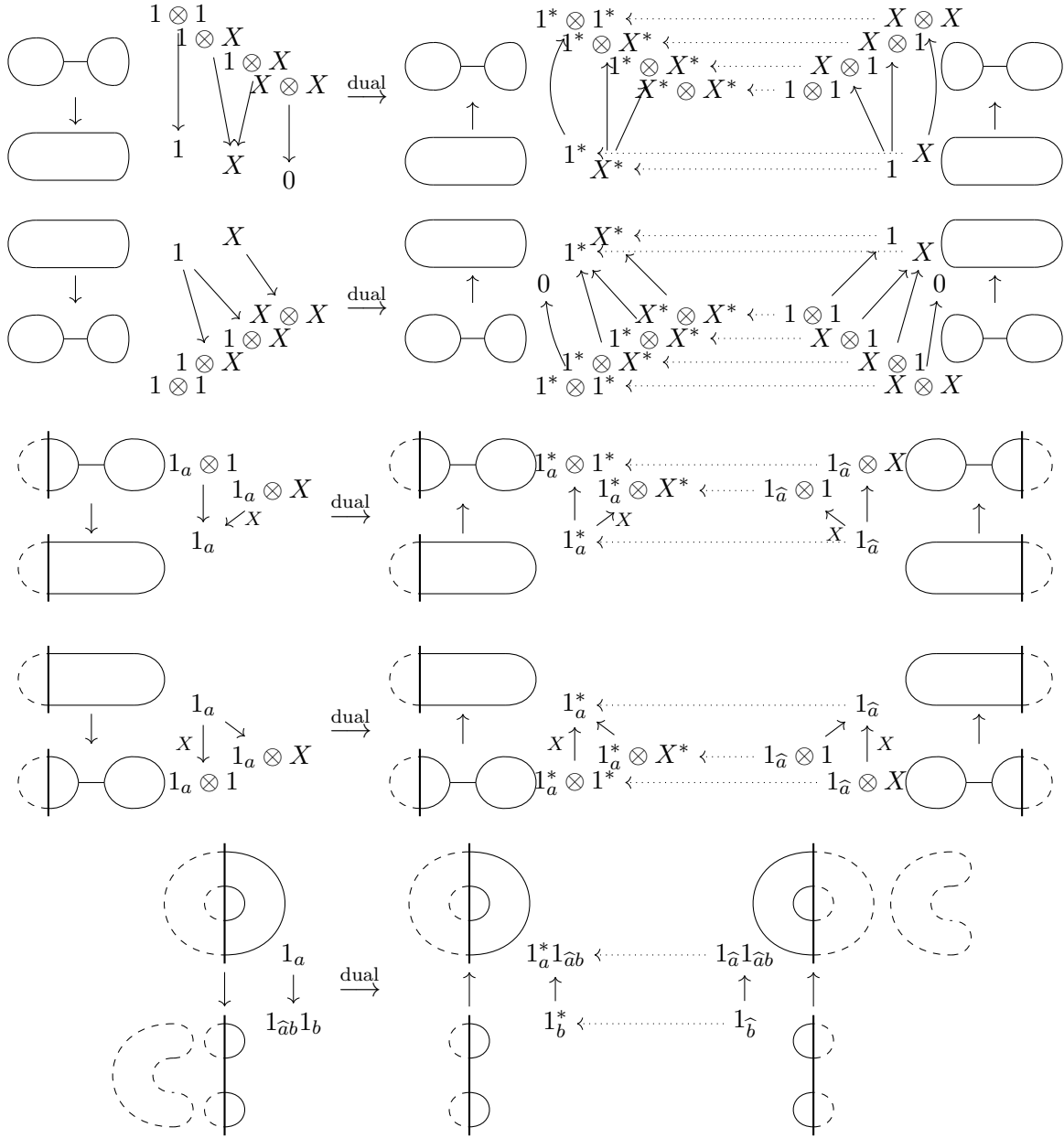
\begin{figure}
    \centering
\[
    \mathcenter{\begin{tikzpicture}[scale=.23]
      \node  (2) at (5, 3) {};
      \node  (3) at (5, 0) {};
      \node  (4) at (3.25, 1.5) {};
      \node  (5) at (0, 3) {};
      \node  (6) at (0, 0) {};
      \node  (7) at (1.75, 1.5) {};
      \node  (8) at (5, -3) {};
      \node  (9) at (5, -6) {};
      \node  (11) at (0, -3) {};
      \node  (12) at (0, -6) {};
      \node  (14) at (2.5, -0.75) {};
      \node  (15) at (2.5, -2.5) {};
      \node  (oneone) at (9, 4.5) {$1\otimes 1$};
      \node  (one) at (9, -4) {$1$};
      \node  (oneX) at (11, 3) {$1\otimes X$};
      \node  (Xone) at (14, 1.5) {$1\otimes X$};
      \node  (XX) at (16, 0) {$X\otimes X$};
      \node  (zero) at (16, -6) {$0$};
      \node  (X) at (12.5, -5) {$X$};
      \draw [bend left=90, looseness=1.00] (2.center) to (3.center);
      \draw [bend right=45] (2.center) to (4.center);
      \draw [bend right=45] (4.center) to (3.center);
      \draw [in=180, out=-180, looseness=2.00,solid] (5.center) to (6.center);
      \draw [bend left=45] (5.center) to (7.center);
      \draw [bend left=45] (7.center) to (6.center);
      \draw (7.center) to (4.center);
      \draw [bend left=90, looseness=1.00] (8.center) to (9.center);
      \draw [bend right=90, looseness=2.00, solid] (11.center) to (12.center);
      \draw [->] (14.center) to (15.center);
      \draw (11.center) to (8.center);
      \draw (12.center) to (9.center);
      \draw[->] (oneone) to (one);
      \draw[->] (oneX) to (X);
      \draw[->] (Xone) to (X);
      \draw[->] (XX) to (zero);
    \end{tikzpicture}}
    \stackrel{\text{dual}}{\longrightarrow}
        \mathcenter{\begin{tikzpicture}[scale=.23]
      \node  (2) at (5, 3) {};
      \node  (3) at (5, 0) {};
      \node  (4) at (3.25, 1.5) {};
      \node  (5) at (0, 3) {};
      \node  (6) at (0, 0) {};
      \node  (7) at (1.75, 1.5) {};
      \node  (8) at (5, -3) {};
      \node  (9) at (5, -6) {};
      \node  (11) at (0, -3) {};
      \node  (12) at (0, -6) {};
      \node  (14) at (2.5, -0.75) {};
      \node  (15) at (2.5, -2.5) {};
      \node  (oneone) at (9, 4.5) {$1^*\otimes 1^*$};
      \node  (one) at (9, -4) {$1^*$};
      \node  (oneX) at (11, 3) {$1^*\otimes X^*$};
      \node  (Xone) at (14, 1.5) {$1^*\otimes X^*$};
      \node  (XX) at (16, 0) {$X^*\otimes X^*$};
      \node  (X) at (11, -5) {$X^*$};
      \draw [bend left=90, looseness=1.00] (2.center) to (3.center);
      \draw [bend right=45] (2.center) to (4.center);
      \draw [bend right=45] (4.center) to (3.center);
      \draw [in=180, out=-180, looseness=2.00,solid] (5.center) to (6.center);
      \draw [bend left=45] (5.center) to (7.center);
      \draw [bend left=45] (7.center) to (6.center);
      \draw (7.center) to (4.center);
      \draw [bend left=90, looseness=1.00] (8.center) to (9.center);
      \draw [bend right=90, looseness=2.00, solid] (11.center) to (12.center);
      \draw [->] (15.center) to (14.center);
      \draw (11.center) to (8.center);
      \draw (12.center) to (9.center);
      \draw[->, bend left=30] (one) to (oneone);
      \draw[->] (X) to (oneX);
      \draw[->] (X) to (Xone);
      \node  (b2) at (33, 3) {};
      \node  (b3) at (33, 0) {};
      \node  (b4) at (34.75, 1.5) {};
      \node  (b5) at (38, 3) {};
      \node  (b6) at (38,0) {};
      \node  (b7) at (36.25, 1.5) {};
      \node  (b8) at (33, -3) {};
      \node  (b9) at (33, -6) {};
      \node  (b11) at (38, -3) {};
      \node  (b12) at (38, -6) {};
      \node  (b14) at (35.5, -0.75) {};
      \node  (b15) at (35.5, -2.5) {};
      \node  (boneone) at (31, 4.5) {$X\otimes X$};
      \node  (bone) at (31, -4) {$X$};
      \node  (boneX) at (29, 3) {$X\otimes 1$};
      \node  (bXone) at (26, 1.5) {$X\otimes 1$};
      \node  (bXX) at (24, 0) {$1\otimes 1$};
      \node  (bX) at (29, -5) {$1$};
      \draw [bend left=-90, looseness=1.00] (b2.center) to (b3.center);
      \draw [bend right=-45] (b2.center) to (b4.center);
      \draw [bend right=-45] (b4.center) to (b3.center);
      \draw [bend right=-90,looseness=2,solid] (b5.center) to (b6.center);
      \draw [bend left=-45] (b5.center) to (b7.center);
      \draw [bend left=-45] (b7.center) to (b6.center);
      \draw (b7.center) to (b4.center);
      \draw [bend left=-90, looseness=1.00] (b8.center) to (b9.center);
      \draw [bend right=-90, looseness=2.00, solid] (b11.center) to (b12.center);
      \draw [->] (b15.center) to (b14.center);
      \draw (b11.center) to (b8.center);
      \draw (b12.center) to (b9.center);
      \draw[->, bend right=15] (bone) to (boneone);
      \draw[->] (bX) to (boneX);
      \draw[->] (bX) to (bXone);
      \draw[->, dotted] (boneone) to (oneone);
      \draw[->, dotted] (boneX) to (oneX);
      \draw[->, dotted] (bXone) to (Xone);
      \draw[->, dotted] (bXX) to (XX);
      \draw[->, dotted] (bX) to (X);
      \draw[->, dotted] (bone) to (one);
    \end{tikzpicture}}
  \]
%Interior interior split
\[
    \mathcenter{\begin{tikzpicture}[scale=.23]
      \node  (2) at (5, -3) {};
      \node  (3) at (5, 0) {};
      \node  (4) at (3.25, -1.5) {};
      \node  (5) at (0, -3) {};
      \node  (6) at (0, 0) {};
      \node  (7) at (1.75, -1.5) {};
      \node  (8) at (5, 3) {};
      \node  (9) at (5, 6) {};
      \node  (11) at (0, 3) {};
      \node  (12) at (0, 6) {};
      \node  (14) at (2.5, 0.75) {};
      \node  (15) at (2.5, 2.5) {};
      \node  (oneone) at (9, -4.5) {$1\otimes 1$};
      \node  (one) at (9, 4) {$1$};
      \node  (oneX) at (11, -3) {$1\otimes X$};
      \node  (Xone) at (14, -1.5) {$1\otimes X$};
      \node  (XX) at (16, 0) {$X\otimes X$};
      \node  (X) at (12.5, 5) {$X$};
      \draw [bend right=90, looseness=1.00] (2.center) to (3.center);
      \draw [bend left=45] (2.center) to (4.center);
      \draw [bend left=45] (4.center) to (3.center);
      \draw [in=180, out=-180, looseness=2.00,solid] (5.center) to (6.center);
      \draw [bend right=45] (5.center) to (7.center);
      \draw [bend right=45] (7.center) to (6.center);
      \draw (7.center) to (4.center);
      \draw [bend right=90, looseness=1.00] (8.center) to (9.center);
      \draw [bend left=90, looseness=2.00, solid] (11.center) to (12.center);
      \draw [->] (15.center) to (14.center);
      \draw (11.center) to (8.center);
      \draw (12.center) to (9.center);
      \draw[->] (one) to (oneX);
      \draw[->] (one) to (Xone);
      \draw[->] (X) to (XX);
    \end{tikzpicture}}
    \stackrel{\text{dual}}{\longrightarrow}
        \mathcenter{\begin{tikzpicture}[scale=.23]
      \node  (2) at (5, -3) {};
      \node  (3) at (5, 0) {};
      \node  (4) at (3.25, -1.5) {};
      \node  (5) at (0, -3) {};
      \node  (6) at (0, 0) {};
      \node  (7) at (1.75, -1.5) {};
      \node  (8) at (5, 3) {};
      \node  (9) at (5, 6) {};
      \node  (11) at (0, 3) {};
      \node  (12) at (0, 6) {};
      \node  (14) at (2.5, 0.75) {};
      \node  (15) at (2.5, 2.5) {};
      \node  (oneone) at (9, -4.5) {$1^*\otimes 1^*$};
      \node  (one) at (9, 4) {$1^*$};
      \node  (oneX) at (11, -3) {$1^*\otimes X^*$};
      \node  (Xone) at (14, -1.5) {$1^*\otimes X^*$};
      \node  (XX) at (16, 0) {$X^*\otimes X^*$};
      \node  (X) at (11, 5) {$X^*$};
      \node  (zero) at (7, 2) {$0$};
      \draw [bend right=90, looseness=1.00] (2.center) to (3.center);
      \draw [bend left=45] (2.center) to (4.center);
      \draw [bend left=45] (4.center) to (3.center);
      \draw [in=180, out=-180, looseness=2.00,solid] (5.center) to (6.center);
      \draw [bend right=45] (5.center) to (7.center);
      \draw [bend right=45] (7.center) to (6.center);
      \draw (7.center) to (4.center);
      \draw [bend right=90, looseness=1.00] (8.center) to (9.center);
      \draw [bend left=90, looseness=2.00, solid] (11.center) to (12.center);
      \draw [->] (14.center) to (15.center);
      \draw (11.center) to (8.center);
      \draw (12.center) to (9.center);
      \draw[->] (XX) to (X);
      \draw[->] (oneX) to (one);
      \draw[->] (Xone) to (one);
      \draw[->, bend left=10] (oneone) to (zero);
      \node  (b2) at (33, -3) {};
      \node  (b3) at (33, 0) {};
      \node  (b4) at (34.75, -1.5) {};
      \node  (b5) at (38, -3) {};
      \node  (b6) at (38,0) {};
      \node  (b7) at (36.25, -1.5) {};
      \node  (b8) at (33, 3) {};
      \node  (b9) at (33, 6) {};
      \node  (b11) at (38, 3) {};
      \node  (b12) at (38, 6) {};
      \node  (b14) at (35.5, 0.75) {};
      \node  (b15) at (35.5, 2.5) {};
      \node  (boneone) at (31, -4.5) {$X\otimes X$};
      \node  (bone) at (31, 4) {$X$};
      \node  (boneX) at (29, -3) {$X\otimes 1$};
      \node  (bXone) at (26, -1.5) {$X\otimes 1$};
      \node  (bXX) at (24, 0) {$1\otimes 1$};
      \node  (bX) at (29, 5) {$1$};
      \node  (bzero) at (32, 2) {$0$};
      \draw [bend right=-90, looseness=1.00] (b2.center) to (b3.center);
      \draw [bend left=-45] (b2.center) to (b4.center);
      \draw [bend left=-45] (b4.center) to (b3.center);
      \draw [bend left=-90,looseness=2,solid] (b5.center) to (b6.center);
      \draw [bend right=-45] (b5.center) to (b7.center);
      \draw [bend right=-45] (b7.center) to (b6.center);
      \draw (b7.center) to (b4.center);
      \draw [bend right=-90, looseness=1.00] (b8.center) to (b9.center);
      \draw [bend left=-90, looseness=2.00, solid] (b11.center) to (b12.center);
      \draw [->] (b14.center) to (b15.center);
      \draw (b11.center) to (b8.center);
      \draw (b12.center) to (b9.center);
      \draw[->] (bXX) to (bX);
      \draw[->] (bXone) to (bone);
      \draw[->] (boneX) to (bone);
      \draw[->] (boneone) to (bzero);
      \draw[->, dotted] (boneone) to (oneone);
      \draw[->, dotted] (boneX) to (oneX);
      \draw[->, dotted] (bXone) to (Xone);
      \draw[->, dotted] (bXX) to (XX);
      \draw[->, dotted] (bX) to (X);
      \draw[->, dotted] (bone) to (one);
    \end{tikzpicture}}
  \]
    %Boundary interior merge.
    \[
    \mathcenter{\begin{tikzpicture}[scale=.25]
      \node  (2) at (5, 3) {};
      \node  (3) at (5, 0) {};
      \node  (4) at (3.25, 1.5) {};
      \node  (5) at (0, 3) {};
      \node  (6) at (0, 0) {};
      \node  (7) at (1.75, 1.5) {};
      \node  (8) at (5, -3) {};
      \node  (9) at (5, -6) {};
      \node  (11) at (0, -3) {};
      \node  (12) at (0, -6) {};
      \node  (14) at (2.5, -0.75) {};
      \node  (15) at (2.5, -2.5) {};
      \node  (16) at (9, 1.5) {$1_a\otimes 1$};
      \node  (1a) at (9, -3) {$1_a$};
      \node  (1aX) at (13, 0) {$1_a\otimes X$};
      \draw [bend left=90, looseness=2.00] (2.center) to (3.center);
      \draw [bend right=45] (2.center) to (4.center);
      \draw [bend right=45] (4.center) to (3.center);
      \draw [in=180, out=-180, looseness=2.00,dashed] (5.center) to (6.center);
      \draw [bend left=45] (5.center) to (7.center);
      \draw [bend left=45] (7.center) to (6.center);
      \draw (7.center) to (4.center);
      \draw [bend left=90, looseness=2.00] (8.center) to (9.center);
      \draw [bend right=90, looseness=2.00, dashed] (11.center) to (12.center);
      \draw [->] (14.center) to (15.center);
      \draw (11.center) to (8.center);
      \draw (12.center) to (9.center);
      \draw[->] (16) to (1a);
      \draw[->] (1aX) to node[right]{$\scriptstyle{X}$} (1a);
      \node (lline1) at (0,4) {};
      \node (lline2) at (0,-1) {};
      \node (lline3) at (0,-2) {};
      \node (lline4) at (0,-7) {};
      \draw[-, thick] (lline1) to (lline2);
      \draw[-, thick] (lline3) to (lline4);
    \end{tikzpicture}}
    \stackrel{\text{dual}}{\longrightarrow}
        \mathcenter{\begin{tikzpicture}[scale=.25]
      \node  (2) at (5, 3) {};
      \node  (3) at (5, 0) {};
      \node  (4) at (3.25, 1.5) {};
      \node  (5) at (0, 3) {};
      \node  (6) at (0, 0) {};
      \node  (7) at (1.75, 1.5) {};
      \node  (8) at (5, -3) {};
      \node  (9) at (5, -6) {};
      \node  (11) at (0, -3) {};
      \node  (12) at (0, -6) {};
      \node  (14) at (2.5, -0.75) {};
      \node  (15) at (2.5, -2.5) {};
      \node  (16) at (9, 1.5) {$1_a^*\otimes 1^*$};
      \node  (1a) at (9, -3) {$1_a^*$};
      \node  (1aX) at (13, 0) {$1_a^*\otimes X^*$};
      \draw [bend left=90, looseness=2.00] (2.center) to (3.center);
      \draw [bend right=45] (2.center) to (4.center);
      \draw [bend right=45] (4.center) to (3.center);
      \draw [in=180, out=-180, looseness=2.00,dashed] (5.center) to (6.center);
      \draw [bend left=45] (5.center) to (7.center);
      \draw [bend left=45] (7.center) to (6.center);
      \draw (7.center) to (4.center);
      \draw [bend left=90, looseness=2.00] (8.center) to (9.center);
      \draw [bend right=90, looseness=2.00, dashed] (11.center) to (12.center);
      \draw [->] (15.center) to (14.center);
      \draw (11.center) to (8.center);
      \draw (12.center) to (9.center);
      \draw[->] (1a) to (16);
      \draw[->] (1a) to node[right]{$\scriptstyle{X}$} (1aX);
      \node (lline1) at (0,4) {};
      \node (lline2) at (0,-1) {};
      \node (lline3) at (0,-2) {};
      \node (lline4) at (0,-7) {};
      \draw[-, thick] (lline1) to (lline2);
      \draw[-, thick] (lline3) to (lline4);
      \node  (b2) at (30, 3) {};
      \node  (b3) at (30, 0) {};
      \node  (b4) at (31.75, 1.5) {};
      \node  (b5) at (35, 3) {};
      \node  (b6) at (35,0) {};
      \node  (b7) at (33.25, 1.5) {};
      \node  (b8) at (30, -3) {};
      \node  (b9) at (30, -6) {};
      \node  (b11) at (35, -3) {};
      \node  (b12) at (35, -6) {};
      \node  (b14) at (32.5, -0.75) {};
      \node  (b15) at (32.5, -2.5) {};
      \node  (b16) at (26, 1.5) {$1_{\Wmirror{a}}\otimes X$};
      \node  (b1a) at (26, -3) {$1_{\Wmirror{a}}$};
      \node  (b1aX) at (22, 0) {$1_{\Wmirror{a}}\otimes 1$};
      \draw [bend left=-90, looseness=2.00] (b2.center) to (b3.center);
      \draw [bend right=-45] (b2.center) to (b4.center);
      \draw [bend right=-45] (b4.center) to (b3.center);
      \draw [bend right=-90,looseness=2,dashed] (b5.center) to (b6.center);
      \draw [bend left=-45] (b5.center) to (b7.center);
      \draw [bend left=-45] (b7.center) to (b6.center);
      \draw (b7.center) to (b4.center);
      \draw [bend left=-90, looseness=2.00] (b8.center) to (b9.center);
      \draw [bend right=-90, looseness=2.00, dashed] (b11.center) to (b12.center);
      \draw [->] (b15.center) to (b14.center);
      \draw (b11.center) to (b8.center);
      \draw (b12.center) to (b9.center);
      \draw[->] (b1a) to (b16);
      \draw[->] (b1a) to node[below]{$\scriptstyle{X}$} (b1aX);
      \node (blline1) at (35,4) {};
      \node (blline2) at (35,-1) {};
      \node (blline3) at (35,-2) {};
      \node (blline4) at (35,-7) {};
      \draw[-, thick] (blline1) to (blline2);
      \draw[-, thick] (blline3) to (blline4);
      \draw[->, dotted] (b1aX) to (1aX);
      \draw[->, dotted] (b16) to (16);
      \draw[->, dotted] (b1a) to (1a);
    \end{tikzpicture}}
\]
    %Boundary interior split.
    \[
    \mathcenter{\begin{tikzpicture}[scale=.25]
      \node  (2) at (5, -3) {};
      \node  (3) at (5, 0) {};
      \node  (4) at (3.25, -1.5) {};
      \node  (5) at (0, -3) {};
      \node  (6) at (0, 0) {};
      \node  (7) at (1.75, -1.5) {};
      \node  (8) at (5, 3) {};
      \node  (9) at (5, 6) {};
      \node  (11) at (0, 3) {};
      \node  (12) at (0, 6) {};
      \node  (14) at (2.5, 0.75) {};
      \node  (15) at (2.5, 2.5) {};
      \node  (16) at (9, -1.5) {$1_a\otimes 1$};
      \node  (1a) at (9, 3) {$1_a$};
      \node  (1aX) at (13, 0) {$1_a\otimes X$};
      \draw [bend right=90, looseness=2.00] (2.center) to (3.center);
      \draw [bend left=45] (2.center) to (4.center);
      \draw [bend left=45] (4.center) to (3.center);
      \draw [in=180, out=180, looseness=2.00,dashed] (5.center) to (6.center);
      \draw [bend right=45] (5.center) to (7.center);
      \draw [bend right=45] (7.center) to (6.center);
      \draw (7.center) to (4.center);
      \draw [bend right=90, looseness=2.00] (8.center) to (9.center);
      \draw [bend left=90, looseness=2.00, dashed] (11.center) to (12.center);
      \draw [->] (15.center) to (14.center);
      \draw (11.center) to (8.center);
      \draw (12.center) to (9.center);
      \draw[->] (1a) to node[left]{$\scriptstyle{X}$} (16);
      \draw[->] (1a) to  (1aX);
      \node (lline1) at (0,-4) {};
      \node (lline2) at (0,1) {};
      \node (lline3) at (0,2) {};
      \node (lline4) at (0,7) {};
      \draw[-, thick] (lline1) to (lline2);
      \draw[-, thick] (lline3) to (lline4);
    \end{tikzpicture}}
    \stackrel{\text{dual}}{\longrightarrow}
        \mathcenter{\begin{tikzpicture}[scale=.25]
      \node  (2) at (5, -3) {};
      \node  (3) at (5, 0) {};
      \node  (4) at (3.25, -1.5) {};
      \node  (5) at (0, -3) {};
      \node  (6) at (0, 0) {};
      \node  (7) at (1.75, -1.5) {};
      \node  (8) at (5, 3) {};
      \node  (9) at (5, 6) {};
      \node  (11) at (0, 3) {};
      \node  (12) at (0, 6) {};
      \node  (14) at (2.5, 0.75) {};
      \node  (15) at (2.5, 2.5) {};
      \node  (16) at (9, -1.5) {$1_a^*\otimes 1^*$};
      \node  (1a) at (9, 3) {$1_a^*$};
      \node  (1aX) at (13, 0) {$1_a^*\otimes X^*$};
      \draw [bend right=90, looseness=2.00] (2.center) to (3.center);
      \draw [bend left=45] (2.center) to (4.center);
      \draw [bend left=45] (4.center) to (3.center);
      \draw [in=180, out=-180, looseness=2.00,dashed] (5.center) to (6.center);
      \draw [bend right=45] (5.center) to (7.center);
      \draw [bend right=45] (7.center) to (6.center);
      \draw (7.center) to (4.center);
      \draw [bend right=90, looseness=2.00] (8.center) to (9.center);
      \draw [bend left=90, looseness=2.00, dashed] (11.center) to (12.center);
      \draw [->] (14.center) to (15.center);
      \draw (11.center) to (8.center);
      \draw (12.center) to (9.center);
      \draw[->] (16) to node[left]{$\scriptstyle{X}$} (1a);
      \draw[->] (1aX) to  (1a);
      \node (lline1) at (0,-4) {};
      \node (lline2) at (0,1) {};
      \node (lline3) at (0,2) {};
      \node (lline4) at (0,7) {};
      \draw[-, thick] (lline1) to (lline2);
      \draw[-, thick] (lline3) to (lline4);
      \node  (b2) at (30, -3) {};
      \node  (b3) at (30, 0) {};
      \node  (b4) at (31.75, -1.5) {};
      \node  (b5) at (35, -3) {};
      \node  (b6) at (35,0) {};
      \node  (b7) at (33.25, -1.5) {};
      \node  (b8) at (30, 3) {};
      \node  (b9) at (30, 6) {};
      \node  (b11) at (35, 3) {};
      \node  (b12) at (35, 6) {};
      \node  (b14) at (32.5, 0.75) {};
      \node  (b15) at (32.5, 2.5) {};
      \node  (b16) at (26, -1.5) {$1_{\Wmirror{a}}\otimes X$};
      \node  (b1a) at (26, 3) {$1_{\Wmirror{a}}$};
      \node  (b1aX) at (22, 0) {$1_{\Wmirror{a}}\otimes 1$};
      \draw [bend right=-90, looseness=2.00] (b2.center) to (b3.center);
      \draw [bend left=-45] (b2.center) to (b4.center);
      \draw [bend left=-45] (b4.center) to (b3.center);
      \draw [bend left=-90,looseness=2,dashed] (b5.center) to (b6.center);
      \draw [bend right=-45] (b5.center) to (b7.center);
      \draw [bend right=-45] (b7.center) to (b6.center);
      \draw (b7.center) to (b4.center);
      \draw [bend right=-90, looseness=2.00] (b8.center) to (b9.center);
      \draw [bend left=-90, looseness=2.00, dashed] (b11.center) to (b12.center);
      \draw [->] (b14.center) to (b15.center);
      \draw (b11.center) to (b8.center);
      \draw (b12.center) to (b9.center);
      \draw[->] (b16) to  node[right]{$\scriptstyle{X}$} (b1a);
      \draw[->] (b1aX) to (b1a);
      \node (blline1) at (35,-4) {};
      \node (blline2) at (35,1) {};
      \node (blline3) at (35,2) {};
      \node (blline4) at (35,7) {};
      \draw[-, thick] (blline1) to (blline2);
      \draw[-, thick] (blline3) to (blline4);
      \draw[->, dotted] (b1aX) to (1aX);
      \draw[->, dotted] (b16) to (16);
      \draw[->, dotted] (b1a) to (1a);
    \end{tikzpicture}}
\]
%Boundary boundary merge
\[
  \mathcenter{\begin{tikzpicture}[scale=.5]
    \coordinate (p1) at (0,0);
    \coordinate (p2) at (0,-1);
    \coordinate (p3) at (0,-2);
    \coordinate (p4) at (0,-3) {};
    \coordinate  (top) at (0,.25);
    \coordinate (bottom) at (0,-3.25);
    \coordinate (tail) at (0,-3.5);
    \coordinate (head) at (0,-4.5);
    \coordinate (q1) at (0,-5);
    \coordinate (q2) at (0,-6);
    \coordinate (q3) at (0,-7);
    \coordinate (q4) at (0,-8) {};
    \coordinate (r1) at (-1.5,-5);
    \coordinate (r2) at (-1.5,-6);
    \coordinate (r3) at (-1.5,-7);
    \coordinate (r4) at (-1.5,-8) {};
    \coordinate  (qtop) at (0,-4.75);
    \coordinate (qbottom) at (0,-8.25);
    \draw[-, thick] (top) to (bottom);
    \draw[in=180, out=180, dashed, looseness=2] (p1) to (p4);
    \draw[in=180, out=180, dashed, looseness=2] (p2) to (p3);
    \draw[bend right=90, looseness=2] (p4) to (p1);
    \draw[bend right=90, looseness=2] (p3) to (p2);
    \draw[in=180, out=180, dashed, looseness=2] (q1) to (q2);
    \draw[in=180, out=180, dashed, looseness=2] (q3) to (q4);
    \draw[bend right=90, looseness=2] (q4) to (q3);
    \draw[bend right=90, looseness=2] (q2) to (q1);
    \draw[in=180, out=180, dashed, looseness=2] (r1) to (r4);
    \draw[in=180, out=180, dashed, looseness=2] (r2) to (r3);
    \draw[bend right=90, looseness=2, dashed] (r4) to (r3);
    \draw[bend right=90, looseness=2, dashed] (r2) to (r1);
    \draw[-, thick] (qtop) to (qbottom);
    \draw[->] (tail) to (head);
    \node at (2,-3) (topelt) {$1_a$};
    \node at (2,-5) (botelt) {$1_{\Wmirror{a}b}1_b$};
    \draw[->] (topelt) to (botelt);
  \end{tikzpicture}}
\stackrel{\text{dual}}{\longrightarrow}
  \mathcenter{\begin{tikzpicture}[scale=.5]
    \coordinate (p1) at (0,0);
    \coordinate (p2) at (0,-1);
    \coordinate (p3) at (0,-2);
    \coordinate (p4) at (0,-3) {};
    \coordinate  (top) at (0,.25);
    \coordinate (bottom) at (0,-3.25);
    \coordinate (tail) at (0,-3.5);
    \coordinate (head) at (0,-4.5);
    \coordinate (q1) at (0,-5);
    \coordinate (q2) at (0,-6);
    \coordinate (q3) at (0,-7);
    \coordinate (q4) at (0,-8) {};
    \coordinate  (qtop) at (0,-4.75);
    \coordinate (qbottom) at (0,-8.25);
    \draw[-, thick] (top) to (bottom);
    \draw[in=180, out=180, dashed, looseness=2] (p1) to (p4);
    \draw[in=180, out=180, dashed, looseness=2] (p2) to (p3);
    \draw[bend right=90, looseness=2] (p4) to (p1);
    \draw[bend right=90, looseness=2] (p3) to (p2);
    \draw[in=180, out=180, dashed, looseness=2] (q1) to (q2);
    \draw[in=180, out=180, dashed, looseness=2] (q3) to (q4);
    \draw[bend right=90, looseness=2] (q4) to (q3);
    \draw[bend right=90, looseness=2] (q2) to (q1);
    \draw[-, thick] (qtop) to (qbottom);
    \draw[->] (head) to (tail);
    \node at (2.5,-3) (topelt) {$1_a^*1_{\Wmirror{a}b}$};
    \node at (2.5,-5) (botelt) {$1_b^*$};
    \draw[->] (botelt) to (topelt);
    \coordinate (bp1) at (10,0);
    \coordinate (bp2) at (10,-1);
    \coordinate (bp3) at (10,-2);
    \coordinate (bp4) at (10,-3) {};
    \coordinate  (btop) at (10,.25);
    \coordinate (bbottom) at (10,-3.25);
    \coordinate (btail) at (10,-3.5);
    \coordinate (bhead) at (10,-4.5);
    \coordinate (bq1) at (10,-5);
    \coordinate (bq2) at (10,-6);
    \coordinate (bq3) at (10,-7);
    \coordinate (bq4) at (10,-8) {};
    \coordinate (r1) at (14,0);
    \coordinate (r2) at (14,-1);
    \coordinate (r3) at (14,-2);
    \coordinate (r4) at (14,-3) {};
    \coordinate  (bqtop) at (10,-4.75);
    \coordinate (bqbottom) at (10,-8.25);
    \draw[-, thick] (btop) to (bbottom);
    \draw[in=180, out=180,  looseness=2] (bp1) to (bp4);
    \draw[in=180, out=180, looseness=2] (bp2) to (bp3);
    \draw[bend right=90, dashed, looseness=2] (bp4) to (bp1);
    \draw[bend right=90, dashed, looseness=2] (bp3) to (bp2);
    \draw[in=180, out=180, looseness=2] (bq1) to (bq2);
    \draw[in=180, out=180, looseness=2] (bq3) to (bq4);
    \draw[bend right=90, dashed, looseness=2] (bq4) to (bq3);
    \draw[bend right=90, dashed, looseness=2] (bq2) to (bq1);
    \draw[in=180, out=180, dashed, looseness=2] (r1) to (r4);
    \draw[in=180, out=180, dashed, looseness=2] (r2) to (r3);
    \draw[bend right=90, looseness=2, dashed] (r4) to (r3);
    \draw[bend right=90, looseness=2, dashed] (r2) to (r1);
    \draw[-, thick] (bqtop) to (bqbottom);
    \draw[->] (bhead) to (btail);
    \node at (8,-3) (btopelt) {$1_{\Wmirror{a}}1_{\Wmirror{a}b}$};
    \node at (8,-5) (bbotelt) {$1_{\Wmirror{b}}$};
    \draw[->] (bbotelt) to (btopelt);
    \draw[->, dotted] (btopelt) to (topelt);
    \draw[->, dotted] (bbotelt) to (botelt);
  \end{tikzpicture}}
\]
\caption{\textbf{Checking the duality isomorphisms induce chain maps.}
  We check that the duality isomorphisms commute with merge and split
  maps. There are cases depending on how many of the circles being
  merged or split are in the interior versus the boundary of the
  tangle diagram. Resolutions of the tangle are drawn with solid lines, and
  crossingless matchings capping it off are dashed. The duality
  isomorphism is indicated with dotted arrows. In the last case,
  $1_{\Wmirror{a}b}$ denotes labeling the circle $\Wmirror{a}b$ by
  $1$.}
    \label{fig:dual-chain}
  \end{figure}

  For the general case, we apply the isomorphism of the previous
  paragraph at each vertex. Rather than giving an abstract argument
  that these are chain maps, we simply check all the cases; see
  Figure~\ref{fig:dual-chain}.
  
  Turning to the gradings, the isomorphism exchanges $0$ and $1$
  resolutions, positive and negative crossings, and the generators $1$
  and $X$. Dualizing also negates the grading. Hence, given a
  generator of $\KhCx(T,P)$ in
  $V(aT_v)$ with grading $q$, the dual generator of
  $\RHom_{\KhCx(m)}(\KhCx(T,P),\KhCx(m))$ has grading
  $|v|-2N+3P-q$. The corresponding generator of
  $\KhCx(\Wmirror{T},N-P)$ has grading
  \[
    m/2-(N-|v|)+2N-3(N-P)-q=m/2+|v|-2N+3P-q,
  \]
  which is $m/2$ higher, as claimed.

  For the homological grading, every generator of $V(aT_v)$ has
  homological grading $N-|v|-P$, their dual generators of
  $\RHom_{\KhCx(m)}(\KhCx(T,P),\KhCx(m))$ have homological grading
  $P-N+|v|=N-(N-P)-(N-|v|)$, which is the grading of the corresponding
  generators of $\KhCx(\Wmirror{T},N-P)$.
  
  As in Proposition~\ref{prop:dual}, the second statement follows from
  the first, Proposition~\ref{prop:Kh-dualizable}, and the composition
  theorem for the tangle invariants.
\end{proof}

\subsection{Duality for spectral modules}\label{sec:spec-dual}
We have the following spectral refinement of Proposition~\ref{prop:dual}:
\begin{proposition}\label{prop:spec-dual}
  Let $T$ be a $(m,n)$-tangle with $N$ crossings and $\Wmirror{T}$ its mirror. Then, there is a weak equivalence 
  \[
    \RHom_{\KhSpace(m)}(\KhSpace(T,P),\KhSpace(m))_{q}\simeq \KhSpace^{q+\frac{n-m}{2}}(\Wmirror{T},N-P).
  \]
  (This is the space of morphisms as left module spectra.)  In particular, given
  $m$-tangles $T_1$ and $T_2$ with $N_1$ and $N_2$ crossings,
  respectively, and integers $P_1$ and $P_2$, we have
  \[
    \RHom_{\KhSpace(m)}(\KhSpace(T_1,P_1),\KhSpace(T_2,P_2))_{q}
    \simeq \KhSpace^{q-m/2}(\Wmirror{T_1}T_2,N_1-P_1+P_2).
  \]
\end{proposition}
\begin{proof}
  From Section~\ref{sec:planar-spectral}, given a tangle cobordism $\Sigma$ from $T$ to $T'$,
  decomposed as a movie, there is an induced map
  $\KhSpace(\Sigma)\co \KhSpace(T)\to\KhSpace(T')$ of spectral bimodules and a
  commutative diagram
  \begin{equation}\label{eq:homol-com}
    \mathcenter{\xymatrix{
      C_*(\KhSpace(T,P))\ar[r]^-{\KhSpace(\Sigma)_*} & C_*(\KhSpace(T',P'))\grs{\chi'(\Sigma)}{0}\\
      \KhCx(T,P)\ar[r]_-{\KhCx(\Sigma)} \ar[u]_\simeq & \KhCx(T',P')\grs{\chi'(\Sigma)}{0}\ar[u]_\simeq.
    }}
\end{equation}
  In particular, if $\Sigma_{T\Wmirror{T}}$ is the cobordism from Section~\ref{sec:aa-dual}
  then there is an induced map of spectral bimodules
  \[
    \KhSpace(\Sigma_{T\Wmirror{T}})\co
    \KhSpace(T,P)\otimes^L_{\KhSpace(n)}\KhSpace(\Wmirror{T},-P)\to \KhSpace(\Id_{m},0)\grs{n-m}{0}\simeq\KhSpace(m)\grs{\tfrac{n-m}{2}}{0}.
  \]
  There is an induced map
  \[
    \mathcal{D}\co \KhSpace(\Wmirror{T},N-P)_{q+\frac{n-m}{2}}\to \RHom_{\KhSpace(m)}(\KhSpace(T,P),\KhSpace(m))_{q}.
  \]
  By Proposition~\ref{prop:dual}, Diagram~\eqref{eq:homol-com} and
  Propositions~\ref{prop:Kh-dualizable} and~\ref{prop:dual-chains},
  the map $\mathcal{D}$ induces an isomorphism on homology. Hence, $D$
  is a weak equivalence of spectral modules (cf.~\cite[Theorem
  2.18 and proof of Theorem 5]{LLS-kh-tangles}).
\end{proof}

\section{Functoriality of Khovanov's tangle invariants and their
  spectral refinements}\label{sec:functoriality-proof}
\emph{Wherein we} prove that certain modules over the arc algebra and
spectral arc algebra have \textsc{no nontrivial automorphisms} up to sign, and
use this and similar results to \textsc{verify functoriality} for Khovanov
homology and its spectral refinement.

\subsection{Some rigidity results}
The key to Khovanov's proof of functoriality of Khovanov homology, and
hence also the key to ours, is rigidity of certain bimodules,
i.e., the fact that they have no nontrivial automorphisms. For us, the
relevant tangles are the following:
\begin{definition}
  A \emph{bridge tangle} is a diskular $n$-tangle ($n$ even) so that the corresponding
  geometric tangle is isotopic to a collection of embedded arcs in
  $S^1\times \RR \subset D^2\times\RR$. Equivalently, a bridge tangle is a
  tangle with no closed components, such that every component is unknotted and
  there is a collection of disks in the complement of $T$ separating the
  components of $T$.
\end{definition}

\begin{lemma}\label{lem:bridges}
  Let $T$ be a bridge tangle. Then, up to chain homotopy, the only
  grading-preserving chain homotopy autoequivalences of the Khovanov module
  $\KhCx(T,P)$ associated to $T$ are multiplication by $\pm 1$.
\end{lemma}
\begin{proof}
  We want to show that the only units in
  $H_{0,0}\RHom_{\KhCx(n)}(\KhCx(T,P),\KhCx(T,P))$ are $\pm 1$. By
  Proposition~\ref{prop:dual}, this group is exactly
  $\Kh_{0,-n/2}(\Wmirror{T}T,N)=\Kh_{0,-n/2}(U_{n/2},0)$, the Khovanov homology of
  the $n/2$-component unlink. Since $\Kh_{0,-n/2}(U_{n/2},0)\cong \ZZ$, the result
  follows.
\end{proof}

\begin{lemma}\label{lem:spec-bridges}
  Let $T$ be a bridge tangle. Then, up to homotopy, the only
  grading-preserving automorphisms of the spectral Khovanov module
  $\KhSpace(T,P)$ associated to $T$ are multiplication by $\pm 1$.
\end{lemma}
\begin{proof}
  Suppose $T$ has $n/2$ bridges. Let $N$ be the number of crossings of $T$. By Proposition~\ref{prop:spec-dual},
  \[
    \RHom_{\KhSpace(n)}(\KhSpace(T,P),\KhSpace(T,P))_{0,0}\simeq \KhSpace^{-n/2}(\Wmirror{T}T,N)=\KhSpace^{-n/2}(U_{n/2}),
  \]
  the Khovanov spectrum of the $n/2$-component unlink, in quantum
  grading $-n/2$. This space is exactly $\SphereS$, the sphere
  spectrum. Hence, the homotopy classes of endomorphisms are
  $\pi_0\SphereS\cong\ZZ$. The only automorphisms are $\pm1$.  
\end{proof}

\subsection{Functoriality of the arc algebra multi-modules}
In the language of Section~\ref{sec:planar}, functoriality of Khovanov
homology is the following:
\begin{theorem}\label{thm:Kh-functorial}
  The projective multi-functor $\KhCx$ from
  Definition~\ref{def:planar-Kh} descends to a projective
  multi-functor $\KhCx\co \enl{\TangMulticat}\to \BimCat$.
\end{theorem}
\begin{proof}
  We must check that the value of $\KhCx$ on 2-morphisms is invariant under
  type~\ref{item:movies} moves, i.e., under the diskular movie moves. That is,
  we must show that the main parts of the type II movie moves, viewed as maps of
  $m$-tangles (where $0\leq m\leq 8$ depends on the move), give homotopic maps
  of spectral bimodules (up to sign). Recall that the diskular movie
  moves correspond to movie moves
  1--7, 23(a), 23(b), and 35--30 in
  % 1, 2, 3, 4, 5, 6, 7, 23(a), 23(b), 25, 26, 27, 28, 29, and 30 in
  Khovanov's list.
  
  For any cobordism between bridges consisting entirely of
  Reidemeister moves and planar isotopies, Lemma~\ref{lem:bridges}
  implies that the two maps agree up to sign. This handles moves
  1--7, 23(a), 25 and 26. The remaining movie moves are 23(b), 27, 28,
  29, and 30.

  Invariance under move 23b is easy to check directly. (So is
  invariance under lots of other moves, of course.)

  For move 27, both sides are maps from the empty link to the unknot
  $U_1$ of $(h,q)$-bidegree $(0,-1)$. Further, both are compositions
  of the birth map, which maps to the unit $1\in \Kh(U_1)$, with an
  isomorphism. Since up to chain homotopy the only grading-preserving
  automorphisms of $\Kh(U_1)$ are multiplication by $\pm1$, these two
  maps agree up to sign. A similar argument applies to this movie read backwards,
  with a death in place of a birth.
  
  Similarly, both movies in move 28 are $(h,q)$-bidegree $(0,-1)$
  homomorphisms from the invariant of a single bridge $B$ to the
  invariant of a bridge union an unknot, $B\cup U_1$.  By
  Proposition~\ref{prop:dual}, this homomorphism is an element of
  $\Kh_{0,-2}(\Wmirror{B}\cup B\cup U_1)=\Kh_{0,-2}(U_2)$, where $U_2$
  denotes the 2-component unlink. This group is isomorphic to
  $\ZZ$. Further, since both maps are a birth followed by an
  isomorphism, both correspond to $\pm 1$ in $\ZZ$. A similar argument
  applies to move 28 read backwards; again, the map lies in
  $\Kh_{0,-2}(U_2)$, this time because a death map has bidegree
  $(0,-1)$.

  Move 29 is the composition of a saddle and a Reidemeister move.  The
  saddle has bidegree $(0,1)$, so by Proposition~\ref{prop:dual}, both
  sides are represented by elements of $\Kh_{0,-1}(U_1)\cong\ZZ$.
  Further, since there exist invertible cobordism maps containing some
  saddles (e.g., by move 23b), both elements must be $\pm 1$ in this
  group.

  Move 30 is the composition of a saddle and a planar isotopy. Hence,
  the corresponding maps have bidegree $(0,1)$. By
  Proposition~\ref{prop:dual} again, both sides are represented by
  elements of $\Kh_{0,-2}(U_2)\cong\ZZ$. This element is a generator
  by the same argument as for move 29. This completes the proof.
\end{proof}

\begin{corollary}
  Given an oriented link cobordism $\Sigma\co L_0\to L_1$ between
  oriented links there is an induced chain homotopy class of chain
  maps $\KhCx(\Sigma)\co \KhCx(L_0)\to \KhCx(L_1)\grs{\chi(\Sigma)}{0}$, well-defined up to
  sign. Further, given another oriented link cobordism $\Sigma'\co L_1\to L_2$,
  \[
    \KhCx(\Sigma')\circ\KhCx(\Sigma)=\pm \KhCx(\Sigma'\circ\Sigma).
  \]
\end{corollary}

\subsection{Functoriality of the spectral invariants}
Functoriality of the Khovanov stable homotopy type is the following:
\begin{theorem}\label{thm:Kh-spec-functorial}
  The projective multi-functor $\KhSpace$ from
  Definition~\ref{def:planar-Kh-space} descends to a projective
  multi-functor $\KhSpace\co \enl{\TangMulticat}\to \SBimCat$.
\end{theorem}
\begin{proof}
  The proof is the same as the proof of
  Theorem~\ref{thm:Kh-functorial}, using Lemma~\ref{lem:spec-bridges}
  in place of Lemma~\ref{lem:bridges} and
  Proposition~\ref{prop:spec-dual} in place of
  Proposition~\ref{prop:dual}.
\end{proof}

\begin{proof}[Proof of Theorem~\ref{thm:main}]
  Given oriented link diagrams $L_0$, $L_1$ with $P_0$ and $P_1$
  positive crossings and an oriented cobordism $\Sigma$ from $L_0$ to
  $L_1$, we have $P(\Sigma)=P_1-P_0$, so $\Sigma$ goes from
  $(L_0,P_0)$ to $(L_1,P_1)$. Hence,
  Theorem~\ref{thm:Kh-spec-functorial} gives a well-defined homotopy
  class of maps
  $\KhSpace(\Sigma)\co
  \KhSpace^j(L_0)\to\KhSpace^{j-\chi(\Sigma)}(L_1)$. It is immediate
  from that theorem that $\KhSpace$ is functorial in $\Sigma$.

  It remains to verify that the maps associated to Reidemeister moves
  and elementary cobordisms agree with the maps defined in our
  previous papers. This is equivalent to showing that the map
  associated with Reidemeister moves and elementary cobordisms in our
  previous papers~\cite{LLS-kh-tangles,LS-rasmus} commute with
  the gluing map for gluing tangles, up to homotopy. This is
  straightforward from the definitions, and is left to the reader.
\end{proof}

\section{Dotted cobordisms, neck cutting, and ribbon concordance}\label{sec:neck}
\emph{Wherein we} lift a result of Levine-Zemke on \textsc{ribbon concordances}
and Khovanov homology to the stable homotopy type.

A \emph{ribbon concordance} from $K_1$ to $K_2$ is a smoothly embedded
cylinder $\Sigma\subset [0,1]\times S^3$ so that projection to
$[0,1]$ gives a Morse function on $\Sigma$ with only index $0$ and $1$
critical points. (We are following Zemke's
convention~\cite{Zem-hf-ribbon}, not Gordon's convention~\cite{Gordon-top-ribbon},
because our results parallel Zemke's.) Following a pioneering result of Zemke's for the Heegaard
Floer knot invariants~\cite{Zem-hf-ribbon}, Levine-Zemke showed that
the map $\Kh(\Sigma)\co\Kh(K_1)\to\Kh(K_2)$ associated to a ribbon
concordance $\Sigma$ is a split injection~\cite{LV-kh-ribbon}.
The goal of this section is to prove the analogue for our map of
Khovanov spectra:
\begin{theorem}\label{thm:ribbon}
  If $\Sigma$ is a ribbon concordance from $K_1$ to $K_2$ then the map
  $\KhSpace(\Sigma)\co\KhSpace^j(K_1)\to\KhSpace^j(K_2)$ has a left
  homotopy inverse. That is, there is a map $G\co\KhSpace^j(K_2)\to\KhSpace^j(K_1)$ so
  that $G\circ \KhSpace(\Sigma)$ is homotopic to the identity map of $\KhSpace^j(K_1)$.
\end{theorem}
In other words, $\KhSpace^j(K_1)$ is a retract of $\KhSpace^j(K_2)$.

The main tool in the proof of Theorem~\ref{thm:ribbon} is a neck
cutting relation for $\KhSpace(\Sigma)$, analogous to
the neck cutting relation for Khovanov homology~\cite[Proposition
7]{LV-kh-ribbon}. To state it, we need to define maps associated to dotted
cobordisms. Fix a tangle $T$ and let $p$ be a point on $T$ (not one of
the endpoints). The \emph{elementary dot cobordism} associated to
$(T,p)$ is $[0,1]\times T$ with $(1/2,p)$ marked. Let $\Sigma_{T,p}$
denote the elementary dot cobordism associated to $(T,p)$. For any
integer $P$, Define a
map
\[
  \KhSpace(\Sigma_{T,p})\co \KhSpace(T,P)\to \KhSpace(T,P)
\]
as follows. Let $U$ be a small unknot disjoint from $T$ and adjacent to $p$. There is a
canonical identification $\KhSpace(T\amalg U)\simeq \KhSpace(T,P)\wedge
\KhSpace(U)=\KhSpace(T,P)\wedge( \SphereS_1\vee\SphereS_X)$. Merging $U$
and $T$ at $p$ is a cobordism $T\amalg U\to T$. The map
$\KhSpace(\Sigma_{T,p})$ is the composition
\[
  \KhSpace(T,P)=\KhSpace(T,P)\wedge\SphereS_X\into\KhSpace(T,P)\wedge(\SphereS_1\vee\SphereS_X)=\KhSpace(T\amalg U,P)\to \KhSpace(T,P)
\]
where the first map is the inclusion as the summand where $U$ is labeled $X$
and the second is induced by the merge. This map increases the quantum
grading by $2$.

The key properties of these maps are:
\begin{proposition}\label{prop:dot-rels}
  The map $\KhSpace(\Sigma_{T,p})$ is independent of which side of $T$
  the unknot $U$ lies on and, in general, the specific choice of
  unknot $U$ disjoint from $T$. Further, if $q$ is obtained by moving
  $p$ through a crossing then
  $\KhSpace(\Sigma_{T,p})\sim\pm\KhSpace(\Sigma_{T,q})$. If $\Sigma'$
  is an elementary cobordism from $T$ to $T'$ and $p$ is not in the
  support of $\Sigma$ then
  $\KhSpace(\Sigma_{T',p})\circ\KhSpace(\Sigma')\sim\pm\KhSpace(\Sigma')\circ\KhSpace(\Sigma_{T,p})$.
\end{proposition}
\begin{proof}
  The first statement is immediate from the definitions. For the
  second, we may assume that $T$ in fact consists of a single
  crossing, viewed as a $4$-ended diskular tangle. By
  Proposition~\ref{prop:spec-dual}, both of the maps $\Sigma_{T,p}$
  and $\Sigma_{T,q}$ are elements of
  \begin{equation}\label{eq:dot-rels-1}
    \pi_0\RHom_{\KhSpace(4)}(\KhSpace(T,P),\KhSpace(T,P)\{0,-2\})=\pi_0\KhSpace^{2-2}(U_2)=\pi_0(\SphereS\vee\SphereS).
  \end{equation}
  Here, $U_2$ denotes the 2-component unlink. In particular, this
  element is determined by its image under the Hurewicz homomorphism,
  i.e., the induced map on Khovanov's
  tangle invariant. So, it suffices to verify the result for
  Khovanov's combinatorial tangle invariant, which is a
  straightforward computation.
  Finally, the statement that $\KhSpace(\Sigma_{T,p})$ commutes with
  the maps associated to cobordisms with supports not containing $p$
  follows from the definition of $\KhSpace(\Sigma_{T,p})$ in terms of
  a saddle cobordism from $U\amalg T$ and the fact that cobordism maps
  for cobordisms with disjoint supports commute (part of
  Proposition~\ref{prop:Kh-space-movie-multifunc}).
\end{proof}

We can decompose an arbitrary tangle cobordism with dots on it as a
composition of elementary cobordisms from Definition~\ref{def:el-cob}
and elementary dot cobordisms. We will refer to either an elementary
cobordism from Definition~\ref{def:el-cob} or an elementary dot
cobordism as an \emph{elementary dotted cobordism}.  Given such a
decomposition of a tangle cobordism $\Sigma$ with dots, there is an
induced map $\KhSpace(\Sigma)$, by composing the maps from
Section~\ref{sec:planar-spectral} and the maps $\Sigma_{T,p}$. The
resulting maps are, in fact, independent of the decomposition:
\begin{theorem}\label{thm:dot-func}
  Let $T_1$ and $T_2$ be tangles and $P_1,P_2$ integers.  If two
  sequences of dotted elementary cobordisms from $(T_1,P_1)$ to
  $(T_2,P_2)$ correspond to isotopic dotted cobordisms then the maps
  they induce on Khovanov spectra agree up to sign.
\end{theorem}
\begin{proof}
  Given isotopic dotted cobordisms $\Sigma$ and $\Sigma'$, there is an
  isotopy from $\Sigma$ to $\Sigma'$ built from the following components:
  \begin{itemize}
  \item a movie move,
  \item moving a dot past a crossing, and
  \item exchanging the order of an elementary dot cobordism and
    another elementary cobordism with disjoint supports.
  \end{itemize}
  So, the result follows from Theorem~\ref{thm:Kh-spec-functorial} and
  Proposition~\ref{prop:dot-rels}.
\end{proof}

We can now state the neck-cutting relation:
\begin{proposition}\label{prop:neck-cut}
  Let $\Sigma\co T_1\to T_2$ be a tangle
  cobordism, possibly with dots. Let $N$, the \emph{neck}, be a smoothly embedded copy of
  $D^2\times [-1,1]^2$ with 
  $N\cap \Sigma=S^1\times [-1,1]\times\{0\}\subset N$. Let $\Sigma'_+$
  (respectively $\Sigma'_-$) be the result of deleting $N\cap \Sigma$
  from $\Sigma$ and replacing it with
  $D^2\times\{-1,1\}\times\{0\}\subset N$, and putting a new dot on
  $D^2\times\{ 1\}\times\{0\}$ (respectively
  $D^2\times\{ -1\}\times\{0\}$), and smoothing the corners. Then,
  \[
    \KhSpace(\Sigma)=\pm\KhSpace(\Sigma'_+)\pm\KhSpace(\Sigma'_-).
  \]
\end{proposition}
\begin{proof}
  We will use the fact that the map is local and the
  Hurewicz theorem to deduce this result from the corresponding fact
  for Khovanov homology.
  
  Let $\Sigma_{\mathit{neck}}$ be the cobordism
  \[
    \mathcenter{\begin{tikzpicture}
      \filldraw (1,0) circle (.1);
      \draw[dashed] (0,0) circle (1);
      \draw[->,bend right=30] (.7071,.7071) to (.7071,-.7071);
      \draw[<-,bend left=30] (-.7071,.7071) to (-.7071,-.7071);
    \end{tikzpicture}}
  \longrightarrow
  \mathcenter{
    \begin{tikzpicture}
      \filldraw (1,0) circle (.1);
      \draw[dashed] (0,0) circle (1);
      \draw[->,bend left=30] (.7071,.7071) to (-.7071,.7071);
      \draw[<-,bend right=30] (.7071,-.7071) to (-.7071,-.7071);
    \end{tikzpicture}
    }
    \longrightarrow
    \mathcenter{
    \begin{tikzpicture}
      \filldraw (1,0) circle (.1);
      \draw[dashed] (0,0) circle (1);
      \draw[->,bend right=30] (.7071,.7071) to (.7071,-.7071);
      \draw[<-,bend left=30] (-.7071,.7071) to (-.7071,-.7071);
    \end{tikzpicture}
    }
  \]
  from a flat $4$-ended diskular tangle $T_4$ to itself, consisting of two saddles.
  Using Theorem~\ref{thm:dot-func}, we can arrange that the cobordism
  $\Sigma$ is a (vertical) composition $\Sigma_2\circ\Sigma_1\circ\Sigma_0$ where
  $\Sigma_1$ is obtained by gluing (i.e., horizontally composing) $\Sigma_{\mathit{neck}}$ to an identity
  cobordism, and $\Sigma'=\Sigma_2\circ\Sigma_0$. So, it suffices to
  prove the result when $\Sigma$ is just the cobordism $\Sigma_{\mathit{neck}}$,
  viewed as a cobordism from a flat 4-ended tangle to itself.

  The homotopy class of the map associated to $\Sigma_{\mathit{neck}}$ is an element
  of
  \[
    \pi_0\RHom_{\KhSpace(4)}(\KhSpace(T_4,0),\KhSpace(T_4,0)\{0,-2\})=\pi_0\KhSpace^{2-2}(U_2)=\pi_0(\SphereS\vee\SphereS)
  \]
  (cf.\ Equation~\eqref{eq:dot-rels-1}).  The two copies of the sphere
  spectrum correspond to labeling the circles in $U_2$ $1$ and $X$ or
  vice-versa, and the first equality is
  Proposition~\ref{prop:spec-dual}.  So, like in
  Proposition~\ref{prop:dot-rels}, this element is determined by the
  induced map of Khovanov's combinatorial tangle invariant. Similarly,
  the maps induced by $\Sigma'_+$ and $\Sigma'_-$ are elements of
  $\pi_0(\SphereS\vee\SphereS)$ and so are also determined by the
  induced maps of Khovanov's tangle invariants. At the level of
  combinatorial Khovanov homology, the specified equality is a simple,
  direct computation (and also follows from the corresponding theorem
  for Bar-Natan's picture world~\cite{Bar-kh-tangle-cob}).
\end{proof}

As was the case for Khovanov homology, the maps of Khovanov spectra
associated to unknotted spheres with dots are simple:
\begin{lemma}\label{lem:dot-sphere}
  Let $\Sigma\co K_1\to K_2$ be a cobordism and
  $S\subset [0,1]\times\RR^3$ a $2$-sphere disjoint and
  geometrically unlinked from $\Sigma$ (but possibly knotted). Let $S_k$ be the result of putting $k$
  dots on $S$. Then, the map of Khovanov spectra associated to $\Sigma\amalg S_k$ is:
  \begin{itemize}
  \item Nullhomotopic if $k\neq 1$, and
  \item Homotopic to plus or minus the map induced by $\Sigma$ if $k=1$.
  \end{itemize}
\end{lemma}
\begin{proof}
  Again, we deduce this from the case of Khovanov homology.  By
  construction, the map associated to a disjoint union of cobordisms
  is the smash product of the maps associated to the individual
  cobordisms. (This uses Theorem~\ref{thm:dot-func} to show that we
  can assume all the intermediate diagrams are disjoint unions.) So,
  it suffices to prove the result when $K_1=K_2=\Sigma=\emptyset$. It
  follows from the behavior of the quantum grading that if $k\neq 1$
  then $\KhSpace(S_k)$ is nullhomotopic. For the remaining case,
  $\KhSpace(S_1)$ is a map from
  $\SphereS=\KhSpace(\emptyset)\to\KhSpace(\emptyset)=\SphereS$, i.e.,
  an element of $\pi_0(\SphereS)=H_0(\SphereS)=\ZZ$.

  So, it suffices to know that $S_1$ induces $\pm\Id$ on Khovanov
  homology. This is well known (see, for instance,~\cite[Corollary
  1.3]{LG-kh-split}, and see also~\cite[Proof of Theorem
  1.1]{Tanaka-kh-closed}), but for completeness and since the argument
  is short, we give a proof, following Sundberg-Swann~\cite[Theorem
  4.2]{SS-kh-surf}. Let $D$ be the result of deleting a small disk
  from $S$, so $D$ is a cobordism from the empty set to the unknot
  $U$. Let $T$ be the result of composing $D$ and a punctured
  torus. Then,
  $2\Kh(S)=\pm\Kh(T)\co \ZZ=\Kh(\emptyset)\to\Kh(\emptyset)=\ZZ$, but
  Rasmussen and Tanaka showed that $\Kh(T)$ is multiplication by
  $\pm 2$~\cite{Rasmussen-kh-closed,Tanaka-kh-closed}.
\end{proof}

\begin{proof}[Proof of Theorem~\ref{thm:ribbon}]
  This is the same as the proof in the combinatorial case~\cite[Theorem
  1]{LV-kh-ribbon}, using Lemma~\ref{lem:dot-sphere} and
  Proposition~\ref{prop:neck-cut} in place of their combinatorial
  analogues~\cite[Proposition 7]{LV-kh-ribbon}.
\end{proof}

As another corollary of the neck cutting relation, the maps of
Khovanov spectra, like the maps of Khovanov homology, do not see local
knotting:
\begin{corollary}
  If a cobordism $\Sigma'$ is obtained from a cobordism $\Sigma$ by
  taking a (standard) connected sum with a knotted $2$-sphere then
  $\KhSpace^j(\Sigma)=\KhSpace^j(\Sigma')$.
\end{corollary}

\begin{remark}
  As already observed~\cite[Lemma 4]{Bod-kh-sq}, weak functoriality of
  the Khovanov stable homotopy type already gives an obstruction to
  the existence of ribbon concordances: weak functoriality implies
  that if there is a ribbon concordance from $K_1$ to $K_2$ then there
  is a split injection from $\Kh(K_1)$ to $\Kh(K_2)$ respecting the
  action of the Steenrod squares. (So, for example, none of the pairs
  of knots with isomorphic Khovanov homology found by
  Seed~\cite{Seed-Kh-square} are related by ribbon concordances.)

  Theorem~\ref{thm:ribbon} implies that, in fact, for any generalized
  homology theory $h$, there is a split injection from $h(K_1)$ to
  $h(K_2)$.  As an example of the difference, recall that
  $\pi_5^s(\CC P^2)=\ZZ/12\ZZ$. Let $X$ be the result of attaching a
  $6$-cell to $\CC P^2$ via the element $6\in\ZZ/12\ZZ$, and let
  $Y=\CC P^2\vee S^6$. Then $X$ and $Y$ have isomorphic homology
  groups and the same action by the mod-$p$ Steenrod algebra for all
  primes $p$, but $\pi_5^s(X)=\ZZ/6\ZZ$ while
  $\pi_5^s(Y)=\ZZ/12\ZZ$. In particular, if for some knots $K_1$ and
  $K_2$ and quantum grading $j$, $\KhSpace^j(K_1)=X$ and
  $\KhSpace^j(K_2)=Y$ then there is no ribbon concordance relating
  $K_1$ and $K_2$.
\end{remark}

\begin{remark}
  Given two surfaces $\Sigma,\Sigma'\subset [0,1]\times\RR^3$, a more
  general connected sum operation is to delete a disk from each of
  $\Sigma$ and $\Sigma'$ and then attach an annulus in
  $[0,1]\times\RR^3\setminus (\Sigma\cup\Sigma')$ to the new boundary
  components. It follows from neck cutting and the main theorem of
  Gujral-Levine~\cite{LG-kh-split} that taking this kind of
  generalized connected sum with a sphere (even one knotted and linked
  with $\Sigma$) does not change the induced map on Khovanov
  homology. We do not know if Gujral-Levine's result holds for the
  maps of Khovanov spectra so, in particular, we do not know if this
  generalized connected sum with a sphere sometimes changes the map of
  Khovanov spectra.
\end{remark}

\section{Computations}\label{sec:computations}
\emph{Wherein we} describe an example of a \textsc{Hopf-like invariant
  of link cobordisms} coming from naturality of the Khovanov spectrum.

Maps of spaces are much richer than maps of abelian groups. In
particular, there can be non-nullhomotopic maps of spaces when the
induced maps on homology vanish for grading reasons: the familiar Hopf
map in $\pi_1^s(S^0)=\ZZ/2$ is an example. Another example is the
Hopf-like map in $\pi_1(M(\ZZ/2))=[S^{n+2},\Sigma^n\RR P^2]=\ZZ/2$
($n\geq 2$). For the Khovanov spectrum,
this phenomenon can even occur for maps between Khovanov-thin knots,
even though the Khovanov spectra for Khovanov-thin knots are wedge sums of Moore
spectra~\cite[Section 9.3]{RS-khovanov} and, consequently, determined
by their homology. One way to detect interesting maps is to study
their mapping cones. As an example, we have the following proposition.

% If this document doesn't compile for you, comment out this picture.
% It uses relatively new tikz features.
% Also remove the "braids" library from the tikz package include statement.
\begin{figure}
  \centering
  \begin{center}
    \scalebox{.75}{
    \begin{tikzpicture}
      \pic[
      braid/.cd,
      number of strands=8,
      line width = 2pt,
      rotate=90,
      gap=0.1,
      name prefix=braid]{braid = {s_1-s_3^{-1}-s_6^{-1} s_1-s_4 s_2-s_5^{-1} s_2}};
      \draw[bend left=90, line width=2pt] (braid-1-s) to (braid-6-s);
      \draw[bend left=90, line width=2pt] (braid-2-s) to (braid-3-s);
      \draw[bend left=90, line width=2pt] (braid-4-s) to (braid-5-s);
      \draw[bend left=90, line width=2pt] (braid-7-s) to (braid-8-s);
      \draw[bend left=90, line width=2pt] (braid-8-e) to (braid-1-e);
      \draw[bend left=90, line width=2pt] (braid-6-e) to (braid-2-e);
      \draw[bend left=90, line width=2pt] (braid-3-e) to (braid-4-e);
      \draw[bend left=90, line width=2pt] (braid-7-e) to (braid-5-e);
      \draw[line width=2pt, red, dotted] ($(braid-7-1)!0.5!(braid-7-0)$) circle (15pt); 
    \end{tikzpicture}}
  \end{center}
  \caption{\textbf{The knot $8_{19}$.}}
  \label{fig:819}
\end{figure}
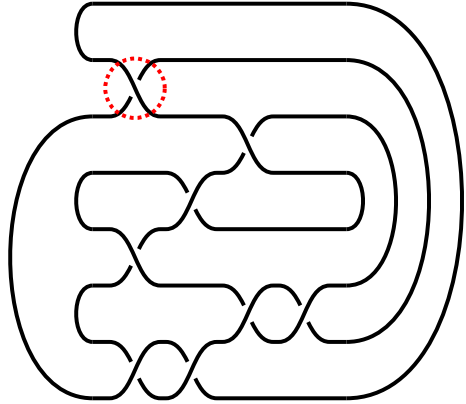

\begin{proposition}
  There is an orientable cobordism $\Sigma$ from the knot $K_0=5_2$ to the link
  $K_1=5_1\cup\text{meridian}$ so that the induced map of Khovanov
  spectra
  \begin{equation}\label{eq:eg}
   S^0\vee S^1\simeq \KhSpace^{3}(K_0)\to
    \KhSpace^{4}(K_1)\simeq S^0\vee \Sigma^{-1}\RP^2
  \end{equation}
  sends $S^1$ to $\Sigma^{-1}\RP^2$ via the Hopf map. More precisely,
  we can choose the homotopy equivalences in Formula~\eqref{eq:eg} so
  that the maps $S^0\to \Sigma^{-1}\RR P^2$ and $S^1\to S^0$ are
  nullhomotopic, in which case the map $S^1\to \Sigma^{-1}\RR P^2$ is
  the Hopf map.
\end{proposition}
\begin{proof}
  Let $K=8_{19}=m(T(3,4))$, which is shown in
  Figure~\ref{fig:819}. The $0$-resolution (respectively $1$-resolution) of the
  \textcolor{red}{circled} crossing is $K_0$ (respectively
  $K_1$). Hence, this crossing  corresponds to a single saddle
  cobordism from $K_0$ to $K_1$. Since $K_1$ has two components, this
  cobordism is orientable.

  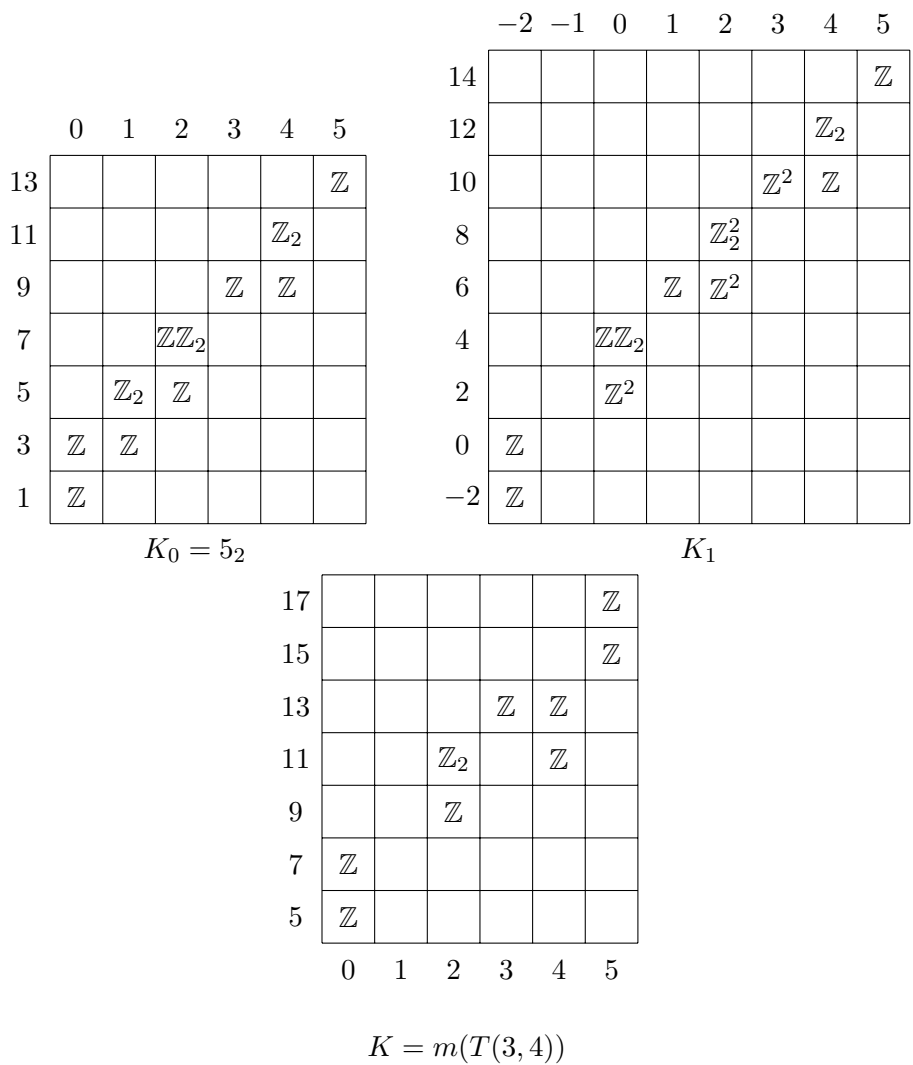
\begin{figure}
    \centering
      \begin{tikzpicture}[scale=0.7,yscale=0.5]
    \draw[ystep=2,xshift=-0.5cm] (6,14) grid (0,0);
    \foreach \h in {0,...,5}{
      \node at (\h,15) {$\h$};}
    \foreach \q in {1,3,...,13}{
      \node at (-1,\q) {$\q$};}

    \foreach \i/\j/\v in {1/0/\ZZ,3/0/\ZZ,3/1/\ZZ,5/1/\ZZ_2,5/2/\ZZ,7/2/\ZZ\ZZ_2,9/3/\ZZ,9/4/\ZZ,11/4/\ZZ_2,13/5/\ZZ}{
      \node at (\j,\i) {$\v$};}
     \node at (2.25,-1) (label) {$K_0=5_2$};
  \end{tikzpicture}
    \qquad
  \begin{tikzpicture}[scale=0.7,yscale=0.5]
    \draw[ystep=2,yshift=-1cm,xshift=-0.5cm] (6,16) grid (-2,-2);
    \foreach \h in {-2,...,5}{
      \node at (\h,16) {$\h$};}
    \foreach \q in {-2,0,...,14}{
      \node at (-3,\q) {$\q$};}

    \foreach \i/\j/\v in {-2/-2/\ZZ,0/-2/\ZZ,2/0/\ZZ^2,4/0/\ZZ\ZZ_2,6/1/\ZZ,6/2/\ZZ^2,8/2/\ZZ_2^2,10/3/\ZZ^2,10/4/\ZZ,12/4/\ZZ_2,14/5/\ZZ}{
      \node at (\j,\i) {$\v$};}
    \node at (1.5,-4) (label) {$K_1$};
  \end{tikzpicture}
    \qquad
    \begin{tikzpicture}[scale=0.7,yscale=0.5]
    \draw[ystep=2,xshift=-0.5cm] (6,18) grid (0,4);
    \foreach \h in {0,...,5}{
      \node at (\h,3) {$\h$};}
    \foreach \q in {5,7,...,17}{
      \node at (-1,\q) {$\q$};}

    \foreach \i/\j/\v in {5/0/\ZZ,7/0/\ZZ,9/2/\ZZ,11/2/\ZZ_2,11/4/\ZZ,13/3/\ZZ,13/4/\ZZ,15/5/\ZZ,17/5/\ZZ}{
      \node at (\j,\i) {$\v$};}
    \node at (2.25,0) (label) {$K=m(T(3,4))$};
  \end{tikzpicture}
    \caption{\textbf{Khovanov homology of $K_0=5_2$, $K_1$, and
        $K=m(T(3,4))$.} The symbol $\oplus$ has been suppressed in 
     $\ZZ\oplus\ZZ/2$ in two places. The homological grading is
     horizontal and the quantum grading is vertical.}
    \label{fig:Kh-example}
  \end{figure}
    
  The Khovanov homologies of $K$, $K_0$, and $K_1$ are shown in
  Figure~\ref{fig:Kh-example}. These computations were extracted from
  the Knot Atlas and Mathematica KnotTheory
  packages~\cite{KAT-kh-knotatlas}. Knot Atlas is not consistent about
  the distinction between a knot and its mirror, but since $K$ is a
  negative knot, with our conventions its Khovanov homology is
  supported in positive gradings (see
  Remark~\ref{rem:grading-help}). For the 2-component link $K_1$, the
  KnotTheory package gives idiosyncratic gradings; we have shifted the
  results to agree with our conventions.

  We have
  \begin{align*}
    \KhCx(K)\grs{-7}{-2}&\simeq \Cone(\KhCx(\Sigma)\co \KhCx(K_0)\grs{1}{0}\to\KhCx(K_1)),\\
    \Sigma^{-2}\KhSpace^{j}(K)&\simeq \Cone(\KhSpace(\Sigma)\co \KhSpace^{j-8}(K_0)\to\KhSpace^{j-7}(K_1)).
  \end{align*}
  One can verify the grading shift either from the diagram and grading formulas
  or by examining the Khovanov homologies: this is the only possibility
  consistent with a long exact sequence
  $\cdots\to\Kh(K_1)\to\Kh(K)\grs{b}{a}\to\Kh(K_0)\to\cdots$.

  Consider $\KhSpace^{11}(K)$. It was calculated previously~\cite{RS-steenrod,JLS-kh-Morse-moves} that
  \[
    \KhSpace^{11}(K)\simeq \Sigma^{-1}\RR P^5/\RR P^2.
  \]
  (Note that our conventions are different from~\cite{RS-steenrod}.) On the
  other hand, since $K_0$ and $K_1$ are thin we have
  \begin{align*}
    \KhSpace^{3}(K_0)&\simeq S^{0}\vee S^1\\
    \KhSpace^{4}(K_1)&\simeq S^0\vee \Sigma^{-1}\RR P^2.
  \end{align*}
  Write the map $S^{0}\vee S^1\to S^0\vee \Sigma^{-1}\RR P^2$ as
  \[
    (a,b,c,d)\in \Bigl(\pi_0^s(S^0)\oplus \pi_0^s(\Sigma^{-1}\RR
    P^2)\oplus\pi_1^s(S^0)\oplus\pi_1^s(\Sigma^{-1}\RR P^2)\Bigr)\cong  \ZZ\oplus(\ZZ/2\ZZ)^3.
  \]
  By considering the homology of
  $\KhSpace^{11}(K)$, $a$ must be a unit. So, we can pre-compose with
  an automorphism of $S^{0}\vee S^1$ and post-compose with an
  automorphism of $S^0\vee \Sigma^{-1}\RR P^2$ so that $b=c=0$. Then,
  considering the Steenrod squares on $\RR P^5/\RR P^2$, the map 
  $\KhSpace^3(K_0)\supset S^1\to \Sigma^{-1}\RR P^2\subset
  \KhSpace^4(K_1)$ must be the Hopf map, as claimed.
\end{proof}

\begin{remark}
  The Khovanov stable homotopy type does not give an interesting
  invariant of closed surfaces in an obvious way. Given a closed
  surface $\Sigma$, viewed as a map from the empty link to the empty
  link, there is an induced map
  \[
    \KhSpace(\Sigma)\co \KhSpace^j(\emptyset)\to \KhSpace^{j-\chi(\Sigma)}(\emptyset).
  \]
  Since $\KhSpace^j(\emptyset)$ is the sphere spectrum $\SphereS$ if
  $j=0$ and trivial for $j\neq 0$, the map $\KhSpace(\Sigma)$ can only
  be nontrivial if $\chi(\Sigma)=0$. In this case, by the Hurewicz
  theorem, the homotopy class of the map $\KhSpace^j(\Sigma)$ is
  determined by the induced map on homology. This map $\ZZ\to\ZZ$
  sends $1$ to $2^{b_0}$ if $\Sigma$ consists of
  $b_0$ tori, and $0$ if $\Sigma$ has any non-toroidal
  components~\cite{Rasmussen-kh-closed,Tanaka-kh-closed,LG-kh-split}.
\end{remark}

\section*{Table of notation}
  
% \begin{table}
%   \centering
  \begin{longtable}[tab:notation]{cp{4.25in}}
    \toprule
    Notation & Meaning\\
    \midrule
    \endhead
    $a,b,\dots$ & Crossingless matchings (\pageref{not:crossingless})\\
    $\Crossingless{n}$ & Set of crossingless matchings of $n$ points
                         ($n$ even) (\pageref{not:crossingless})\\
    $\Wmirror{a}$, $\Wmirror{T}$ & The mirror of a tangle or crossingless matching (\pageref{not:crossingless})\\
    $N$  & The number of crossings of a link $L$ or tangle $T$\\
    $P$ & Auxiliary integer. Morally, number of positive crossings (\pageref{sec:gradings})\\
    $\Crossings$ & The set of crossings of a tangle $T$ (\pageref{not:crossings})\\
    $\annulus$ & A specific annulus (\pageref{sec:tang-multicat})\\
    $T^{m_1,\dots,m_k;n}$ or $T$      & A diskular tangle (\pageref{def:diskular})\\
    $T\circ_i S$, $T\circ(S_1,\dots,S_k)$  & Composition of diskular tangles (\pageref{not:TangComp})\\
    $T_v$ & Resolution of tangle $T$ associated to vertex $v$ of $\CCat{\Crossings}$ (\pageref{eq:KhCx-T-hocolim})\\
    $\Sigma$ & Cobordism of diskular tangles (\pageref{def:cobordism})\\
    $P(\Sigma)$ & Effect of $\Sigma$ on number of positive crossings (\pageref{def:cobordism})\\
    $\chi'(\Sigma)$ & Modified Euler characteristic of $\Sigma$ (\pageref{def:cobordism})\\
           $\TangMovMulticat$  & The tangle cobordism movie multicategory (\pageref{def:tangle-movie-multicat})\\
    $\TangMulticat$    & The tangle cobordism multicategory (\pageref{def:tangle-multicat})\\
    $\enl{\TangMovMulticat}$, $\enl{\TangMulticat}$ & Canonical groupoid enrichments of ${\TangMovMulticat}$, ${\TangMulticat}$ (\pageref{def:enlargement})\\
    $\AbelianGroups$ & (Multi-)Category of abelian groups (\pageref{not:crossings})\\
    $\Spectra$ & (Multi-)Category of symmetric spectra (\pageref{not:KhSpaceAlg})\\
    $\BimCat$ & Multicategory of \dg multimodules (\pageref{def:BimCat})\\
    $\SBimCat$ & Multicategory of spectral multimodules (\pageref{def:SBimCat})\\
    $\SphereS$ & The sphere spectrum (\pageref{not:sphere})\\
    \midrule
    \pagebreak[2]
    $\KhCx(L)$  & The Khovanov complex of a link $L$ (\pageref{not:KhCx})\\
    $\Kh(L)$  & Khovanov homology of a link $L$ (\pageref{not:KhCx})\\
    $V$ & The Khovanov Frobenius algebra or TQFT (\pageref{not:V})\\
    $V(Z)$ & The Khovanov TQFT applied to a closed $1$-manifold $Z$ (\pageref{not:V})\\
    $\CCat{\Crossings}$ & Cube category on the set $\Crossings$ (\pageref{not:crossings})\\
    $\CCatP{\Crossings}$ & Result of doubling terminal object in $\CCat{\Crossings}$ (\pageref{not:crossings})\\
    $|v|$ & Height of a vertex $v$ of $\CCat{\Crossings}$ (\pageref{not:abs-v})\\
    $\KhCx(n)$  & Khovanov's arc algebra on $n$ points ($n$ even) (\pageref{not:crossingless})\\
    $\KhCx(T)$  & Khovanov's complex of bimodules associated to $(2m,2n)$-tangle $T$ (\pageref{not:KhCxT})\\
    $\KhCx(\Sigma)$ & Khovanov map associated to a tangle cobordism $\Sigma$ (\pageref{def:planar-Kh})\\
    $\grs{q}{h}$ & Homological grading shift by $h$, quantum grading shift by $q$ (\pageref{not:grs}) \\
    $\gr_q$, $\gr_h$ & Quantum and homological gradings (\pageref{not:grh})\\
    \midrule
    \pagebreak[2]
    $\KhSpace(K)$, $\KhSpace^j(K)$ & Khovanov spectrum of a link $K$, in quantum grading $j$ (\pageref{not:KhSpace},\pageref{not:KhSpaceT2})\\
    $\KhSpace(n)$  & Spectral arc algebra on $n$ points ($n$ even) (\pageref{not:KhSpaceAlg})\\
    $\KhSpace(T)$, $\KhSpace(T,P)$  & Spectral arc algebra bimodule associated to a $(m,n)$-tangle or diskular tangle $T$ (\pageref{not:KhSpaceT1}, \pageref{not:KhSpaceT2})\\
    $\KhSpace(\Sigma)$ & Map of Khovanov spectra associated to tangle cobordism $\Sigma$ (\pageref{not:KhSpSigma})\\
    \midrule
    \pagebreak[2]
$\mHshapeS{n}$ &  Arc algebra shape multicategory (\pageref{sec:spec-arc-alg})\\
$\SmTshape{m}{n}$, $\SmTshape{m_1,\dots,m_k}{n}$ &  Tangle shape multicategory (\pageref{not:KhSpaceAlg}, \pageref{not:big-tang-shape})\\
$\mGlueS{m}{n}{p}$ & Gluing shape multicategory (\pageref{not:mGlueS})\\
    $\mHshape{n}$, $\mTshape{m}{n}$, $\mGlue{m}{n}{p}$ & Groupoid enriched versions of $\mHshapeS{n}$, $\SmTshape{m}{n}$, $\mGlueS{m}{n}{p}$ (\pageref{not:KhSpaceAlg}, \pageref{not:mGlueS}, \pageref{not:mGlue-2})\\
  $\CobD$ & Divided cobordism category (\pageref{def:CobD}, \pageref{not:enlCobD}, \pageref{def:CobD-annulus})\\
    $\enl{\Cat}$ & Canonical groupoid enrichment of $\Cat$ (\pageref{not:enlCobD}, \pageref{def:enlargement})\\
    $\tId$ & Particular morphism related to canonical groupoid enrichment (\pageref{lem:tId})\\
    $\CCat{\Crossings}\ttimes \mTshape{m}{n}$ & Thickened product of $\CCat{\Crossings}$ and $\mTshape{m}{n}$ (\pageref{not:ttimes1})\\
    \bottomrule
  % \end{tabular}
  \caption[Table of notation]{\textbf{Table of notation.} The page
    where each notation is introduced is noted in parentheses.}
  \label{tab:notation}
\end{longtable}

\vspace{-0.3cm}
\bibliographystyle{myalpha}
\bibliography{newbibfile}

\providecommand{\bysame}{\leavevmode\hbox to3em{\hrulefill}\thinspace}
\providecommand{\MR}{\relax\ifhmode\unskip\space\fi MR }
\providecommand{\MRhref}[2]{%
  \href{http://www.ams.org/mathscinet-getitem?mr=#1}{#2}
}
\providecommand{\href}[2]{#2}
\begin{thebibliography}{EGNO15}

\bibitem[Bar05]{Bar-kh-tangle-cob}
Dror Bar-Natan, \emph{Khovanov's homology for tangles and cobordisms}, Geom.
  Topol. \textbf{9} (2005), 1443--1499.

\bibitem[Bla10]{Bla-kh-oriented}
Christian Blanchet, \emph{An oriented model for {K}hovanov homology}, J. Knot
  Theory Ramifications \textbf{19} (2010), no.~2, 291--312.

\bibitem[BM]{KAT-kh-knotatlas}
Dror Bar-Natan, Scott Morrison, and et~al., \emph{The {K}not {A}tlas},
  \url{http://katlas.org/}.

\bibitem[BM12]{BM-top-spectral}
Andrew~J. Blumberg and Michael~A. Mandell, \emph{Localization theorems in
  topological {H}ochschild homology and topological cyclic homology}, Geom.
  Topol. \textbf{16} (2012), no.~2, 1053--1120.

\bibitem[Bod]{Bod-kh-sq}
Holt Bodish, \emph{Non-trivial {S}teenrod squares on prime, hyperbolic and
  satellite knots}, arXiv:2006.10881.

\bibitem[Cap08]{Cap-kh-functoriality}
Carmen~Livia Caprau, \emph{{$\rm sl(2)$} tangle homology with a parameter and
  singular cobordisms}, Algebr. Geom. Topol. \textbf{8} (2008), no.~2,
  729--756.

\bibitem[CMW09]{CMW-kh-functoriality}
David Clark, Scott Morrison, and Kevin Walker, \emph{Fixing the functoriality
  of {K}hovanov homology}, Geom. Topol. \textbf{13} (2009), no.~3, 1499--1582.

\bibitem[CS93]{CS-knot-movie}
J.~Scott Carter and Masahico Saito, \emph{Reidemeister moves for surface
  isotopies and their interpretation as moves to movies}, J. Knot Theory
  Ramifications \textbf{2} (1993), no.~3, 251--284.

\bibitem[EGNO15]{EGNO-other-tens-book}
Pavel Etingof, Shlomo Gelaki, Dmitri Nikshych, and Victor Ostrik, \emph{Tensor
  categories}, Mathematical Surveys and Monographs, vol. 205, American
  Mathematical Society, Providence, RI, 2015.

\bibitem[EM06]{EM-top-machine}
A.~D. Elmendorf and M.~A. Mandell, \emph{Rings, modules, and algebras in
  infinite loop space theory}, Adv. Math. \textbf{205} (2006), no.~1, 163--228.

\bibitem[GL]{LG-kh-split}
Onkar~Singh Gujral and Adam~Simon Levine, \emph{{K}hovanov homology and
  cobordisms between split links}, arXiv:2009.03406.

\bibitem[Gor81]{Gordon-top-ribbon}
C.~McA. Gordon, \emph{Ribbon concordance of knots in the {$3$}-sphere}, Math.
  Ann. \textbf{257} (1981), no.~2, 157--170.

\bibitem[HKK16]{HKK-Kh-htpy}
Po~Hu, Daniel Kriz, and Igor Kriz, \emph{Field theories, stable homotopy theory
  and {K}hovanov homology}, Topology Proc. \textbf{48} (2016), 327--360.

\bibitem[HLS16]{HLS-flexible}
Kristen Hendricks, Robert Lipshitz, and Sucharit Sarkar, \emph{A flexible
  construction of equivariant {F}loer homology and applications}, J. Topol.
  \textbf{9} (2016), no.~4, 1153--1236.

\bibitem[Jac04]{Jac-kh-cobordisms}
Magnus Jacobsson, \emph{An invariant of link cobordisms from {K}hovanov
  homology}, Algebr. Geom. Topol. \textbf{4} (2004), 1211--1251 (electronic).

\bibitem[JLS17]{JLS-kh-Morse-moves}
Dan Jones, Andrew Lobb, and Dirk Sch\"{u}tz, \emph{Morse moves in flow
  categories}, Indiana Univ. Math. J. \textbf{66} (2017), no.~5, 1603--1657.

\bibitem[Jon]{Jones-oth-planar}
Vaughan F.~R. Jones, \emph{Planar algebras, {I}}, arXiv:math/9909027.

\bibitem[Kho00]{Kho-kh-categorification}
Mikhail Khovanov, \emph{A categorification of the {J}ones polynomial}, Duke
  Math. J. \textbf{101} (2000), no.~3, 359--426.

\bibitem[Kho02]{Kho-kh-tangles}
\bysame, \emph{A functor-valued invariant of tangles}, Algebr. Geom. Topol.
  \textbf{2} (2002), 665--741 (electronic).

\bibitem[Kho06]{Kho-kh-cobordism}
\bysame, \emph{An invariant of tangle cobordisms}, Trans. Amer. Math. Soc.
  \textbf{358} (2006), no.~1, 315--327.

\bibitem[KW]{KW-kh-Blanchet}
Vyacheslav Krushkal and Paul Wedrich, \emph{{$\mathfrak{gl}(2)$} foams and the
  {K}hovanov homotopy type}, arXiv:2101.05785.

\bibitem[Lee05]{Lee-kh-endomorphism}
Eun~Soo Lee, \emph{An endomorphism of the {K}hovanov invariant}, Adv. Math.
  \textbf{197} (2005), no.~2, 554--586.

\bibitem[LLSa]{LLS-kh-CK}
Tyler Lawson, Robert Lipshitz, and Sucharit Sarkar, \emph{{C}hen-{K}hovanov
  spectra for tangles}, arXiv:1909.12994.

\bibitem[LLSb]{LLS-kh-tangles}
\bysame, \emph{Khovanov spectra for tangles}, arXiv:1706.02346.

\bibitem[LLS20]{LLS-khovanov-product}
Tyler Lawson, Robert Lipshitz, and Sucharit Sarkar, \emph{Khovanov homotopy
  type, {B}urnside category and products}, Geom. Topol. \textbf{24} (2020),
  no.~2, 623--745.

\bibitem[LS14a]{RS-khovanov}
Robert Lipshitz and Sucharit Sarkar, \emph{A {K}hovanov stable homotopy type},
  J. Amer. Math. Soc. \textbf{27} (2014), no.~4, 983--1042.

\bibitem[LS14b]{LS-rasmus}
\bysame, \emph{A refinement of {R}asmussen's {$S$}-invariant}, Duke Math. J.
  \textbf{163} (2014), no.~5, 923--952.

\bibitem[LS14c]{RS-steenrod}
\bysame, \emph{A {S}teenrod square on {K}hovanov homology}, J. Topol.
  \textbf{7} (2014), no.~3, 817--848.

\bibitem[LZ19]{LV-kh-ribbon}
Adam~Simon Levine and Ian Zemke, \emph{Khovanov homology and ribbon
  concordances}, Bull. Lond. Math. Soc. \textbf{51} (2019), no.~6, 1099--1103.

\bibitem[Ras]{Rasmussen-kh-closed}
Jacob Rasmussen, \emph{Khovanov's invariant for closed surfaces},
  arXiv:math/0502527.

\bibitem[Rob17]{Roberts-kh-planar}
Lawrence~P. Roberts, \emph{Planar algebras and the decategorification of
  bordered {K}hovanov homology}, J. Knot Theory Ramifications \textbf{26}
  (2017), no.~4, 1750023, 23.

\bibitem[San21]{Sano-Kh-func}
Taketo Sano, \emph{Fixing the functoriality of {K}hovanov homology: a simple
  approach}, J. Knot Theory Ramifications \textbf{30} (2021), no.~11, Paper No.
  2150074, 12.

\bibitem[See]{Seed-Kh-square}
Cotton Seed, \emph{Computations of the {L}ipshitz-{S}arkar {S}teenrod square on
  {K}hovanov homology}, arXiv:1210.1882.

\bibitem[SP14]{SP-other-dual}
Christopher~J. Schommer-Pries, \emph{Dualizability in low-dimensional higher
  category theory}, Topology and field theories, Contemp. Math., vol. 613,
  Amer. Math. Soc., Providence, RI, 2014, pp.~111--176.

\bibitem[SS]{SS-kh-surf}
Isaac Sundberg and Jonah Swann, \emph{Relative {K}hovanov-{J}acobsson classes},
  arXiv:2103.01438.

\bibitem[Tan06]{Tanaka-kh-closed}
Kokoro Tanaka, \emph{Khovanov-{J}acobsson numbers and invariants of
  surface-knots derived from {B}ar-{N}atan's theory}, Proc. Amer. Math. Soc.
  \textbf{134} (2006), no.~12, 3685--3689.

\bibitem[Zem19]{Zem-hf-ribbon}
Ian Zemke, \emph{Knot {F}loer homology obstructs ribbon concordance}, Ann. of
  Math. (2) \textbf{190} (2019), no.~3, 931--947.

\end{thebibliography}
\vspace{1cm}
%\newpage
\end{document}